\documentclass[a4paper]{amsart}

\pdfoutput=1 

\newcommand{\cC}{\mathcal{C}}

\newcommand{\cG}{\mathcal{G}}
\newcommand{\cI}{\mathcal{I}}
\newcommand{\cL}{\mathcal{L}}

\newcommand{\cV}{\mathcal{V}}

\newcommand{\bC}{\mathbb{C}}

\newcommand{\bL}{\mathbb{L}}
\newcommand{\bN}{\mathbb{N}}

\newcommand{\bQ}{\mathbb{Q}}\newcommand{\bR}{\mathbb{R}}

\newcommand{\bZ}{\mathbb{Z}}

\newcommand{\ft}{\mathfrak{t}}

\newcommand{\fh}{\mathfrak{h}}
\newcommand{\fu}{\mathfrak{u}}
\newcommand{\fg}{\mathfrak{g}}

\newcommand{\fp}{\mathfrak{p}}
\newcommand{\bunit}{\mathbbm{1}}

\newcommand{\CircNum}[1]{\ooalign{\hfil\raise .00ex\hbox{\scriptsize #1}\hfil\crcr\mathhexbox20D}}

\usepackage{lmodern}
\usepackage{microtype}
\usepackage{booktabs}
\usepackage[dvipsnames,svgnames,x11names,hyperref]{xcolor}
\usepackage{mathtools,url,verbatim,amssymb,stmaryrd,nicefrac,subcaption,booktabs,bbm}
\usepackage{enumerate}
\usepackage[pagebackref,colorlinks,citecolor=Mahogany,linkcolor=Mahogany,urlcolor=Mahogany,filecolor=Mahogany]{hyperref}
\usepackage[capitalise]{cleveref}
\usepackage[margin=1.45in]{geometry}

\AtBeginDocument{%
	\def\MR#1{}
}

\usepackage{tikz, tikz-cd}
\usetikzlibrary{matrix,calc,positioning,arrows,decorations.pathreplacing,decorations.markings,patterns}
\tikzset{->-/.style={decoration={
			markings,
			mark=at position #1 with {\arrow{>}}},postaction={decorate}}}
\tikzset{-<-/.style={decoration={
					markings,
					mark=at position #1 with {\arrow{<}}},postaction={decorate}}}
		

\newtheorem{theorem}{Theorem}[section]
\newtheorem{lemma}[theorem]{Lemma}
\newtheorem{proposition}[theorem]{Proposition}
\newtheorem{corollary}[theorem]{Corollary}
\newtheorem*{corollary*}{Corollary}

\newtheorem{atheorem}{Theorem}

\theoremstyle{definition}
\newtheorem{definition}[theorem]{Definition}
\newtheorem{notation}[theorem]{Notation}

\theoremstyle{remark}

\newtheorem{remark}[theorem]{Remark}
\newtheorem*{remark*}{Remark}

\newtheorem{question}[theorem]{Question}

\newcommand{\half}{\nicefrac{1}{2}}
\newcommand{\cat}[1]{\mathsf{#1}}
\newcommand{\mr}[1]{{\rm #1}}

\newcommand{\ul}[1]{\underline{#1}}
\newcommand{\ol}[1]{\overline{#1}}
\newcommand{\fS}{\mathfrak{S}}

\newcommand{\lra}{\longrightarrow}

\newcommand{\Diff}{\mr{Diff}}
\newcommand{\hAut}{\mr{hAut}}
\newcommand{\Emb}{\mr{Emb}}
\newcommand{\Tor}{\mr{Tor}}
\newcommand{\GrLCS}{{\mr{Gr}^\bullet_\mr{LCS}\,}}

\DeclareMathOperator*{\colim}{colim}

\definecolor{darkgreen}{RGB}{0,100,0}

\title{On the Torelli Lie algebra}

\author{Alexander Kupers}
\address{Department of Computer and Mathematical Sciences \\
 	University of Toronto Scarborough \\
 	 1265 Military Trail \\
 	 Toronto, ON M1C 1A4 \\ Canada}
\email{a.kupers@utoronto.ca}
 
\author{Oscar Randal-Williams}

\address{Centre for Mathematical Sciences\\
  Wilberforce Road\\
 Cambridge CB3 0WB\\
 UK}
\email{o.randal-williams@dpmms.cam.ac.uk}

\begin{document}

\begin{abstract}
We prove two theorems about the Malcev Lie algebra associated to the Torelli group of a surface of genus $g$: stably, it is Koszul and the kernel of the Johnson homomorphism consists only of trivial $\mr{Sp}_{2g}(\bZ)$-representations lying in the centre.
\end{abstract}

\maketitle

\tableofcontents

\newpage

\section{Introduction} Let $\Sigma_{g,1}$ denote a compact oriented surface of genus $g$ with 1 boundary component. Its \emph{Torelli group} is the group $T_{g,1}$ of isotopy classes of orientation-preserving diffeomorphisms of $\Sigma_{g,1}$ which act as the identity on $H_1(\Sigma_{g,1};\bZ)$. There is an initial pro-unipotent $\bQ$-algebraic group under $T_{g,1}$, whose pro-nilpotent Lie algebra is the \emph{unipotent completion}, or \emph{Malcev completion}, $\ft_{g,1}$ of $T_{g,1}$. Hain has proved that, as long as $g \geq 4$, $\ft_{g,1}$ is isomorphic to the completion of the associated graded $\mr{Gr}^\bullet_\mr{LCS}\, \ft_{g,1}$ of its lower central series, and that this Lie algebra is quadratically presented. In this paper we will prove that it is Koszul in a stable range. Throughout this paper we shall take ``Koszul" to mean the diagonal criterion: the additional ``weight" grading of $\GrLCS \ft_{g,1}$ induces a weight grading on its Lie algebra homology (see \cref{sec:preliminaries} for our grading conventions), and Koszulness means vanishing of this bigraded Lie algebra homology away from the diagonal.

\begin{atheorem}\label{athm:main}
The Lie algebra $\GrLCS \ft_{g,1}$ is Koszul in weight $ \leq \tfrac{g}{3}$.
\end{atheorem}

More generally, for $\Sigma_{g,n}^r$ a surface of genus $g$ with $n$ boundary components and $r$ marked points, let $\smash{\ft^r_{g,n}}$ be the unipotent completion of its Torelli group. We also prove that $\GrLCS\ft_g$ and $\GrLCS \ft_{g}^1$ are Koszul in the same range, as well as the \emph{relative unipotent completions} $\GrLCS \fu_{g}$, $\GrLCS \fu_g^1$, and $\GrLCS \fu_{g,1}$ (see \cref{sec:torelli-defs} for detailed definitions). This result for $\GrLCS \fu_{g}$ has simultaneously been obtained by Felder--Naef--Willwacher \cite{FNW}; we discuss the relation between the arguments in \cref{rem:FNW}.

\begin{remark*}
	\begin{enumerate}[(i)]
		\item \cref{athm:main} and its variants imply Conjecture 16.2 of \cite{HainJohnson} and answer Questions 9.13 and 9.14 of \cite{HainLooijenga} affirmatively. By Corollary 16.5 of \cite{HainJohnson}, Conjecture 16.1 of loc.~cit.~is also true. This is part (iii) of the ``the most optimistic landscape'' in Section 19 of \cite{HainJohnson}.
		\item Remark \ref{rem:NotKoszul} shows that the Lie algebra $\GrLCS\ft_{g,1}$ cannot actually be Koszul.
	\end{enumerate}
\end{remark*}

Garoufalidis--Getzler \cite{GG} have used work of Looijenga and Madsen--Weiss to compute the stable character of the quadratic dual of $\GrLCS \fu_g$ as graded algebraic $\mr{Sp}_{2g}(\bZ)$-representation and then in Theorem 1.3 of loc.~cit.~have computed from this---under the assumption of stable Koszulness---the stable character of $\GrLCS \ft_g$ as a graded algebraic $\mr{Sp}_{2g}(\bZ)$-representation. In \cite[Sections 6, 8.1]{KR-WTorelli} we gave a different approach to the first of these computations, and explained how it applies to surfaces with a boundary component or a marked point too. This method, along with \cref{athm:main} and its generalisation, renders the characters of the graded algebraic $\mr{Sp}_{2g}(\bZ)$-representations $\GrLCS \ft_g$, $\GrLCS\ft^1_g$, $\GrLCS \ft_{g,1}$, $\GrLCS \fu_{g}$, $\GrLCS \fu_g^1$, and $\GrLCS \fu_{g,1}$, amenable to computer calculation in weight $\leq \tfrac{g}{3}$.

\vspace{1ex}

We further use \cref{athm:main} to analyse the map
\[\tau_{g,1} \colon \GrLCS \ft_{g,1} \lra \fh_{g,1} \subset \mr{Der}(\mr{Lie}(H))\]
obtained from the action of the mapping class group of $\Sigma_{g,1}$ on its fundamental group, whose image lies in the Lie subalgebra $\fh_{g,1} \subset \mr{Der}(\mr{Lie}(H))$ of \emph{symplectic derivations}. Following Hain we call this the \emph{geometric Johnson homomorphisms}, and Morita has asked whether it is injective in weight $\neq 2$. The map $\tau_{g,1}$ is one of Lie algebras with additional weight grading in the category of algebraic $\mr{Sp}_{2g}(\bZ)$-representations and here we give strong evidence for this injectivity by tightly constraining its kernel in a range.

\begin{atheorem}\label{athm:morita} 
In weight $\leq \frac{g}{3}$, the kernel of $\tau_{g,1}$ lies in the centre of the Lie algebra $\GrLCS \ft_{g,1}$ and consists of trivial $\mr{Sp}_{2g}(\bZ)$-representations.
\end{atheorem}

We also prove the analogous statement for $\ft_g^1$, $\ft_g$, $\fu_{g,1}$, $\fu_g^1$, and $\fu_g$.

\begin{remark*}\
	\begin{enumerate}[(i)]
		\item The map $\tau_{g,1}$ factors over a map $\GrLCS \ft_{g,1} \to \GrLCS \fu_{g,1}$ whose kernel is a central $\bQ$ in weight 2 on which $\mr{Sp}_{2g}(\bZ)$ acts trivially \cite[Theoren 3.4]{HainTorelli}, so \cref{athm:morita} is sharp in this sense. However, comparing \cite[Section 7]{GG} with \cite[Table 1]{MSS} one sees that $\GrLCS \fu_{g,1} \to \fh_{g,1}$ is injective in weight $\leq 6$ in a stable range, and it might well be the case that this map is injective.
		\item Stably and neglecting trivial $\mr{Sp}_{2g}(\bZ)$-representations in the centre, this answers several questions in the literature: part (i) of ``the most optimistic landscape'' in Section 1.9 of \cite{HainJohnson},  \cite[Question 9.7]{HainLooijenga}, \cite[Problem 6.2]{MoritaStructure}, \cite[Problem 3.3, Problem 8.1]{MoritaCohomological}, \cite[Question 8.2]{HabiroMassuyeau}, \cite[Conjecture 1.8]{MSSTorelli}. It also has consequences for finite type invariants of 3-manifolds (see \cref{rem:habiro-massuyeau}), c.f.~p.\,381 of \cite{MoritaStructure}.
	\end{enumerate}
\end{remark*}

Our arguments use in an essential manner the classifying spaces $B\Tor_\partial(W_{g,1})$ of the Torelli groups of the high-dimensional analogue of a surface of genus $g$ with one boundary component. As a consequence, we also obtain results about the rational homotopy Lie algebra of $B\Tor_\partial(W_{g,1})$ (see \cref{sec:HighDimApplications}).

\vspace{1ex}

\noindent\textbf{Outline of proofs.} We now outline the proof of \cref{athm:main}. We will suppress explicit ranges for the sake of readability, writing ``in a stable range'' when a statement holds in a range of degrees or weights tending to $\infty$ with $g$.
\begin{enumerate}[(1)]
	\item \label{step:reductions} \emph{Use Koszul duality, and relate different genera $g$ and dimensions $n$.} By a result of Hain, the Lie algebra $\mr{Gr}^\bullet_\mr{LCS}\,\ft_{g,1}$ is quadratically presented when $g \geq 4$. Hence, to prove \cref{athm:main}, it suffices to prove the quadratic dual commutative algebra of $\mr{Gr}^\bullet_\mr{LCS}\,\ft_{g,1}$ is Koszul in a stable range, in the sense that its bigraded commutative algebra homology groups vanish away from the diagonal.
	
	The quadratic dual of $\mr{Gr}^\bullet_\mr{LCS}\,\ft_{g,1}$ can be identified as an algebra of twisted Miller--Morita--Mumford classes as in \cite[Section 5]{KR-WTorelli}. Taking into account the $\mr{Sp}_{2g}(\bZ)$-action this description is uniform in $g$: there is a commutative algebra object $E_1/(\kappa_{e^2}) \colon \cat{dsBr} \to \cat{Gr}(\bQ\text{-}\cat{mod})$ in the category of representations of the downward signed Brauer category whose realisation $K^\vee \otimes^{\cat{dsBr}} E_1/(\kappa_{e^2}) \in \cat{Gr}(\cat{Rep}(\mr{Sp}_{2g}(\bZ)))$ is the quadratic dual of  $\mr{Gr}^\bullet_\mr{LCS}\,\ft_{g,1}$ in a stable range (implicitly $K^\vee$ depends on $g$ but $E_1/(\kappa_{e^2})$ does not). This realisation is Koszul in a stable range if and only if $E_1/(\kappa_{e^2})$ is if and only if $E_1$ is. The algebra object $E_1$ is one of a family $E_n$ which up to rescaling only depends on the parity of $n$, and this has a closely-related variant $Z_n$. The commutative algebra object $E_1$ is Koszul in a stable range if and only if $E_n$ is if and only if $Z_n$ is.
	\item \label{step:mfd-theory} \emph{Use high-dimensional manifold theory to get vanishing of the non-trivial $\mr{Sp}_{2g}(\bZ)$-representations.} The reason for the reductions performed in Step \eqref{step:reductions} is that the algebra object $Z_n$ has appeared in \cite{KR-WDisks}: its realisation $K^\vee \otimes^{\cat{d(s)Br}} Z_n$ is an algebra of twisted Miller--Morita--Mumford classes for the framed Torelli group $\mr{Tor}^\mr{fr}_\partial(W_{g,1})$ of the $2n$-dimensional analogue $W_{g,1} = D^{2n} \# (S^n \times S^n)^{\# g}$ of the surface $\Sigma_{g,1}$. These classes account for nearly all of the cohomology of this group: when $2n \geq 6$, a finite cover of the classifying space of this group fits in a fibration sequence
	\[X_1(g) \lra \overline{B\mr{Tor}}^\mr{fr}_\partial(W_{g,1}) \lra X_0\]
	of nilpotent spaces with $\mr{Sp}_{2g}(\bZ)$-action (the action is trivial on $X_0$) and $K^\vee \otimes^{\cat{d(s)Br}} Z_n = H^*(X_1(g);\bQ)$ in a stable range. It thus suffices to verify the diagonal criterion for the commutative algebra homology of $H^*(X_1(g);\bQ)$.
	
	There is an unstable rational Adams spectral sequence of $\mr{Sp}_{2g}(\bZ)$-representations
	\[H^\mr{Com}_s(H^*(X_1(g);\bQ))_{rn} \Longrightarrow \mr{Hom}(\pi_{rn-s+1}(X_1(g)),\bQ).\] 
	As $H^*(X_1(g);\bQ)$ is sparse, so are the entries in this spectral sequence. Not only does this imply a sparsity result for $\pi_*(X_1(g)) \otimes \bQ$, but we can exclude differentials in any desired range by increasing the dimension $2n$. In \cite{KR-WDisks} we also accessed the rational homotopy groups of $\smash{\overline{B\mr{Tor}}^\mr{fr}_\partial}(W_{g,1})$ using embedding calculus. Up to contributions from $\pi_*(X_0)\otimes \bQ$, this gives a different sparsity result for $\pi_*(X_1(g))$. Verification of the diagonal criterion \emph{up to trivial representations} then follows in a stable range by combining the two sparsity results.
	\item \emph{Apply a transfer argument in graph complexes to get vanishing of the trivial $\mr{Sp}_{2g}(\bZ)$-representations.} In terms of the commutative algebra object $Z_n$, we have proven that its commutative algebra homology vanishes away from the diagonal when evaluated on those objects of $\cat{d(s)Br}$ given by non-empty sets, and it remains to prove the same holds for the empty set. To do so, we resolve $Z_n$ by a complex $RB^{Z_n}$ of red-and-black graphs: the commutative algebra homology groups of $RB^{Z_n}$ equal those of $Z_n$, but $RB^{Z_n}$ is chosen so that these can be computed by taking strict commutative algebra indecomposables; this is a graph complex $RB^{Z_n}_\mr{conn}$ of connected red-and-black graphs. We show that this	is closely related to a complex $G^{Z_n}$ of connected black graphs, and in particular certain vanishing ranges can be passed back and forth between these complexes. 
	
	Step \eqref{step:mfd-theory} gives the required vanishing of the homology of $RB^{Z_n}_\mr{conn}(S)$ for all non-empty sets $S$, and in turn this implies a certain vanishing of the homology of $G^{Z_n}(S)$ for all non-empty sets $S$. A transfer argument shows that the homology of $G^{Z_n}(\varnothing)$ injects into that of $G^{Z_n}(\ul{1})$, giving a certain vanishing of the homology of $G^{Z_n}(\varnothing)$; this then implies the required vanishing of the homology of $RB^{Z_n}_\mr{conn}(\varnothing)$.
		
\end{enumerate}

\noindent We next outline the proof of \cref{athm:morita}. Using \cref{athm:main} we can identify the geometric Johnson homomorphism in terms of high-dimensional manifold theory: up to trivial representations, it amounts to the map induced on rational homotopy groups by the composition
\[X_1(g) \lra \smash{\overline{B\mr{Tor}}^\mr{fr}_\partial}(W_{g,1}) \lra B\mr{hAut}_*(W_{g,1}).\]
Its injectivity can then be deduced from the unstable rational Adams spectral sequence and sparsity results of Step \eqref{step:mfd-theory}.

\vspace{1ex}

\noindent\textbf{Acknowledgements.} The authors thank M.\ Kassabov and T.\ Willwacher for useful discussions, and both anonymous referees for helpful comments. AK acknowledges the support of the Natural Sciences and Engineering Research Council of Canada (NSERC) [funding reference number 512156 and 512250], as well as the Research Competitiveness Fund of the University of Toronto at Scarborough. AK was supported by an Alfred J.~Sloan Research Fellowship. ORW was partially supported by the ERC under the European Union's Horizon 2020 research and innovation programme (grant agreement No.\ 756444), and by a Philip Leverhulme Prize from the Leverhulme Trust.

\section{Preliminaries}\label{sec:preliminaries}

\begin{notation}Let $\bQ\text{-}\cat{mod}$ denote the category of $\bQ$-vector spaces, $\cat{Gr}(\bQ\text{-}\cat{mod})$ the category of non-negatively graded $\bQ$-vector spaces, and $\cat{Ch}$ the category of non-negatively graded chain complexes over $\bQ$. We consider $\cat{Gr}(\bQ\text{-}\cat{mod})$ as the subcategory of $\cat{Ch}$ of chain complexes without differential, and $\bQ\text{-}\cat{mod}$  as the subcategory of $\cat{Gr}(\bQ\text{-}\cat{mod})$ of graded vector spaces supported in degree 0. We endow $\cat{Ch}$ with the usual symmetric monoidality, incorporating the Koszul sign rule, and give the other categories the induced symmetric monoidalities.

We will want to consider objects (chain complexes, etc.) equipped with an additional $\bN$-grading, called the ``weight grading", and we write $\cat{Ch}^\bN := \cat{Fun}(\bN, \cat{Ch})$ and so on for the categories of these. Treating $\bN$ as a symmetric monoidal category having only identity maps and symmetric monoidal structure given by addition, Day convolution endows these categories with their own symmetric monoidalities. We emphasise that when interchanging elements the sign incurred depends only on the homological degree, and not on the weight. In particular $X \mapsto \bigoplus_{n \in \bN} X(n): \cat{Ch}^\bN \to \cat{Ch}$ is symmetric monoidal.
\end{notation}

\subsection{Commutative algebras} \label{sec:prelim-com} Let $\cat{Com}$ denote the \emph{nonunital commutative operad} in $\bQ\text{-}\cat{mod}$, having $\cat{Com}(r) = 0$ for $r=0$ and $\cat{Com}(r) = \bQ$ as a trivial $\fS_r$-representations otherwise. A \emph{commutative algebra} in $\cat{Ch}$ is an algebra for the operad $\cat{Com}$, or equivalently the associated monad. 

The (desuspended) \emph{Harrison complex} of commutative algebra $A$ is
\[\mr{Harr}(A)[-1] \coloneqq \mr{coLie}(A[1])[-1]\]
where $[1]$ denotes suspension and $[-1]$ denotes desuspension of chain complexes, and the differential $d=d_A + d_\mr{Harr}$ is the sum of the differential $d_A$ induced by that on $A$ and the differential $d_\mr{Harr}$ induced by the unique map of Lie coalgebras given on cogenerators by the commutative multiplication map $\mr{coLie}^2(A[1]) = \Lambda^2(A[1])= \mr{Sym}^2(A)[2] \to A[1]$. These satisfy $d_Ad_\mr{Harr}+d_\mr{Harr}d_A = 0$, so $d^2 = 0$.

If $A$ is a nonunital commutative algebra object in the subcategory $\cat{Gr}(\bQ\text{-}\cat{mod}) \subset \cat{Ch}$, then the objects $\mr{Harr}(A)$ inherit a further grading, as usual in homological algebra. We may think of this as follows: the grading of $A$ determines a $\bQ^\times$-action on $A$, where $u \in \bQ^\times$ acts as $u^q$ in degree $q$. Functoriality gives a $\bQ^\times$-action on  $\mr{Harr}(A)$, and we write $\mr{Harr}_{*}(A)_q$ for the subspace on which $u \in \bQ^\times$ acts as $u^q$. It is clear that such eigenspaces exhaust $\mr{Harr}(A)$, and we set
\[H^\mr{Com}_p(A)_q \coloneqq H_{p+q}(\mr{Harr}(A))_q = H_{p+q-1}(\mr{Harr}(A)[-1])_q,\]
and call $p$ the \emph{Harrison degree}, $q$ the \emph{internal degree}, and $p+q-1$ the \emph{total degree}.

If $A$ also has a weight grading, i.e.\ is a nonunital commutative algebra object in $\cat{Gr}(\bQ\text{-}\cat{mod})^\bN$, then the Harrison complex and its homology obtains a further weight grading, and we write
\[H^\mr{Com}_p(A)_{q,w} \coloneqq H_{p+q}(\mr{Harr}(A))_{q,w}\]
for the piece of \emph{weight} $w$.

\begin{definition}\label{def:koszul-com-1} For $A \in \cat{Alg}_{\mr{Com}}(\cat{Gr}(\bQ\text{-}\cat{mod})^\bN)$ we say that $A$ is \emph{Koszul in weight $ \leq W$} if $H^\mr{Com}_p(A)_{q,w}=0$ when $p \neq w$ and $w \leq W$. If $W=\infty$ then we simply say that $A$ is \emph{Koszul}.
\end{definition}

\subsection{Lie algebras}  Let $\cat{Lie}$ denote the \emph{Lie operad}, whose definition can be found in \cite[Section 13.2.3]{LodayVallette}. A \emph{Lie algebra} in $\cat{Ch}$ is an algebra for the operad $\cat{Lie}$, or equivalently the associated monad. The (desuspended) \emph{Chevalley--Eilenberg complex} of a Lie algebra $L$ is
\[\mr{CE}(L)[-1] = \mr{coCom}(L[1])[-1],\]
where the differential $d=d_L + d_\mr{CE}$ is again the sum of the differential $d_L$ induced by that on $L$ and the differential $d_\mr{CE}$ induced by the unique map of cocommutative coalgebras given on cogenerators by the Lie bracket map $\mr{coCom}^2(L[1]) = \mr{Sym}^2(L[1]) = \mr{Lie}^2(L)[2] \to L[1]$. These satisfy $d_Ld_\mr{CE}+d_\mr{CE}d_L = 0$, so $d^2 = 0$.

If $L$ is a Lie algebra object in $\cat{Gr}(\bQ\text{-}\cat{mod})$ then as above we can endow the Chevalley--Eilenberg complex with a further grading, and we set
\[H^\mr{Lie}_p(L)_q \coloneqq H_{p+q}(\mr{CE}(L))_q = H_{p+q-1}(\mr{CE}(L)[-1])_q.\]
If $L$ also has a weight grading, then we write
\[H^\mr{Lie}_p(L)_{q,w} \coloneqq H_{p+q}(\mr{CE}(L))_{q,w}\]
for the piece of weight $w$.

\begin{definition}\label{def:koszul-lie} For $L \in \cat{Alg}_\cat{Lie}(\cat{Gr}(\bQ\text{-}\cat{mod})^\bN)$ we say that $L$ is \emph{Koszul in weight $\leq W$} if $H^\mr{Lie}_p(L)_{q,w} = 0$ when $p \neq w$ and $w \leq W$. If $W=\infty$ then we simply say that $L$ is \emph{Koszul}.
\end{definition}

\subsection{Koszul duality}\label{sec:koszul-duality} The operads $\cat{Com}$ and $\cat{Lie}$ are Koszul dual \cite[Corollary 4.2.7]{GinzburgKapranov}, \cite[Proposition 13.1.5]{LodayVallette}, and thus there is a Koszul duality between nonunital commutative algebras and Lie algebras.

A \emph{quadratic datum for a nonunital commutative algebra} is a pair $(V,S)$ of a finite-dimensional graded vector space $V$ and a subspace $S \subset \mr{Com}(2) \otimes_{\fS_2} V^{\otimes 2}$. From this we may form the \emph{quadratic nonunital commutative algebra} $A(V,S)$ as the quotient of the free graded nonunital commutative algebra $\mr{Com}(V)$ by the ideal generated by $S$. As the relations are homogeneous in $V$ (namely purely quadratic) we can endow this with a weight grading by declaring $V$ to have weight 1. Similarly, a \emph{quadratic datum for a Lie algebra}, given by a pair $(W,R)$ of a finite-dimensional graded vector space $W$ and a subspace $R \subset \mr{Lie}(2) \otimes_{\fS_2} W^{\otimes 2}$, yields a \emph{quadratic Lie algebra} $L(W,R)$; it has a weight grading by declaring $W$ to have weight 1.

Given a quadratically presented nonunital commutative algebra $A=A(V,S)$, its \emph{quadratic dual} Lie algebra is $A^! \coloneqq L(V^\vee[-1],S^\perp[-2])$.\footnote{For convenience, our grading conventions are such that if the graded $\bQ$-vector space $V$ is concentrated in non-negative degrees then so is $V^\vee$.} Similarly, given a quadratically presented Lie algebra $L = L(W,R)$, its \emph{quadratic dual} nonunital commutative algebra is $L^! \coloneqq A(W^\vee[1],R^\perp[2])$. These constructions are natural in the quadratic data, and are dualities in the sense that there are natural isomorphisms $A \cong (A^!)^!$ and $L \cong (L^!)^!$. In particular, the quadratic dual preserves any group actions on the objects in question.

The definitions of Koszulness given in \cref{def:koszul-com-1} and \cref{def:koszul-lie} are the diagonal criteria; a reference is \cite[Theorem 4.9 (ii)]{Milles}. By \cite[Theorem 4.11]{Milles}, $A$ is Koszul if and only if $A^!$ is Koszul. The same is true for the notions of Koszulness in a range by the same argument:

\begin{lemma}\label{lem:koszul-lie-vs-com} A quadratic nonunital commutative algebra $A$ is Koszul in weight $\leq W$ if and only if its quadratic dual Lie algebra $A^!$ is Koszul in weight $\leq W$.\qed
\end{lemma}

\subsection{Orthogonal and symplectic representation theory}\label{sec:orth-symp-rep-theory} We recall some facts from \cite[Section 2.1]{KR-WTorelli}. Let $H(g)$ be a $2g$-dimensional rational vector space with non-singular $\epsilon$-symmetric pairing $\lambda \colon H(g) \otimes H(g) \to \bQ$ given by the hyperbolic form. Its dual $\epsilon$-symmetric form $\omega \colon \bQ \to H(g) \otimes H(g)$ characteristic by $(\lambda \otimes \mr{id})(- \otimes \omega) = \mr{id}(-)$. The automorphisms of $H(g)$ preserving $\lambda$ are denoted $\mr{O}_\epsilon(H(g))$; these are the $\bQ$-points of an algebraic group $\mathbf{O}_{g,g}$ for $\epsilon=1$ and $\mathbf{Sp}_{2g}$ for $\epsilon=-1$. Observe that $\mr{O}_\epsilon(H(g))$ is a subgroup of $\mr{GL}_{2g}(\bQ)$ and the intersection $\mr{O}_\epsilon(H(g)) \cap \mr{GL}_{2g}(\bZ)$ is $\mr{O}_{g,g}(\bZ)$ if $\epsilon=1$ and $\mr{Sp}_{2g}(\bZ)$ if $\epsilon=-1$.

For distinct $i$ and $j$ in $\{1,2,\ldots,q\}$, applying the pairing $\lambda$ to the $i$th and $j$th factors yields a map
\[\lambda_{i,j} \colon H(g)^{\otimes q} \lra H(g)^{\otimes q-2}\]
and inserting the form $\omega$ in the $i$th and $j$th factors yields a map
\[\omega_{i,j} \colon H(g)^{\otimes q-2} \lra H(g)^{\otimes q}.\]
We obtain irreducible representations of $\mr{O}_\epsilon(H(g))$ by
\begin{align*}H(g)^{[q]} &\coloneqq \ker\Big(H(g)^{\otimes q} \overset{\lambda_{i,j}}\lra \bigoplus_{i,j} H(g)^{\otimes q-2}\Big), \\
H(g)_{[q]} &\coloneqq \mr{coker}\Big(\bigoplus_{i,j} H(g)^{\otimes q-2} \overset{\omega_{i,j}}\lra H(g)^{\otimes q}  \Big),\end{align*}
and the composition $H(g)^{[q]} \to H(g)^{\otimes q} \to H(g)_{[q]}$ is an isomorphism.  The action of the symmetric group $\Sigma_q$ permuting the terms in the tensor product $H(g)^{\otimes q}$ descends to an action on $H(g)^{[q]}$ and $H(g)_{[q]}$.

Recall that the irreducible rational representations of $\fS_q$ are in bijection with partitions $\lambda$ of $q$ into positive integers; these are denoted $S^\lambda$. Then we define
\[V_\lambda \coloneqq [S^\lambda \otimes H(g)^{[q]}]^{\fS_q}.\]
By construction these are algebraic $\mr{O}_\epsilon(H(g))$-representations. They are either zero or irreducible, and the irreducible ones are distinct and exhaust all isomorphism classes.

We will have a use for representations of \emph{arithmetic subgroups} $G \subset \mr{O}_\epsilon(H(g))$, that is, groups $G \subset \mr{O}_\epsilon(H(g))$ which are commensurable with $\mr{O}_{g,g}(\bZ)$ or $\mr{Sp}_{2g}(\bZ)$. As long as $g \geq 2$, the $V_\lambda$ restrict to irreducible algebraic representations of $G$ and any algebraic representation of $G$ is a direct sum of these \cite[Section 2.1]{KR-WTorelli}.

\section{The Torelli Lie algebra}

\subsection{Basic results}\label{sec:basic-results}

\subsubsection{Definitions}\label{sec:torelli-defs}
Let $\Sigma^r_{g,n}$ be a surface of genus $g$ with $n$ boundary components and $r$ marked points, and $\Gamma^r_{g,n}$ its mapping class group, consisting of isotopy classes of orientation-preserving diffeomorphism fixing pointwise the boundary components and the marked points. Its \emph{Torelli group} is the subgroup $T_{g,n}^r \subset \Gamma_{g,n}^r$ of those isotopy classes which act as the identity on $H_1(\Sigma_g;\bZ)$ through the inclusion $\Sigma_{g,n}^r \hookrightarrow \Sigma_g$. Its pro-unipotent, or Malcev, completion $\ft^r_{g,n}$ is the Lie algebra of the initial pro-unipotent algebraic group over $\bQ$ under $T^r_{g,n}$. 

There is a tautological extension of groups $T_{g,n}^r \to \Gamma_{g,n}^r \to \mr{Sp}_{2g}(\bZ)$ and
the induced representation $\Gamma_{g,n}^r \to \mr{Sp}_{2g}(\bQ)$ into the $\bQ$-points of the algebraic group $\mathbf{Sp}_{2g}$ is Zariski dense. There is an initial pro-algebraic group $\mathcal{G}_{g,n}^r$ over $\bQ$ with Zariski dense homomorphism to $\mathbf{Sp}_{2g}$ and pro-unipotent kernel \cite[\S 2]{HainCompletions}. The Lie algebra of $\mathcal{G}_{g,n}^r$ is $\mathfrak{g}_{g,n}^r$ and the pro-nilpotent Lie algebra of this kernel is the \emph{relative unipotent completion} $\fu_{g,n}^r$ of $T_{g,n}^r$.  

Let $\fp^r_{g,n}$ denote the Lie algebra of the pro-unipotent completion of the fundamental group $\pi^r_{g,n}$ of the ordered configuration space in $\Sigma_g$ of $r$ points and $n$ points with non-zero tangent vector (in the notation of \cite{HainTorelli} we have $\fp_g = \fp^1_g$). The usual extensions of groups relating mapping class groups yield extensions of Lie algebras (Theorem 3.4 and Proposition 3.6 of \cite{HainTorelli}):

\begin{lemma}\label{lem:tg-ug-ses} For $g \geq 3$ there are extensions of Lie algebras
	\[\fp^r_{g,n} \lra \ft^r_{g,n} \lra \ft_g, \quad\quad\quad\quad \bQ \lra \ft^r_{g,n} \lra \ft^{r+1}_{g,n-1},\]
	 \begin{equation*}
	\bQ \lra \ft^r_{g,n} \lra \fu^r_{g,n}.\eqno\qed
	\end{equation*}
\end{lemma}

\subsubsection{Mixed Hodge structures and lower central series} The associated graded of the lower central series of a Lie algebra $\fg$ is a Lie algebra $\GrLCS \fg$. We consider this as a Lie algebra in $\cat{Gr}(\bQ\text{-}\cat{mod})^\bN$ by giving it homological degree 0, and giving $\mr{Gr}^w_\mr{LCS}\, \fg$ weight $w$. In the following there are unfortunately two things called \emph{weight}: our additional grading, and the Hodge theoretic weight filtration. We do not think any confusion is likely, but we shall refer to our weight grading as ``the additional grading" in places where confusion is possible.

As long as $g \geq 3$, Theorem 4.10 of \cite{HainTorelli} says that $\ft_{g,n}^r$ is equipped with a mixed Hodge structure whose weight filtration agrees, up to negation of the indexing, with the lower central filtration. By Corollary 4.8 and 4.9 of loc.\ cit., the same is true for $\fu_{g,n}^r$. Corollary 5.3 of loc.\ cit.\ (stated over $\bC$, see \cite[Section 7.4]{HainJohnson} for why it is also true over $\bQ$) implies that for $g \geq 3$ there are isomorphisms of pro-nilpotent Lie algebras
\[\ft^r_{g,n} \cong \prod_{s>0} \mr{Gr}^{s}_\mr{LCS}\, \ft^r_{g,n} \qquad \text{and}  \quad \fu^r_{g,n} \cong \prod_{s>0} \mr{Gr}^{s}_\mr{LCS}\, \fu^r_{g,n}.\]
Similarly $\fp^r_{g,n}$ comes with a mixed Hodge structure \cite[Theorem 1]{HaindeRhamI}, and by Section 2 and Lemma 4.7 of \cite{HainTorelli} its weight filtration agrees, up to negation of the indexing, with the lower central series filtration. 

The existence of compatible mixed Hodge structures implies that taking the associated graded of the weight filtration is exact and hence so is taking the associated graded of the lower central series filtration. In \cref{lem:tg-ug-ses} the two copies of $\bQ$ are the Hodge structures $\bQ(1)$, so we obtain (Section 13 and Theorem 4.10 of \cite{HainTorelli}), where we clarify that the $\bQ[2]$ are in weight $2$ and homological degree $0$.

\begin{lemma}\label{lem:grlcs-tg-ug-ses} For $g \geq 3$ there are extensions of Lie algebras with additional grading
	\[\GrLCS \fp^r_{g,n} \lra \GrLCS \ft^r_{g,n} \lra \GrLCS \ft_g \quad\quad\quad \bQ[2] \lra \GrLCS \ft^r_{g,n} \lra \GrLCS \ft^{r+1}_{g,n-1} ,\]
	\begin{equation*}
	\bQ[2] \lra \GrLCS \ft^r_{g,n} \lra \GrLCS \fu^r_{g,n}.\eqno \qed
	\end{equation*}
\end{lemma}

\subsubsection{Koszulness} 

The general version of \cref{athm:main} is as follows

\begin{theorem}\label{thm:mainPrime}
The Lie algebras with additional grading $\GrLCS \ft_{g,1}$, $\GrLCS \ft_{g}^1$, $\GrLCS \ft_g$, $\GrLCS \fu_{g,1}$, $\GrLCS \fu^1_g$, and $\GrLCS \fu_g$ are Koszul in weight $\leq \tfrac{g}{3}$.
\end{theorem}

The following proposition shows that \cref{athm:main} implies \cref{thm:mainPrime}, so we may focus on \cref{athm:main} for the remainder of this paper.

\begin{proposition}\label{prop:koszul} Suppose $g \geq 3$ and that $\GrLCS \ft_{g,1}$ is Koszul in weight $\leq W$. Then the same is true for $\GrLCS \ft_{g}^1$, $\GrLCS \ft_g$, $\GrLCS \fu_{g,1}$, $\GrLCS \fu^1_g$, and $\GrLCS \fu_g$.
\end{proposition}

\begin{proof}	Though we defined Koszulness in terms of Chevalley--Eilenberg homology, to verify its vanishing we may as well use Chevalley--Eilenberg cohomology, and take advantage of its multiplicative structure. We first prove that $\GrLCS \ft^1_g$ is Koszul in additional gradings $ \leq W$, by considering the spectral sequence in Lie algebra cohomology for the extension 
	\[ \bQ[2] \lra \GrLCS \ft_{g,1} \lra \GrLCS \ft_g^1\]
	of \cref{lem:grlcs-tg-ug-ses}.	This yields a Gysin sequence
	\[\cdots \lra H_\mr{Lie}^{p-2}(\GrLCS \ft_{g}^1)_{0,w-2} \overset{e \cdot -}\lra H_\mr{Lie}^{p}(\GrLCS \ft_{g}^1)_{0,w} \lra H_\mr{Lie}^{p}(\GrLCS \ft_{g,1})_{0,w} \lra \cdots\]
	If $w \leq W$ then the right-hand term is zero whenever $p \neq w$, in which case $H_\mr{Lie}^{p}(\GrLCS \ft_{g}^1)_{0,w}$ is infinitely divisible by $e$ and hence vanishes. Thus $\GrLCS \ft_{g}^1$ is also Koszul in additional gradings $ \leq W$.
	
	\vspace{.5em}
	
	We next prove that $\GrLCS\ft_g$ is Koszul in additional gradings $ \leq W$, using the extension
	\[ \GrLCS \fp^1_g \lra \GrLCS \ft^1_g \lra \GrLCS \ft_g\]
	of \cref{lem:grlcs-tg-ug-ses}. We will make use of continuous Lie algebra cohomology, see \cite[Section 5]{HainTorelli} for background. By Proposition 5.5 of \cite{HainTorelli} we have an isomorphism
	\[H_\mr{Lie}^*(\GrLCS \mathfrak{p}_g) = H_\mr{Lie}^*(\mr{Gr}^{-\bullet}_\mr{W} \mathfrak{p}_g) \cong  \mr{Gr}^{-\bullet}_\mr{W} H_{\mr{Lie}, \text{cts}}^*(\mathfrak{p}_g)\]
	and as surface groups are pseudo-nilpotent \cite{KohnoOda} we have $H^*_{\mr{Lie}, \text{cts}}(\mathfrak{p}_g) = H^*(\Sigma_g ; \bQ)$. In particular we have $H_\mr{Lie}^2(\GrLCS \mathfrak{p}_g) = \bQ$ and the cohomology in higher degrees vanishes. Thus, in the spectral sequence for this extension projection to the top row provides a fibre-integration map
	\[\pi_!\colon H_\mr{Lie}^*(\GrLCS \ft_g^1) \lra H_\mr{Lie}^{*-2}(\GrLCS \ft_g)\]
	which as usual is a map of $H^{*}_\mr{Lie}(\GrLCS \ft_g)$-modules and satisfies $\pi_!(e) = 2-2g \neq 0  \in H_\mr{Lie}^{0}(\GrLCS \ft_g)$ by comparison to the corresponding extension of groups $\pi_1(\Sigma_g) \to  T_g^1 \to T_g$. In particular this implies that $\pi^*\colon H_\mr{Lie}^{p}(\GrLCS \ft_g)_{0,w} \to H_\mr{Lie}^{p}(\GrLCS \ft_g^1)_{0,w}$ is injective. Assuming $\GrLCS \ft_g^1$ is Koszul the target of this map vanishes for $p \neq w$ and $w \leq W$: thus the source does too.
	
	\vspace{.5em}
	
	For $\fu_{g,1}$ we use the extension $\bQ(1) \to \ft_{g,1} \to \fu_{g,1}$ 	of Lie algebras with mixed Hodge structure, and proceed as we did above for $\ft^1_g$, taking associated graded and using the Gysin sequence; $\fu_g^1$ and $\fu_g$ may be treated similarly.
\end{proof}

\begin{remark}\label{rem:NotKoszul}
The Lie algebra $\GrLCS \ft_{g}$ for $g \geq 4$ can never actually be Koszul, and hence neither can $\GrLCS \ft_{g}^1$ or $\GrLCS \ft_{g,1}$ be by the proof of \cref{prop:koszul}. If $\GrLCS \ft_{g}$ were Koszul then $\GrLCS \fu_{g}$ would be too by the Gysin sequence argument of Proposition \ref{prop:koszul} applied to 
\[\bQ[2] \lra \GrLCS \ft_{g} \lra \GrLCS \fu_{g}.\]
This would imply that the cohomology algebras $H^*_\mr{Lie}(\GrLCS \ft_{g})$ and $H^*_\mr{Lie}(\GrLCS \fu_{g})$ are both generated in tridegree $(1,0,1)$ by the same vector space $(\mr{Gr}_{\text{LCS}}^1\, \fu_{g})^\vee \overset{\sim} \to (\mr{Gr}_{\text{LCS}}^1\, \ft_{g})^\vee$. This is finite-dimensional for $g \geq 3$ by a theorem of Johnson \cite{JohnsonAb}, so the map
\[H^*_\mr{Lie}(\GrLCS \fu_{g}) \lra H^*_\mr{Lie}(\GrLCS \ft_{g})\]
is surjective and both cohomology groups vanish above some cohomological degree. Let the top degrees in which the cohomologies are nontrivial be called $U$ and $T$, so this surjection shows that $U \geq T$. On the other hand the Gysin sequence for the extension shows that $T = U+1$, a contradiction.
\end{remark}

As far as we can tell Koszulness of $\GrLCS\fu_g$ cannot be ruled out in this way:

\begin{question}Is it possible that $\GrLCS\fu_g$ is Koszul?\end{question}

\subsubsection{Quadratic presentations}\label{sec:quadratic-presentations} We will take from the work of Hain  only the \emph{fact} that the the Lie algebra $\GrLCS \ft_{g,1}$ is quadratically presented, and not a particular presentation: in the following section we will deduce from this fact a presentation that is convenient for us.

For $g \geq 6$, the following is \cite[Corollary 7.8]{HainTorelli} (beware that what is denoted $\ft_{g,n}$ in \cite[Theorem 7.8]{HainImproved} is denoted $\ft_{g}^n$ in \cite{HainTorelli} and here). It can be generalised to $g \geq 4$ using the results in \cite[Section 7]{HainImproved}.

\begin{theorem}[Hain]\label{thm:quadratic-presentations} For $g \geq 4$ and $r,n \geq 0$, the Lie algebras with additional grading $\GrLCS \ft_{g,r}^n$ and $\GrLCS \fu_{g,r}^n$ are quadratically presented.\qed
\end{theorem}

\subsection{Cohomology of $\ft_{g,1}$ in low degrees}\label{sec:CohTg}
Following Morita and Kawazumi \cite{KM}, in \cite[Section 5.1]{KR-WTorelli} we have constructed, for each $a,r \geq 0$ with $r+2a -2 \geq 0$ and tuple $v_1, v_2, \ldots, v_r \in H^1(\Sigma_g;\bQ)$, a cohomology class 
\[\kappa_{e^a}(v_1 \otimes v_2 \otimes \cdots \otimes v_r) \in H^{r-2 + 2a}(T_{g,1};\bQ)\]
generalising the Miller--Morita--Mumford class $\kappa_{e^a}$; we refer to it as a \emph{twisted MMM-class}. (We often write $1$ for $e^0$.) This construction is linear in the $v_i$, is alternating under permuting the $v_i$, and is $\mr{Sp}_{2g}(\bZ)$-equivariant. 

\begin{lemma}\label{lem:MMMrels}
If $\{a_i\}$ is a basis of $H^1(\Sigma_g;\bQ)$, and $\{a_i^\#\}$ is the dual basis characterised by $\langle a_i^\# \cdot a_j, [\Sigma_g] \rangle = \delta_{ij}$, then there are relations
\begin{align*}
\sum_i \kappa_{e^a}(v_1 \otimes \cdots \otimes v_j \otimes a_i) \cdot \kappa_{e^b}(a_i^\# \otimes v_{j+1} \otimes \cdots \otimes v_r) &= \kappa_{e^{a+b}}(v_1 \otimes \cdots \otimes v_r)\\
\sum_i \kappa_{e^a}(v_1 \otimes \cdots \otimes v_r \otimes a_i \otimes a_i^\#) &=\kappa_{e^{a+1}}(v_1 \otimes \cdots \otimes v_r)\\
\sum_{i,j,k} \kappa_1(a_i \otimes a_j \otimes a_k) \cdot \kappa_1(a_i^\# \otimes a_j^\# \otimes a_k^\#) &=0.
\end{align*}
\end{lemma}
\begin{proof}
The first two are given in \cite[Section 5.2]{KR-WTorelli}. Repeatedly applying the first two relations shows that $\sum_{i,j,k} \kappa_1(a_i \otimes a_j \otimes a_k) \cdot \kappa_1(a_i^\# \otimes a_j^\# \otimes a_k^\#) = -  \kappa_{e^2}$. Now $e^2 = p_1 = 3 \mathcal{L}_1$ and we have $\kappa_{\mathcal{L}_1}=0 \in H^2(T_{g,1};\bQ)$ as a consequence of the Family Signature Theorem (see e.g.\ \cite{AtiyahFib}).
\end{proof}

This construction for $a=0$ and $r=3$ gives a map of $\mr{Sp}_{2g}(\bZ)$-representations
\[\kappa_1\colon \Lambda^3 V_1 \lra H^{1}(T_{g,1};\bQ).\]
\begin{remark}The linear dual $H_{1}(T_{g,1};\bQ) \to \Lambda^3 V_1$ of $\kappa_1$ may be identified with the Johnson homomorphism \cite{JohnsonAb}.\end{remark}

Taking cup products of such classes determines an algebra homomorphism
\[\phi \colon \Lambda^*[\Lambda^3 V_1[1]] \lra H^{*}(T_{g,1};\bQ).\]
It follows from the discussion in \cite[Section 8.2]{KR-WTorelli} and the estimate in \cite[Section 9.3]{KR-WTorelli} that as long as $g \geq 4$ this map is an isomorphism in degree 1 and its kernel in degree 2 is spanned by the terms
\begin{equation}\label{eq:IH}\tag{IH}
\sum_i \kappa_1(v_1 \otimes v_2 \otimes a_i) \cdot \kappa_1(a_i^\# \otimes v_5 \otimes v_6) - \sum_i \kappa_1(v_1 \otimes v_5 \otimes a_i) \cdot \kappa_1(a_i^\# \otimes v_6 \otimes v_2)
\end{equation}
for $v_1, v_2, v_5, v_6 \in H^1(\Sigma_{g,1};\bQ)$, as well as the invariant vector
\begin{equation}\label{eq:Theta}\tag{$\Theta$}
\sum_{i,j,k} \kappa_1(a_i \otimes a_j \otimes a_k) \cdot \kappa_1(a_i^\# \otimes a_j^\# \otimes a_k^\#)
\end{equation}
(which, as we showed in the proof of Lemma \ref{lem:MMMrels} above, represents $-3 \kappa_{\mathcal{L}_1}$). Furthermore, in degrees $\leq 2$ the image of $\phi$ is the maximal algebraic subrepresentation $H^{*}(T_{g,1};\bQ)^{\text{alg}}$.

\begin{theorem}\label{thm:LowCohTg1}
There is a homomorphism of trigraded algebras
\[\psi \colon \frac{\Lambda^*[\Lambda^3 V_1[1,0,1]]}{((\ref{eq:IH}), (\text{\ref{eq:Theta}}))} \lra H_{\mr{Lie}}^{*}(\GrLCS \ft_{g,1})_{*,*},\]
which as long as $g \geq 4$ is an isomorphism in cohomological degrees $* \leq 2$.
\end{theorem}

\begin{proof}
We will work with the continuous Lie algebra cohomology $H_{\mr{Lie}, \text{cts}}^*(\ft_{g,1})$, cf. \cite[Section 5]{HainTorelli}. By \cite[Proposition 5.5]{HainTorelli} there is an isomorphism
\begin{equation}\label{eq:GrWCommutes}
\mr{Gr}^W_{\bullet} H_{\mr{Lie}, \text{cts}}^*(\ft_{g,1}) \cong H_\mr{Lie}^*(\mr{Gr}^W_\bullet \ft_{g,1}),
\end{equation}
and by \cite[Theorem 4.10]{HainTorelli} there is an identification
\begin{equation}\label{eq:WtIsLCS}
\mr{Gr}^W_\bullet \ft_{g,1} = \mr{Gr}_\mr{LCS}^{-\bullet} \ft_{g,1}.
\end{equation}
 By definition of the lower central series, this Lie algebra with additional grading is generated by its piece of grading 1. Its grading 1 piece is by definition the abelianisation of the Lie algebra $\ft_{g,1}$, which is tautologically identified with $H_1(T_{g,1};\bQ)$ and so for $g \geq 3$ is the algebraic $\mr{Sp}_{2g}(\bZ)$-representation $\Lambda^3 V_1$ by a theorem of Johnson \cite{JohnsonAb}. It follows that the Lie algebra with additional grading $\GrLCS \ft_{g,1}$ is a finite-dimensional algebraic $\mr{Sp}_{2g}(\bZ)$-representation in each additional grading, so $\smash{\mr{Gr}^W_{\bullet}} H_{\mr{Lie}, \text{cts}}^*(\ft_{g,1})$ is a finite-dimensional algebraic $\mr{Sp}_{2g}(\bZ)$-representation in each bigrading. As the weight filtration is finite in each cohomological degree, it follows that $H_{\mr{Lie}, \text{cts}}^*(\ft_{g,1})$ is a finite-dimensional algebraic $\mr{Sp}_{2g}(\bZ)$-representation in each degree.

Consider the natural map
\begin{equation}\label{eq:ContToOrdinary}
H_{\mr{Lie}, \text{cts}}^*(\ft_{g,1}) \lra H^*(T_{g,1};\bQ).
\end{equation}
By the discussion above, it has image in the maximal algebraic subrepresentation  $ H^{*}(T_{g,1};\bQ)^{\text{alg}}$. As $H^1(T_{g,1};\bQ) \cong \Lambda^3 V_1$ is finite-dimensional, \cite[Proposition 5.1]{HainTorelli} implies that \eqref{eq:ContToOrdinary} is an isomorphism in degree 1 and a monomorphism in degree 2. Thus $\phi$ factors as
\[\phi \colon \Lambda^*[\Lambda^3 V_1[1]] \overset{\psi''}\lra H_{\mr{Lie}, \text{cts}}^*(\ft_{g,1}) \lra H^{*}(T_{g,1};\bQ)^{\text{alg}}\]
and the relations (\ref{eq:IH}) and (\ref{eq:Theta}) hold in $H_{\mr{Lie}, \text{cts}}^2(\ft_{g,1})$, so $\psi''$ descends to give
\[\frac{\Lambda^*[\Lambda^3 V_1[1]]}{((\ref{eq:IH}), (\text{\ref{eq:Theta}}))} \overset{\psi'}\lra H_{\mr{Lie}, \text{cts}}^{*}(\ft_{g,1}) \lra H^{*}(T_{g,1};\bQ)^{\text{alg}}\]
with $\psi'$ an isomorphism in degree 1 and a monomorphism in degree 2. By the discussion above the composition is an isomorphism in degrees $\leq 2$, so it follows that $\psi'$ is too.

It follows from \eqref{eq:GrWCommutes} and \eqref{eq:WtIsLCS} that $H_{\mr{Lie}, \text{cts}}^{1}(\ft_{g,1})$ is a pure Hodge structure of weight $-1$, so if we give the domain of $\psi'$ the weight filtration given by minus its degree then $\psi'$ is an isomorphism of weight-filtered graded vector spaces in degrees $* \leq 2$: passing to associated graded for the weight filtration, applying \eqref{eq:GrWCommutes} and \eqref{eq:WtIsLCS} again, and recalling that $\GrLCS \ft_{g,1}$ is supported in homological degree 0, gives the map $\psi$ in the statement of the lemma and shows it is an isomorphism in cohomological degrees $* \leq 2$.
\end{proof}

\begin{remark}
	It will be a consequence of \cref{athm:main} that $\psi$ is in fact an isomorphism in a stable range of degrees. 
\end{remark}

Our next goal is to describe the quadratic presentation of $\GrLCS \ft_{g,1}$. This uses the following lemma:

\begin{lemma}\label{lem:quadpres}
Let $W$ be a vector space and $Q \leq \Lambda^2 W$ be a subspace, and form the quadratic Lie algebra $\mathfrak{g} = \mr{Lie}(W)/(Q)$. Then
\begin{enumerate}[(i)]
\item The abelianisation map $\mathfrak{g} \to W$ induces an isomorphism $W^\vee \overset{\sim}\to H_\mr{Lie}^1(\mathfrak{g})$.

\item The kernel of $\Lambda^2 H_\mr{Lie}^1(\mathfrak{g}) \to H_\mr{Lie}^2(\mathfrak{g})$ is the annihilator of $Q$.

\end{enumerate}
\end{lemma}
\begin{proof}
As $\mathfrak{g}$ is quadratically presented it admits an additional grading by weight, where $W$ is put in weight 1. This endows the Chevalley--Eilenberg chains $(\mr{coCom}(\mathfrak{g}[1])[-1], d_{CE})$ with an additional weight grading, with respect to which it takes the form (homological degree is displayed horizontally,  weight vertically)
\begin{equation*}
	\begin{tikzcd}[row sep=0.1cm]
0 & ? & ? \arrow[l, two heads]& ? \arrow[l]\\
0 & \mr{Lie}^2(W)/Q & \Lambda^2 W \arrow[l, two heads]& 0\\
0 & W & 0 & 0\\
0 & 0 & 0 & 0
	\end{tikzcd}
\end{equation*}
The maps from the second to the first column are given by $[-,-] \colon \Lambda^2 \mathfrak{g} \to \mathfrak{g}$ so are surjective in weights $>1$, as $\mathfrak{g}$ is generated in weight 1. Dualising this complex, it follows that $H^1_\mr{Lie}(\mathfrak{g})=W^\vee$ and that there is an exact sequence as claimed
\[0 \lra (\mr{Lie}^2(W)/Q)^\vee \lra \Lambda^2 W^\vee = \Lambda^2 H^1_\mr{Lie}(\mathfrak{g}) \lra H^2_\mr{Lie}(\mathfrak{g}).\qedhere\]
\end{proof}

Let us $R \coloneqq \langle (\ref{eq:IH}), (\text{\ref{eq:Theta}}) \rangle \leq \Lambda^2 \Lambda^3 V_1$ and $R^\perp \leq (\Lambda^2 \Lambda^3 V_1)^\vee \cong \Lambda^2 \Lambda^3 V_1$ for its annihilator. 

\begin{remark}\label{rem:tg1-explicit}
As long as $g \geq 6$ we have the decomposition
\[\Lambda^2 \Lambda^3 V_1 = 2 V_0 + 3 V_{1^2} + 2V_{1^4} + V_{1^6} + V_{2,1^2} + V_{2^2} + V_{2^2,1^2}\]
into irreducible $\mr{Sp}_{2g}(\bZ)$-representations. The vectors \eqref{eq:IH} generate the representation $V_0 + V_{1^2} + V_{2^2}$ (see \cite[Proof of Corollary 2.2 (ii)]{GN}\footnote{This paper is corrected in \cite{Akazawa}, but that does not affect this point.}), and the invariant vector \eqref{eq:Theta} can be checked not to lie in \eqref{eq:IH}. Therefore $R^\perp \cong  2 V_{1^2} + 2V_{1^4} + V_{1^6} + V_{2,1^2}  + V_{2^2,1^2}$. For $g < 6$ the representation theory is degenerate, but in principle a similar analysis can be made.
\end{remark}

The following is essentially in \cite{HainTorelli}, though Hain does not give the final answer in the case of $\ft_{g,1}$. The explicit presentation produced by Hain's method in this case was given in \cite{HabeggerSorger} (and agrees with the following).

\begin{corollary}[Hain] \label{cor:hain-presentation}
For $g \geq 4$ there is an isomorphism
\[\mr{Lie}(\Lambda^3 V_1 [0,1])/ ( R^\perp ) \overset{\sim}\lra \GrLCS \ft_{g,1}.\]
\end{corollary}

\begin{proof}
As $\GrLCS \ft_{g,1}$ admits a quadratic presentation by \cref{thm:quadratic-presentations}, it is isomorphic to $\mr{Lie}(W)/ ( Q )$ for some vector space $W$ of homological degree 0 and weight 1 and some $Q \leq \Lambda^2 W$, and we may apply Lemma \ref{lem:quadpres} to determine $W$ and $Q$ in terms of $H_\mr{Lie}^*(\GrLCS \ft_{g,1})$ for $* \leq 2$, which is given by \cref{thm:LowCohTg1}. This gives $W = (\Lambda^3 V_1)^\vee \cong \Lambda^3 V_1$ and $Q = R^\perp$.
\end{proof}

Using \cref{lem:koszul-lie-vs-com}, it follows that for $g \geq 4$ the Lie algebra $\GrLCS \ft_{g,1}$ is Koszul in weight $\leq W$ if and only if the commutative algebra object $\frac{\Lambda^*[\Lambda^3 V_1[1,1]]}{((\ref{eq:IH}), (\text{\ref{eq:Theta}}))}$ in $\cat{Gr}(\bQ\text{-}\cat{mod})^\bN$ is Koszul in weight $\leq W$, allowing us to work with the latter.

\section{Representations of the downward Brauer category}

In this section we define the objects $E_n$ and $Z_n$, which are augmented unital commutative algebra objects in the category of representations of the downward (signed) Brauer category, and explain how they relate to $\GrLCS \ft_{g,1}$ as well as each other.

\subsection{The downward Brauer category}\label{sec:Brauer}

The following is identical to \cite[Definition 2.15]{KR-WTorelli}. One thinks of the morphisms of this category as in Figure \ref{fig:brauerdownward}.

\begin{definition}
The \emph{downward Brauer category} $\cat{dBr}$ is the  $\bQ\text{-}\cat{mod}$-enriched category with objects finite sets, and with vector space of morphisms $\cat{dBr}(S,T)$ having basis the pairs $(f, m_S)$ of an injection $f\colon T \hookrightarrow S$ and an unordered matching $m_S$ on the set  $S \setminus f(T)$. It has a symmetric monoidality given by disjoint union.
\end{definition}

\begin{figure}[h]
	\begin{tikzpicture}
		\draw (0,3) to[out=-90,in=-90,looseness=1.5] (2,3);
		\draw (4,3) to[out=-90,in=90] (2,0);
		\draw (3,3) to[out=-90,in=90] (1,0);
		\draw (-2,3) to[out=-90,in=90] (-1,0);
		\draw (-1,3) to[out=-90,in=90] (0,0);
		\draw [line width=2mm,white] (1,3) to[out=-90,in=90] (3,0);
		\draw (1,3) to[out=-90,in=90] (3,0);
		\foreach \x in {-2,...,4}
		\node at (\x,3) {$\bullet$};
		\foreach \x in {-1,...,3}
		\node at (\x,0) {$\bullet$};
		\node at (-3,3) {$S$};
		\node at (-2,0) {$T$};
	\end{tikzpicture}
	\caption{A graphical representation of a morphism $(f,m_S)$ in $\cat{dBr}(S,T)$ from a $7$-element set $S$ to a $5$-element set $T$. The order of crossings is irrelevant.}
	\label{fig:brauerdownward}
\end{figure}
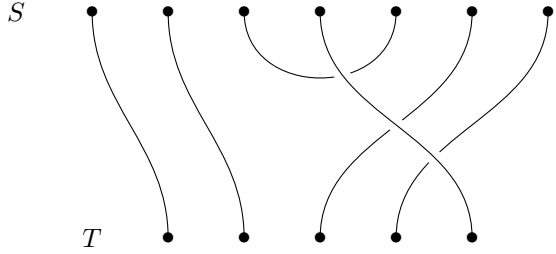

There is also the following signed version, identical to \cite[Definition 2.20]{KR-WTorelli}.

\begin{definition}
The \emph{downward signed Brauer category} $\cat{dsBr}$ is the $\bQ\text{-}\cat{mod}$-enriched category with objects finite sets, and with vector space of morphisms $\cat{dsBr}(S,T)$ given as follows. First let $\cat{dsBr}(S,T)'$ denote the vector space with basis the pairs $(f, m_S)$ of an injection $f \colon T \hookrightarrow S$ and an ordered matching $m_S$ of $S \setminus f(T)$. Then let $\cat{dsBr}(S,T)$ be the quotient space of $\cat{dsBr}(S,T)'$ by the subspace spanned by elements $(f, m_S) - (-1)^r (f, m'_S)$ whenever $m'_S$ differ from $m_S$ by reversing precisely $r$ pairs. It has a symmetric monoidality given by disjoint union.
\end{definition}

We will also write $\cat{FB}$ for the category of finite sets and bijections, and consider it (after $\bQ$-linearising) as a subcategory of both $\cat{dBr}$ and $\cat{dsBr}$, consisting of all objects and morphisms spanned by those $(f,m_S)$ with $f$ a bijection and $m_S=\varnothing$. To treat the symmetric and skew-symmetric cases simultaneously, we will write $\cat{d(s)Br}$ for $\cat{dBr}$ if $n$ is even and $\cat{dsBr}$ if $n$ is odd.

\vspace{1ex}

Recall that $\cat{Gr}(\bQ\text{-}\cat{mod})$ is the category of non-negatively graded $\bQ$-vector spaces, considered as a $\bQ\text{-}\cat{mod}$-enriched category. We shall be interested in (graded) representations of the downward (signed) Brauer category, i.e.\ enriched functors $F\colon \cat{d(s)Br} \to \cat{Gr}(\bQ\text{-}\cat{mod})$. The category of such functors $\cat{Fun}({\cat{d(s)Br}}, \cat{Gr}(\bQ\text{-}\cat{mod}))$ is again symmetric monoidal under Day convolution, using the symmetric monoidality on $\cat{d(s)Br}$ given by disjoint union and that on $\cat{Gr}(\bQ\text{-}\cat{mod})$ given by tensor product of graded vector spaces and the Koszul sign rule. Using the tensoring of $\cat{Gr}(\bQ\text{-}\cat{mod})$ over  $\bQ\text{-}\cat{mod}$, it is given by the enriched coend
\[(F \otimes G)(S) = \int^{S',S'' \in \cat{d(s)Br}} \cat{d(s)Br}(S' \sqcup S'',S) \otimes F(S') \otimes G(S''),\]
and the monoidal unit is the functor $\bunit \colon \cat{d(s)Br} \to \cat{Ch}$ given by $\bQ$ on the empty set and 0 on all other finite sets. One can compute the values of the Day convolution explicitly as
\begin{equation}\label{eqn:formula-day}\bigoplus_{S',S''} \cat{Pair}(S' \sqcup S'',S) \otimes_{\fS_{S'} \times \fS_{S''}} F(S') \otimes G(S'') \overset{\cong}\lra (F \otimes G)(S)\end{equation}
where the indexing set runs over all isomorphism classes of pairs of finite sets and $\cat{Pair}(S' \sqcup S'',S) \subset \cat{d(s)Br}(S' \sqcup S'',S)$ is the span of those $(f,m_{S' \sqcup S''})$ where each pair in $m_{S' \sqcup S''}$ contains elements of both $S'$ and $S''$. To see this, use that Day convolution preserves colimits in each variable to reduce to the case that $F$ and $G$ are representable, and verify it by hand there.

\subsubsection{Homotopy theory of representations of the downward Brauer category} We can do homotopy theory with these objects if we enlarge their target to $\cat{Ch}$. One can use $\infty$-categories or model structures; we opt for the latter because we need to do some explicit computations.

We endow $\cat{Ch}$ with the \emph{projective model structure} \cite[Theorem 1.5]{GoerssSchemmerhorn}: the weak equivalences are quasi-isomorphisms, the cofibrations are degreewise monomorphisms, and the fibrations are epimorphisms in positive degrees. Let us define chain complexes
\begin{equation}\label{eqn:gen-cof}\begin{aligned}D(n) &\coloneqq (\bQ\{z_n\}[n] \oplus \bQ\{z_{n-1}\}[n-1],dz_n=z_{n-1}) \\
	S(n-1) &\coloneqq (\bQ\{z_{n-1}\}[n-1],0).\end{aligned}\end{equation}
Then this model structure is cofibrantly generated by the set $I$ of generating cofibrations given by $0 \to S(0)$ and the inclusions $S(n-1) \hookrightarrow D(n)$ for $n \geq 1$, and the set $J$ of generating trivial cofibrations given by the inclusions $0 \hookrightarrow D(n)$ for $n \geq 1$ \cite[Examples 3.4(1)]{GoerssSchemmerhorn}. 

\begin{lemma}\label{lem:functor-model-cat}The category $\cat{Fun}(\cat{d(s)Br}, \cat{Ch})$ has a model structure whose weak equivalences and fibrations are objectwise. This is cofibrantly generated and monoidal with respect to Day convolution.\qed
\end{lemma}

The proof proceeds as in the unenriched case \cite[Theorem 11.6.1]{Hirschhorn} by right transfer as in \cite[Theorem 11.3.2]{Hirschhorn}. The set of generating cofibrations $I'$ is given by the union over all $S \in \mr{ob}(\cat{d(s)Br})$ of the image of $I$ under the left adjoint $S_*$ to the evaluation $ \cat{Fun}(\cat{d(s)Br},\cat{Ch}) \to \cat{Ch}$ given by $F \mapsto F(S)$, and the set of generating trivial cofibrations $J'$ is in the same manner with $J$ in place of $I$. That this model structure is monoidal is \cite[Proposition 2.2.15]{Isaacson}.

By \cref{lem:functor-model-cat}, the Day convolution tensor product preserves weak equivalences in each entry when all objects involved are cofibrant. The latter hypothesis is in fact unnecessary:

\begin{lemma}\label{lem:dsbr-weq} The Day convolution tensor product on $\cat{d(s)Br}$ preserves weak equivalences in each entry.\end{lemma}
	
\begin{proof}Using \eqref{eqn:formula-day}, the result follows from the facts that for rational chain complexes tensor products and quotients by finite group actions preserve weak equivalences.	
\end{proof}

\subsection{Commutative algebras on the downward Brauer category} Recall that $\cat{Com}$ denotes the nonunital commutative algebra operad. We write $\cat{Alg}_\cat{Com}(\cat{Fun}(\cat{d(s)Br}, \cat{Ch}))$ for the category of $\cat{Com}$-algebra objects in the symmetric monoidal category $\cat{Fun}(\cat{d(s)Br}, \cat{Ch})$; these are \emph{nonunital commutative algebra objects}.

\subsubsection{Unital commutative algebras}Let $\cat{Com}^+$ denote the \emph{unital commutative algebra operad}, having $\cat{Com}^+(r) = \bQ$ with trivial $\fS_r$-action for all $r \geq 0$.  We will write $\cat{Alg}_{\cat{Com}^+}(\cat{Fun}(\cat{d(s)Br}, \cat{Ch}))$ for the category of $\cat{Com}^+$-algebra objects in $\cat{Fun}(\cat{d(s)Br}, \cat{Ch})$; these are \emph{unital commutative algebra objects}.

The object $\bunit$ is canonically a unital commutative algebra object; an \emph{augmentation} for a unital commutative algebra object $R$ is a morphism $\epsilon \colon R \to \bunit$. Since the unit provides a section, there is a splitting
\[R = \bunit \oplus \ol{R}\]
with $\ol{R} = \ker(\epsilon)$ the \emph{augmentation ideal}.  This is canonically a nonunital commutative algebra object, and taking the augmentation ideal yields an equivalence of categories 
\begin{equation}\label{eqn:augm-equiv-cats}\cat{Alg}^\mr{augm}_{\cat{Com}^+}(\cat{Fun}(\cat{d(s)Br}, \cat{Ch})) \overset{\cong}\lra \cat{Alg}_{\cat{Com}}(\cat{Fun}(\cat{d(s)Br}, \cat{Ch}))\end{equation}
which has an inverse given by unitalisation.

\subsubsection{Homotopy theory of commutative algebras and Andr{\'e}--Quillen homology} \label{sec:aq-homology}To do homotopy theory with the nonunital commutative algebras in $\cat{Fun}(\cat{d(s)Br},\cat{Ch})$, we transfer the model structure to these objects.

\begin{lemma}\label{lem:alg-model-cat} The category $\cat{Alg}_{\cat{Com}}(\cat{Fun}(\cat{d(s)Br},\cat{Ch}))$ has a model structure whose weak equivalences and fibrations are objectwise. This is cofibrantly generated.\qed
\end{lemma}

The proof proceeds as in the unenriched case by right transfer \cite[Example 3.7]{GoerssSchemmerhorn}. The set of generating cofibrations is $F^\mr{Com}I'$ and the set of generating trivial cofibrations is $F^\mr{Com} J'$ with $F^\mr{Com}$ the left adjoint in the free-forgetful adjunction
\[\begin{tikzcd} \cat{Alg}_{\cat{Com}}(\cat{Fun}(\cat{d(s)Br},\cat{Ch})) \rar[shift left=.5ex]{U^\mr{Com}} &[10pt] \cat{Fun}(\cat{d(s)Br},\cat{Ch}). \lar[shift left=.5ex]{F^\mr{Com}}\end{tikzcd}\]

\begin{remark} The same construction endows $\cat{Alg}_{\cat{Com}^+}(\cat{Fun}(\cat{d(s)Br},\cat{Ch}))$ with a model structure whose weak equivalences and fibrations are objectwise. Moreover,  \cite[Theorem 7.6.5]{Hirschhorn} endows $\cat{Alg}^\mr{augm}_{\cat{Com}^+}(\cat{Fun}(\cat{d(s)Br},\cat{Ch}))$ with a model structure whose weak equivalences, cofibrations, and fibrations are such when forgetting the augmentation. Then \eqref{eqn:augm-equiv-cats} is a Quillen equivalence.\end{remark}

The monad $\mr{Com}$ has an augmentation $\epsilon \colon \mr{Com} \to \mr{Id}$ by projection to the $n=1$ summand, and we define \emph{indecomposables} $Q^{\mr{Com}}(\ol{R})$ of a nonunital commutative algebra object $\ol{R}$ as the following reflexive coequaliser in $\cat{Fun}(\cat{d(s)Br},\cat{Ch})$
\begin{equation*}
	\begin{tikzcd}
		\mr{Com}(\ol{R}) \arrow[r, shift left=1ex, "\text{act}"] \arrow[r, shift left=-1ex, swap, "\epsilon"]& \lar \ol{R} \rar & Q^\mr{Com}(\ol{R}).
	\end{tikzcd}
\end{equation*}
This is the left adjoint in an adjunction
\[\begin{tikzcd}\cat{Fun}(\cat{d(s)Br},\cat{Ch}) \rar[shift left=.5ex]{Z^\mr{Com}} &[10pt] \cat{Alg}_{\cat{Com}}(\cat{Fun}(\cat{d(s)Br},\cat{Ch})) \lar[shift left=.5ex]{Q^{\mr{Com}}}\end{tikzcd}\] 
with right adjoint the \emph{trivial algebra functor} $Z^\mr{Com}$ endowing $F$ with the zero multiplication. With the model structures of \cref{lem:functor-model-cat} and \cref{lem:alg-model-cat} this is a Quillen adjunction, because the right adjoint preserves the underlying objects and hence (trivial) fibrations. We write $\bL Q^\mr{Com}$ for the total left derived functor of $Q^\mr{Com}$.

\begin{definition}For a nonunital commutative algebra object $\ol{R} \colon \cat{d(s)Br} \to \cat{Ch}$ we define its \emph{Andr{\'e}--Quillen homology}
	\[\mr{AQ}_p(\ol{R}) \coloneqq H_p(\bL Q^\mr{Com}(\ol{R})).\]
\end{definition}

\begin{notation}By a slight abuse of notation, for an augmented commutative algebra $R$ with augmentation ideal $\ol{R}$ we write $\bL Q^{\mr{Com}^+}(R)$ for $\bL Q^{\mr{Com}}(\ol{R})$, and $\mr{AQ}_*(R)$ for $\mr{AQ}_*(\ol{R})$.\end{notation}

\subsubsection{Regular sequences}\label{sec:Regular} If $R \in \cat{Alg}^\mr{augm}_{\cat{Com}^+}(\cat{Fun}(\cat{d(s)Br}, \cat{Ch}))$ and $z \in \overline{R}(\varnothing)$ is a cycle of degree $k \geq 0$, then by adjunction there is a map $(\varnothing)_* S(k) \to R$, where $(\varnothing)_*$ denotes the left Kan extension along the inclusion $\{\varnothing\} \to \cat{d(s)Br}$ and $S(k)$ is the chain complex of \eqref{eqn:gen-cof}. Using that $R$ is a unital commutative algebra object, this further extends to a morphism $F^{\mr{Com}^+}((\varnothing)_* S(k)) \to R$, using which we may form the pushout 
\begin{equation}\label{eq:QuotByZ}
	\begin{tikzcd} 
		F^{\mr{Com}^+}((\varnothing)_* S(k)) \dar \rar& R \dar\\[-2pt]
		F^{\mr{Com}^+}(0) \rar& R/(z)
	\end{tikzcd}
\end{equation}
in the category $\cat{Alg}^\mr{augm}_{\cat{Com}^+}(\cat{Fun}(\cat{d(s)Br},\cat{Ch}))$.

\begin{definition}Let $R$ and $z$ be as above, then we say that $z$ is \emph{not a zerodivisor} if the morphisms $z \cdot -\colon \bQ[k] \otimes R(S) \to R(S)$ are injective for all finite sets $S$.\end{definition}

In particular, as $z^2 = (-1)^k z^2$ by graded-commutativity, $k$ must be even.

\begin{lemma}If $z \in \ol{R}(\varnothing)$ is not a zerodivisor then \eqref{eq:QuotByZ} is a homotopy pushout.
\end{lemma}

\begin{proof}
	Consider the generating cofibration $S(k) \hookrightarrow D(k+1)$ of \eqref{eqn:gen-cof}. Then
	\[F^{\mr{Com}^+}((\varnothing)_*S(k)) \lra F^{\mr{Com}^+}((\varnothing)_*D(k+1))\]
	is a cofibration modelling the left vertical map of \eqref{eq:QuotByZ} (it is a generating cofibration). If $cR \overset{\sim}\to R$ is a cofibrant replacement (we may assume it is surjective) and $z'  \in \overline{cR}(\varnothing)$ is a lift of $z$, the homotopy pushout is by definition given by
	\[cR \otimes_{F^{\mr{Com}^+}((\varnothing)_*S(k))} F^{\mr{Com}^+}((\varnothing)_*D(k+1))\]
	which is $cR \otimes F^{\mr{Com}^+}((\varnothing)_*S(k+1))$ with differential determined by $d (x \otimes 1) = d_{cR}(x) \otimes 1$, $d (1 \otimes z_{k+1}) = z' \otimes 1$, and the Leibniz rule.
	
	We thus first prove that the map
	\[cR \otimes_{F^{\mr{Com}^+}((\varnothing)_*S(k))} F^{\mr{Com}^+}((\varnothing)_*D(k+1)) \lra R \otimes_{F^{\mr{Com}^+}((\varnothing)_*S(k))} F^{\mr{Com}^+}((\varnothing)_*D(k+1))\]
	induced by $cR \overset{\sim}\to R$ is a weak equivalence. To do so, we use that there is an isomorphism
	\begin{equation}\label{eqn:f-com-iso}F^{\mr{Com}^+}((\varnothing)_*S(k+1)) \cong (\varnothing)_*(\Lambda[z_{k+1}])\end{equation}
	and filter both sides by powers of $z_{k+1}$ (there will be only two non-zero powers, $1$ and $z_{k+1}$, since $k+1$ is odd). On the associated gradeds we get either a shift of the map $cR \to R$ or the map $0 \to 0$ so the result follows.
	
	Next we observe that the isomorphism \eqref{eqn:f-com-iso} together with the fact that $k+1$ is odd, identifies $R \otimes F^{\mr{Com}^+}((\varnothing)_*S(k+1))$ with the mapping cone of
	\[z \cdot -\colon \bQ[k] \otimes R \lra R.\]
	As we have assumed that $z$ is not a zerodivisor this map is injective when evaluated on any finite set, so its cokernel and mapping cone are quasiisomorphic: it follows that $R \otimes F^{\mr{Com}^+}((\varnothing)_*S(k+1)) \simeq R/(z)$.
\end{proof}

\begin{corollary}\label{cor:RegSeq}
	If $z \in \ol{R}(\varnothing)$ is not a zerodivisor then there is a homotopy cofibre sequence
	\begin{equation*}
		(\varnothing)_*(\bQ\{z\}[k]) \lra \bL Q^{\mr{Com}^+}(R) \lra \bL Q^{\mr{Com}^+}(R/(z)).\eqno\qed
	\end{equation*}
\end{corollary}

\subsubsection{The Harrison complex}\label{sec:Harrison} Koszul duality between the nonunital commutative and Lie operads \cite[Section 13.1.5]{LodayVallette} gives for $\ol{R} \in \cat{Alg}_{\cat{Com}}(\cat{Fun}(\cat{d(s)Br}, \cat{Ch}))$, a weak equivalence of the form
\[F^{\mr{Com}}(\mr{coLie}(\ol{R}[1])[-1]) \overset{\sim}\lra \ol{R},\]
where the differential of the domain is a certain functorial deformation of the differential induced by that on $\ol{R}$ \cite[Corollary 11.3.5]{LodayVallette}, and the domain is in addition cofibrant \cite[Proof of Theorem 12.1.6]{LodayVallette}; the proofs in \cite{LodayVallette} are given in $\cat{Ch}$ but go through for $\cat{Fun}(\cat{d(s)Br},\cat{Ch})$. Taking indecomposables shows that the (shifted) \emph{Harrison complex} (c.f.~ \cref{sec:prelim-com}),
\[\mr{Harr}(\ol{R})[-1] \coloneqq \mr{coLie}(\ol{R}[1])[-1]\]
is a model for the derived indecomposables (cf.~\cite[Proposition 13.1.4]{LodayVallette}).

\begin{lemma}\label{lem:qcom-harr} $\bL Q^\mr{Com}(\ol{R}) \simeq \mr{Harr}(\ol{R})[-1]$.\end{lemma}

\begin{proof}For $c\ol{R} \overset{\sim}\lra \ol{R}$ a cofibrant replacement, there is a commutative diagram
	\[\begin{tikzcd} Q^\mr{Com}(c\ol{R}) & \lar \dar Q^\mr{Com}(F^\mr{Com}(\mr{coLie}(c\ol{R}[1])[-1])) \rar[equal] &  \mr{coLie}(c\ol{R}[1])[-1] \dar \\[-2pt]
		& Q^\mr{Com}(F^\mr{Com}(\mr{coLie}(\ol{R}[1])[-1])) \rar[equal] & \mr{coLie}(\ol{R}[1])[-1].\end{tikzcd}\]
	By definition $c\ol{R}$ is a cofibrant, and $F^{\mr{Com}}(\mr{coLie}(c\ol{R}[1])[-1])$ is too by the discussion above. The left horizontal map is thus a weak equivalence; it is a left Quillen functor applied to a weak equivalence between cofibrant objects. The right vertical map is a weak equivalence as $\mr{coLie}(-)$ preserves weak equivalences (as the tensor product does by \cref{lem:dsbr-weq} and coinvariants for a finite group action do).
\end{proof}

\subsubsection{Additional gradings and Koszulness}\label{sec:Gradings}
If a nonunital commutative algebra object $\ol{R}$ takes values in the subcategory $\cat{Gr}(\bQ\text{-}\cat{mod}) \subset \cat{Ch}$, then the objects $\mr{AQ}_p(\ol{R})$ inherit a further grading, as in \cref{sec:prelim-com}. We set
\[H^\mr{Com}_p(\ol{R})_q \coloneqq \mr{AQ}_{p+q-1}(\ol{R})_q \cong H_{p+q}(\mr{Harr}(\overline{R}))_q,\]
and call $p$ the \emph{Harrison degree}, $q$ the \emph{internal degree}, and $p+q-1$ the \emph{total degree}. If $\ol{R}$ also has a weight grading, i.e.\  $\ol{R} \in \cat{Alg}_{\cat{Com}}(\cat{Fun}(\cat{d(s)Br},\cat{Gr}(\bQ\text{-}\cat{mod})^\bN))$, then its derived indecomposables can be formed in $\cat{Ch}^\bN$, giving them an additional weight grading (and similarly for the Harrison complex) and we write
\[H^\mr{Com}_p(\ol{R})_{q,w} \coloneqq \mr{AQ}_{p+q-1}(\ol{R})_{q,w} \cong H_{p+q}(\mr{Harr}(\overline{R}))_{q,w},\]
for the piece of weight $w$.

\begin{definition}\label{def:koszul-com} For $\ol{R} \in \cat{Alg}_{\mr{Com}}(\cat{Fun}(\cat{d(s)Br},\cat{Gr}(\bQ\text{-}\cat{mod})^\bN))$ we say that $\ol{R}$ is \emph{Koszul in weight $ \leq W$} if $H^\mr{Com}_p(\ol{R})_{q,w}=0$ when $p \neq w$ and $w \leq W$. If $W=\infty$ then we simply say that $R$ is \emph{Koszul}.
\end{definition}

The following lemma gives an easy criterion for half of this inequation:

\begin{lemma}\label{lem:KoszEasyEstimate}
	If  each $\ol{R}(S)$ is supported in weight $\geq 1$ then $H^\mr{Com}_p(\ol{R})_{q,w} = 0 \text{ if } p > w$.
\end{lemma}
\begin{proof} Neglecting the differential, the Harrison complex splits as 
	\[\bigoplus_{k \geq 0} \mr{coLie}(k) \otimes_{\fS_k} (\ol{R}[1])^{\otimes k}.\]
	The weight of $\ol{R}[1]$ is $\geq 1$, so the weight of $(\ol{R}[1])^{\otimes p}$ is $\geq p$. As $H^\mr{Com}_p(\ol{R})$ is a subquotient of $\mr{coLie}(p) \otimes_{\fS_p} (\ol{R}[1])^{\otimes p}$, it is supported in weight $\geq p$.
\end{proof}

\subsection{Realisation}\label{sec:realisation}

If $H$ is a vector space with a (skew) symmetric form $\lambda$, then we will now explain how a representation of the downward (signed) Brauer category yields a representation of $\mr{Aut}(H,\lambda)$. 

\begin{definition}\label{defn:Realisation}
If $\lambda$ is a symmetric form on $H$ there is associated a functor
\[K\colon \cat{dBr} \lra \cat{Gr}(\bQ\text{-}\cat{mod})\]
given as follows: to the object $S$ we assign $K(S) = {H}^{\otimes S}$, considered as a graded vector space in grading zero, and to the morphism $(f, m_S)\colon S \to T$ we assign the map ${H}^{\otimes S} \to {H}^{\otimes T}$ given by applying ${\lambda}$ to each of the pairs in $m_S$, and using the bijection $f(T) \overset{\sim}\to T$ to induce an isomorphism ${H}^{\otimes f(T)} \overset{\sim}\to {H}^{\otimes T}$ on the remaining factors.

Similarly if $\lambda$ is a skew-symmetric form on $H$ there is associated a functor
\[K\colon \cat{dsBr} \lra \cat{Gr}(\bQ\text{-}\cat{mod})\]
given in the same way on objects, and given on a morphism $[(f,m_S)]\colon S \to T$ by applying $\lambda$ to each of the ordered pairs in $m_S$; as $\lambda$ is skew-symmetric this is well-defined.
\end{definition}

\begin{remark}In \cite{KR-WTorelli} $K$ was defined on the full Brauer category $\cat{(s)Br}$, and its restriction to $\cat{d(s)Br}$ was denoted $i^*K$. We have no need for the full Brauer category here and hence opt to simplify the notation.\end{remark}

In either case dualising $K$ gives a functor $K^\vee\colon  \cat{d(s)Br}^\mr{op} \to \cat{Gr}(\bQ\text{-}\cat{mod})$, and taking the coend with this defines a functor
\[K^\vee \otimes^{\cat{d(s)Br}} -\colon \cat{Fun}(\cat{d(s)Br},\cat{Gr}(\bQ\text{-}\cat{mod})) \lra \cat{Gr}(\bQ\text{-}\cat{mod}).\]
As explained in \cite[Section 2.2.1]{KR-WTorelli} (following \cite{SS}) this  has a strong symmetric monoidality.

The data $({H},{\lambda})$ is---tautologically---equipped with an action of the group $\mr{Aut}({H},{\lambda})$, and hence the functor $K$ is too. Thus the above construction may be promoted to a functor 
\[K^\vee \otimes^{\cat{d(s)Br}} -\colon \cat{Fun}(\cat{d(s)Br},\cat{Gr}(\bQ\text{-}\cat{mod}))  \lra \cat{Gr}(\cat{Rep}(\mr{Aut}({H},{\lambda}))),\]
i.e.\ to take values in graded $\mr{Aut}({H},{\lambda})$-representations.

From now on, we take $(H,\lambda)$ to be $(H(g),\lambda)$ as in \cref{sec:orth-symp-rep-theory}, so that $\mr{Aut}(H,\lambda) = \mr{O}_\epsilon(H(g))$. This functor is a left adjoint so is right exact, but is not left exact in general. However, it has the following stable left exactness property. Let us say that a functor $A\colon \cat{d(s)Br} \to \bQ\text{-}\cat{mod}$ is \emph{supported on sets of size $\leq N$} if $A(S)=0$ whenever $|S| > N$. (We take the target to be $\bQ$-modules rather than graded $\bQ$-modules as we will later wish to consider functors to graded $\bQ$-modules whose support depends on the grading.)

\begin{lemma}\label{lem:partial-inverse-to-realisation}  For each $A \in \cat{Fun}(\cat{d(s)Br},\bQ\text{-}\cat{mod})$ and finite set $S$ there is a surjective map
	\[A(S) \lra [H(g)_{[S]} \otimes (K^\vee \otimes^{\cat{d(s)Br}} A)]^{\mr{O}_\epsilon(H(g))}\]
which is injective when $|S| \leq g$.
\end{lemma}

\begin{proof}
	Recall from \cref{sec:orth-symp-rep-theory} that $H(g)_{[S]}$ is a certain quotient of $H(g)^{\otimes S}$ and consider
	\[[H(g)_{[S]} \otimes (K^\vee \otimes^{\cat{d(s)Br}} A)]^{\mr{O}_\epsilon(H(g))} = \int^{T \in \cat{d(s)Br}} [H(g)_{[S]} \otimes K(T)^\vee]^{\mr{O}_\epsilon(H(g))} \otimes A(T).\]
	There is a natural transformation of two variables
	\[\cat{d(s)Br}(T, S) \lra [H(g)_{[S]} \otimes K(T)^\vee]^{\mr{O}_\epsilon(H(g))}\]
	given by the functoriality of $K$, which is surjective, and is injective if $|S|+|T| \leq 2g$, by \cite[Theorem 2.6]{KR-WTorelli}. By the co-Yoneda lemma it gives a map
	\[A(S) \lra [H(g)_{[S]} \otimes (K^\vee \otimes^{\cat{d(s)Br}} A)]^{\mr{O}_\epsilon(H(g))}\]
	which is surjective, and injective for $|S| \leq g$.
\end{proof}

\begin{lemma}\label{lem:PartialExactness}
If $0 \to A \to B \to C \to 0$ is a short exact sequence in the category $\cat{Fun}(\cat{d(s)Br},\bQ\text{-}\cat{mod})$ such that $B$ is supported on sets of size $\leq g$, then
\[0 \lra K^\vee \otimes^{\cat{d(s)Br}} A \lra K^\vee \otimes^{\cat{d(s)Br}} B \lra K^\vee \otimes^{\cat{d(s)Br}} C \lra 0\]
is again exact.
\end{lemma}

(The following argument provides the missing second half of the proof of \cite[Corollary 2.18]{KR-WTorelli}, which is not as immediate as we had suggested.)

\begin{proof}
It is right exact as $K^\vee \otimes^{\cat{d(s)Br}} -$ is a left adjoint, so we only need to show that the left-hand map is injective. As $B$ is supported on sets of size $\leq g$ and $A$ is a subobject of $B$, $A$ also is supported on sets of size $\leq g$. 

The map 
	\[A(S) \lra [H_{[S]} \otimes (K^\vee \otimes^{\cat{d(s)Br}} A)]^{\mr{O}_\epsilon(H(g))}\] 
of \cref{lem:partial-inverse-to-realisation} is an isomorphism for $|S| \leq g$. As the domain vanishes if $|S| > g$, this map is an isomorphism. The same argument holds with $A$ replaced by $B$, and so we deduce that the map
\[[H_{[S]} \otimes (K^\vee \otimes^{\cat{d(s)Br}} A)]^{\mr{O}_\epsilon(H(g))} \lra [H_{[S]} \otimes (K^\vee \otimes^{\cat{d(s)Br}} B)]^{\mr{O}_\epsilon(H(g))}\]
is injective for all finite sets $S$.

If $K^\vee \otimes^{\cat{d(s)Br}} A \to K^\vee \otimes^{\cat{d(s)Br}} B$ were not injective then its kernel would contain an irreducible algebraic $\mr{O}_\epsilon(H(g))$-representation, which would be detected by applying $[H_{[S]} \otimes -]^{\mr{O}_\epsilon(H(g))}$ for some finite set $S$. Thus the map is injective as claimed.
\end{proof}

\subsubsection{Realisation and Koszulness}\label{sec:realisation-and-koszul} The above results imply that realisation preserves Koszulness in a range increasing with $g$.

\begin{lemma}\label{lem:realisation-h-com} 
Suppose that the weight $w$ piece of a nonunital commutative algebra $\ol{R} \in \cat{Alg}_{\mr{Com}}(\cat{Fun}(\cat{d(s)Br},\cat{Gr}(\bQ\text{-}\cat{mod})^\bN))$ is supported on sets of size $\leq \lambda w$ for some $\lambda \in \bN_{>0}$. Then as long as $w \leq \tfrac{1}{\lambda}g$ we have
	\[K^\vee \otimes^{\cat{d(s)Br}} H^\mr{Com}_p(\ol{R})_{q,w} \cong H_p^\mr{Com}(K^\vee \otimes^{\cat{d(s)Br}} \ol{R})_{q,w}.\]
\end{lemma}

\begin{proof}As an explicit model for the derived indecomposables of $\ol{R}$ we can take the Harrison complex (\cref{lem:qcom-harr}):
	\[\bL Q^\mr{Com}(\ol{R}) \simeq \mr{coLie}(\ol{R}[1])[-1] = \mr{Harr}(\ol{R})[-1].\]	
	The weight $w$ piece of $\ol{R}$ is supported on sets of size $\leq \lambda w$, i.e.\ $\ol{R}(S)_{q,w}=0$ if $w < \tfrac{1}{\lambda}|S|$. It follows that tensor powers of $\ol{R}$ have the same property, and so Schur functors of $\ol{R}$ have the same property, and hence the chain complex with additional grading $\mr{Harr}(\ol{R})$ has this vanishing property too. In particular, for each $w \leq \tfrac{1}{\lambda} g$, $\mr{Harr}(\ol{R})_{q,w}$ is a chain complex supported on sets of size $\leq g$, so by Lemma \ref{lem:PartialExactness} the operation $K^\vee \otimes^{\cat{dsBr}} -$ is exact on it. Thus, as long as $w \leq \tfrac{1}{\lambda} g$, we have
	\begin{align*}
		K^\vee \otimes^{\cat{d(s)Br}} H^\mr{Com}_{p}(\ol{R})_{q,w} &\cong H_{p+q}(K^\vee \otimes^{\cat{d(s)Br}} \mr{Harr}(\ol{R}))_{q,w} \\
		&\cong H_{p+q}( \mr{Harr}(K^\vee \otimes^{\cat{d(s)Br}} \ol{R}))_{q,w} \\
		&\cong H^\mr{Com}_{p}(K^\vee \otimes^{\cat{d(s)Br}} \ol{R})_{q,w}
	\end{align*}
	where for the second step we have used that $K^\vee \otimes^{\cat{d(s)Br}}-$ is an additive strong symmetric monoidal left adjoint
	and so commutes with the formation of the Harrison complex.
\end{proof}

\subsection{The main examples}\label{sec:MainEx1}

There are two families of examples we shall work with, which are both special cases of the functors $\mathcal{P}(-; \mathcal{B})'_{\geq 0} \otimes \det^{\otimes n}$ from \cite[Definition 1.3]{KR-WTorelli}, and we follow that definition.

\begin{definition}\label{def:zn}
A \emph{partition} of a finite set $S$ is a finite collection of possibly empty subsets $\{S_\alpha\}_{\alpha \in I}$ of $S$ which are pairwise disjoint and whose union is $S$. A partition is \emph{admissible} if each part has size $\geq 3$.

For $n \in \bN_{>0}$ let $Z_n\colon \cat{d(s)Br} \to \cat{Gr}(\bQ\text{-}\cat{mod})$ be the functor which to a finite set $S$ assigns the vector space 
\[Z_n(S) = \bQ\{\text{admissible partitions $\{S_\alpha\}_{\alpha \in I}$ of $S$}\} \otimes \det(\bQ^S)^{\otimes n},\]
made into a graded vector space by declaring a part $S_\alpha$ to have degree $n(|S_\alpha|-2)$, and a partition to have degree the sum of the degrees of its parts.

The linear map $Z_n(S) \to Z_n(T)$ induced by a bijection $(f, \varnothing)\colon S \to T$ is simply given by relabelling elements in a partition and by the induced map on determinants. The linear map induced by $(inc, (x,y))\colon S \to S \setminus \{x,y\}$ assigns to the element $[\{S_\alpha\}_{\alpha \in I}] \otimes (x \wedge y \wedge s_3 \wedge \cdots \wedge s_{|S|})^{\otimes n}$ the following:
\begin{enumerate}[(i)]
\item if some $S_\beta$ contains $\{x,y\}$ then it assigns 0,

\item if $x$ and $y$ lie in different parts $S_\beta$ and $S_\gamma$, then these are merged into a single new part $S_\beta \setminus\{x\} \cup S_\gamma \setminus\{y\}$ and it assigns
\[[\{S_\alpha\}_{\alpha \neq \beta, \gamma} \cup \{S_\beta \setminus\{x\} \cup S_\gamma \setminus\{y\}\}] \otimes (s_3 \wedge \cdots \wedge s_{|S|})^{\otimes n}.\]
\end{enumerate}
On a more general morphism in $\cat{d(s)Br}$ the effect of $Z_n$ is determined by the above and functoriality. 
\end{definition}

\begin{figure}[h]
	\begin{tikzpicture}
	\begin{scope}
		\draw[line width = 20pt,round cap-round cap,black!15!white] (0,1.1) -- (0,-3);
		\node at (0,-2.6) {$S$};
		\begin{scope}
			\foreach \i in {1,...,5}
			{
				\node at (.15,{-(\i-3)/3}) {\footnotesize \i};
				\draw (-1,0) -- (0,{(\i-3)/3});
			}
			\node at (-1,0) {$\bullet$};
		\end{scope}
		\begin{scope}
			\foreach \i in {6,...,8}
			{
				\draw (-1,-1.66) -- (0,{-(\i-2)/3});
				\node at (.15,{-(\i-2)/3}) {\footnotesize \i};
			}
			\node at (-1,-1.66) {$\bullet$};
		\end{scope}
		\draw[|->] (1,-.9) -- (2.25,-.9);
		\node at (1.66,-.6) {\small $(inc,\{1,2\})$};
		\begin{scope}[xshift=3.8cm]
			\draw[line width = 20pt,round cap-round cap,black!15!white] (0,1.1) -- (0,-3);
			\node at (0,-2.6) {$S'$};
			\begin{scope}
				\foreach \i in {1,...,5}
				{
					\draw (-1,0) -- (0,{(\i-3)/3});
				}
				\foreach \i in {3,...,5}
				{
					\node at (.15,{-(\i-3)/3}) {\footnotesize \i};
				}
				\node at (-1,0) {$\bullet$};
			\end{scope}
			\begin{scope}
				\foreach \i in {6,7,8}
				{
					\draw (-1,-1.66) -- (0,{-(\i-2)/3});
					\node at (.15,{-(\i-2)/3}) {\footnotesize \i};
				}
				\node at (-1,-1.66) {$\bullet$};
			\end{scope}
			\draw (0,{2/3}) to[out=0,in=0,looseness=7] (0,{1/3});
		\end{scope}
		\node at (4.9,-.9) {$= 0$};
	\end{scope}

	\begin{scope}[yshift=-5cm]
		\draw[line width = 20pt,round cap-round cap,black!15!white] (0,1.1) -- (0,-3);
		\node at (0,-2.6) {$S$};
		\begin{scope}
			\foreach \i in {2,...,6}
			{
				\node at (.15,{-(\i-4)/3}) {\footnotesize \i};
				\draw (-1,0) -- (0,{(\i-4)/3});
			}
			\node at (-1,0) {$\bullet$};
		\end{scope}
		\begin{scope}
			\foreach \i in {1}
			{
				\draw (-1,-1.66) -- (0,{-(\i+3)/3});
				\node at (.15,{-(\i+3)/3}) {\footnotesize \i};
			}
			\foreach \i in {7,8}
			{
				\draw (-1,-1.66) -- (0,{-(\i-2)/3});
				\node at (.15,{-(\i-2)/3}) {\footnotesize \i};
			}
			\node at (-1,-1.66) {$\bullet$};
		\end{scope}
		\draw[|->] (1,-.9) -- (2.25,-.9);
		\node at (1.66,-.6) {\small $(inc,\{1,2\})$};
		\begin{scope}[xshift=3.8cm]
			\draw[line width = 20pt,round cap-round cap,black!15!white] (0,1.1) -- (0,-3);
			\node at (0,-2.6) {$S'$};
			\begin{scope}
				\foreach \i in {2,...,6}
				{
					\draw (-1,0) -- (0,{(\i-4)/3});
				}
				\foreach \i in {3,...,6}
				{
					\node at (.15,{-(\i-4)/3}) {\footnotesize \i};
				}
				\node at (-1,0) {$\bullet$};
			\end{scope}
			\begin{scope}
				\foreach \i in {1}
				{
					\draw (-1,-1.66) -- (0,{-(\i+3)/3});
				}
				\foreach \i in {7,8}
				{
					\draw (-1,-1.66) -- (0,{-(\i-2)/3});
					\node at (.15,{-(\i-2)/3}) {\footnotesize \i};
				}
				\node at (-1,-1.66) {$\bullet$};
			\end{scope}
			\draw (0,{2/3}) to[out=0,in=0] (0,{-4/3});
		\end{scope}
		\node at (4.9,-.9) {$=$};
		\begin{scope}[xshift=2.4cm]
		\draw[line width = 20pt,round cap-round cap,black!15!white] (4,.9) -- (4,-2.7);
		\node at (4,-2.3) {$S'$};
		\begin{scope}
			\foreach \i in {3,...,8}
			{
				\node at (4.15,{-(\i-4)/3}) {\footnotesize \i};
				\draw (3,-.5) -- (4,{-(\i-4)/3});
			}
			\node at (3,-.5) {$\bullet$};
		\end{scope}
		\end{scope}
	\end{scope}
	\end{tikzpicture}
	\caption{The effect of the morphism $(inc,\{1,2\}) \colon S = \ul{8} \to S' = \ul{8} \setminus \{1,2\}$ on two elements of $Z_n(S)$: the first is the partition of $S$ into parts $S_\alpha = \{1,\ldots,5\}$ and $S_\beta = \{6,7,8\}$ and the second is the partition into parts $S'_{\alpha'} = \{2,\ldots,6\}$ and $S'_{\beta'} = \{1,7,8\}$. Similarly to \cite[Section 5.3]{KR-WTorelli}, we represent each part in a partition by a corolla whose legs are labelled by the elements in that part; we have suppressed the orientations. In this graphical notation, morphisms in $\cat{d(s)Br}$ act by reordering labels and connecting legs, collapsing internal edges and mapping to zero whenever a loop appears.}
\end{figure}
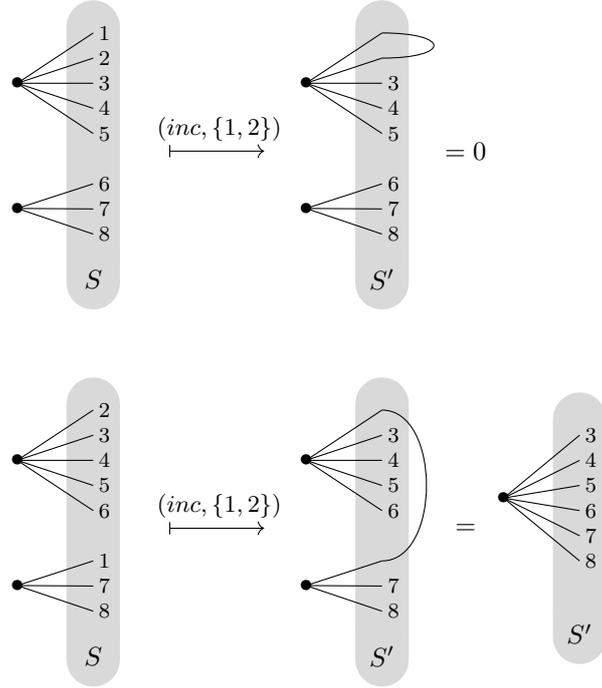

\begin{definition}\label{def:en}
A \emph{weighted partition} of a set $S$ is a partition $\{S_\alpha\}_{\alpha \in I}$ of $S$ along with a weight $g_\alpha \in \{0,1,2,3,\ldots\}$ for each part. A weighted partition is \emph{admissible} if
\begin{enumerate}[(i)]
\item each part of size 0 has weight $\geq 2$,

\item each part of size 1 or 2 has weight $\geq 1$.
\end{enumerate}

For $n \in \bN_{>0}$ let $E_n\colon \cat{d(s)Br} \to \cat{Gr}(\bQ\text{-}\cat{mod})$ be the functor which to a finite set $S$ assigns the vector space 
\[E_n(S) = \bQ\{\text{admissible weighted partitions $\{(S_\alpha, g_\alpha)\}_{\alpha \in I}$ of $S$}\} \otimes \det(\bQ^S)^{\otimes n},\]
made into a graded vector space by declaring a weighted part $(S_\alpha, g_\alpha)$ to have degree $n(2 g_\alpha + |S_\alpha|-2)$, and a weighted partition to have degree the sum of the degrees of its weighted parts.

The linear map $E_n(S) \to E_n(T)$ induced by a bijection $(f, \varnothing)\colon S \to T$ is simply given by relabelling elements in a partition and by the induced map on determinants. The linear map induced by $(inc, (x,y))\colon S \to S \setminus \{x,y\}$ assigns to the element $[\{(S_\alpha, g_\alpha)\}_{\alpha \in I}] \otimes  (x \wedge y \wedge s_3 \wedge \cdots \wedge s_{|S|})^{\otimes n}$ the following:
\begin{enumerate}[(i)]
\item if some $S_\beta$ contains $\{x,y\}$ then it assigns
\[[\{(S_\alpha, g_\alpha)\}_{\alpha \neq \beta} \cup \{(S_\beta \setminus \{x,y\}, g_\beta+1)\}] \otimes ( s_3 \wedge \cdots \wedge s_{|S|})^{\otimes n},\]

\item if $x$ and $y$ lie in different parts $S_\beta$ and $S_\gamma$ then it assigns
\[[\{(S_\alpha, g_\alpha)\}_{\alpha \neq \beta, \gamma} \cup \{(S_\beta \setminus\{x\} \cup S_\gamma \setminus\{y\}, g_\beta+g_\gamma)\}] \otimes ( s_3 \wedge \cdots \wedge s_{|S|})^{\otimes n}.\]
\end{enumerate}
On a more general morphism in $\cat{d(s)Br}$ the effect of $E_n$ is determined by the above and functoriality. 
\end{definition}

\begin{figure}[h]
	\begin{tikzpicture}
	\begin{scope}
		\draw[line width = 20pt,round cap-round cap,black!15!white] (0,1.1) -- (0,-3);
		\node at (0,-2.6) {$S$};
		\begin{scope}
			\foreach \i in {1,...,5}
			{
				\node at (.15,{-(\i-3)/3}) {\footnotesize \i};
				\draw (-1,0) -- (0,{(\i-3)/3});
			}
			\node at (-1,0) {$\bullet$};
			\node at (-1,0) [left] {$g_\alpha$};
		\end{scope}
		\begin{scope}
			\foreach \i in {6,...,8}
			{
				\draw (-1,-1.66) -- (0,{-(\i-2)/3});
				\node at (.15,{-(\i-2)/3}) {\footnotesize \i};
			}
			\node at (-1,-1.66) {$\bullet$};
			\node at (-1,-1.66) [left] {$g_\beta$};
		\end{scope}
		\draw[|->] (1,-.9) -- (2.25,-.9);
		\node at (1.66,-.6) {\small $(inc,\{1,2\})$};
		\begin{scope}[xshift=3.8cm]
			\draw[line width = 20pt,round cap-round cap,black!15!white] (0,1.1) -- (0,-3);
			\node at (0,-2.6) {$S'$};
			\begin{scope}
				\foreach \i in {1,...,5}
				{
					\draw (-1,0) -- (0,{(\i-3)/3});
				}
				\foreach \i in {3,...,5}
				{
					\node at (.15,{-(\i-3)/3}) {\footnotesize \i};
				}
				\node at (-1,0) {$\bullet$};
				\node at (-1,0) [left] {$g_\alpha$};
			\end{scope}
			\begin{scope}
				\foreach \i in {6,7,8}
				{
					\draw (-1,-1.66) -- (0,{-(\i-2)/3});
					\node at (.15,{-(\i-2)/3}) {\footnotesize \i};
				}
				\node at (-1,-1.66) {$\bullet$};
				\node at (-1,-1.66) [left] {$g_\beta$};
			\end{scope}
			\draw (0,{2/3}) to[out=0,in=0,looseness=7] (0,{1/3});
		\end{scope}
		\node at (4.9,-.9) {$=$};
		\begin{scope}[xshift=7.2cm]
			\draw[line width = 20pt,round cap-round cap,black!15!white] (0,.6) -- (0,-3);
			\node at (0,-2.6) {$S'$};
			\begin{scope}
				\foreach \i in {3,...,5}
				{
					\draw (-1,-.33) -- (0,{-(\i-3)/3});
				}
				\foreach \i in {3,...,5}
				{
					\node at (.15,{-(\i-3)/3}) {\footnotesize \i};
				}
				\node at (-1,-.33) {$\bullet$};
				\node at (-1,-.33) [left] {$g_\alpha{+}1$};
			\end{scope}
			\begin{scope}
				\foreach \i in {6,7,8}
				{
					\draw (-1,-1.66) -- (0,{-(\i-2)/3});
					\node at (.15,{-(\i-2)/3}) {\footnotesize \i};
				}
				\node at (-1,-1.66) {$\bullet$};
				\node at (-1,-1.66) [left] {$g_\beta$};
			\end{scope}
		\end{scope}
	\end{scope}
	
	\begin{scope}[yshift=-5cm]
		\draw[line width = 20pt,round cap-round cap,black!15!white] (0,1.1) -- (0,-3);
		\node at (0,-2.6) {$S$};
		\begin{scope}
			\foreach \i in {2,...,6}
			{
				\node at (.15,{-(\i-4)/3}) {\footnotesize \i};
				\draw (-1,0) -- (0,{(\i-4)/3});
			}
			\node at (-1,0) {$\bullet$};
			\node at (-1,0) [left] {$g_{\alpha'}$};
		\end{scope}
		\begin{scope}
			\foreach \i in {1}
			{
				\draw (-1,-1.66) -- (0,{-(\i+3)/3});
				\node at (.15,{-(\i+3)/3}) {\footnotesize \i};
			}
			\foreach \i in {7,8}
			{
				\draw (-1,-1.66) -- (0,{-(\i-2)/3});
				\node at (.15,{-(\i-2)/3}) {\footnotesize \i};
			}
			\node at (-1,-1.66) {$\bullet$};
			\node at (-1,-1.66) [left] {$g_{\beta'}$};
		\end{scope}
		\draw[|->] (1,-.9) -- (2.25,-.9);
		\node at (1.66,-.6) {\small $(inc,\{1,2\})$};
		\begin{scope}[xshift=3.8cm]
			\draw[line width = 20pt,round cap-round cap,black!15!white] (0,1.1) -- (0,-3);
			\node at (0,-2.6) {$S'$};
			\begin{scope}
				\foreach \i in {2,...,6}
				{
					\draw (-1,0) -- (0,{(\i-4)/3});
				}
				\foreach \i in {3,...,6}
				{
					\node at (.15,{-(\i-4)/3}) {\footnotesize \i};
				}
				\node at (-1,0) {$\bullet$};
				\node at (-1,0) [left] {$g_{\alpha'}$};
			\end{scope}
			\begin{scope}
				\foreach \i in {1}
				{
					\draw (-1,-1.66) -- (0,{-(\i+3)/3});
				}
				\foreach \i in {7,8}
				{
					\draw (-1,-1.66) -- (0,{-(\i-2)/3});
					\node at (.15,{-(\i-2)/3}) {\footnotesize \i};
				}
				\node at (-1,-1.66) {$\bullet$};
				\node at (-1,-1.66) [left] {$g_{\beta'}$};
			\end{scope}
			\draw (0,{2/3}) to[out=0,in=0,looseness=1.2] (0,{-4/3});
		\end{scope}
		\node at (4.9,-.9) {$=$};
		\begin{scope}[xshift=3.6cm]
			\draw[line width = 20pt,round cap-round cap,black!15!white] (4,.9) -- (4,-2.7);
			\node at (4,-2.3) {$S'$};
			\begin{scope}
				\foreach \i in {3,...,8}
				{
					\node at (4.15,{-(\i-4)/3}) {\footnotesize \i};
					\draw (3,-.5) -- (4,{-(\i-4)/3});
				}
				\node at (3,-.5) {$\bullet$};
				\node at (3,-.5) [left] {$g_{\alpha'}{+}g_{\beta'}$};
			\end{scope}
		\end{scope}
	\end{scope}
\end{tikzpicture}
	\caption{The effect of the morphism $(inc,\{1,2\}) \colon S = \ul{8} \to S' = \ul{8} \setminus \{1,2\}$ on two elements of $E_n(S)$, once more representing each part in a partition by a corolla and suppressing the orientations. Each part $S_\alpha$ in a partition---that is, each corolla---has a weight $g_\alpha \geq 0$ indicated with a label at the vertex. Morphisms in $\cat{d(s)Br}$ act by reordering labels and connecting labels, collapsing internal edges (adding their weights) and removing loops (adding $1$ to the weight of the vertex it is attached to).}
\end{figure}
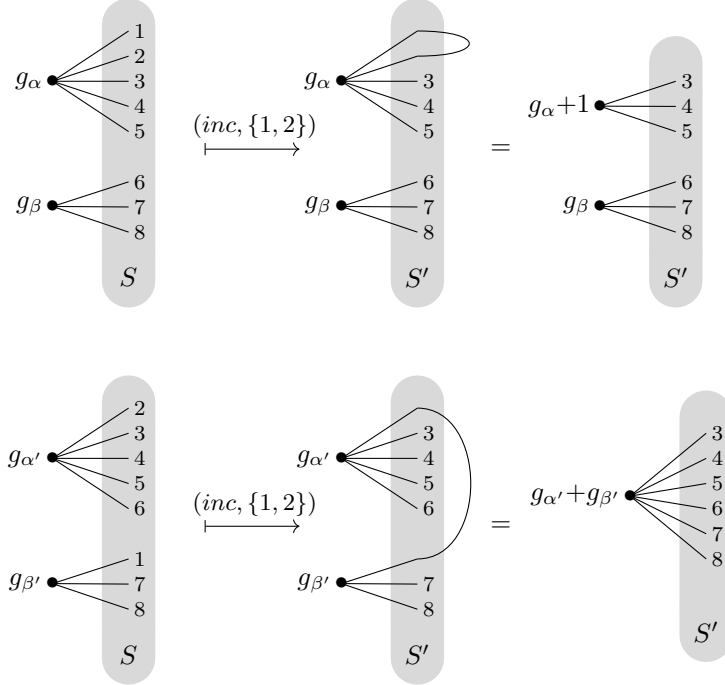

The functors $Z_n$ and $E_n$ have lax symmetric monoidalities given by disjoint union, making them into unital commutative algebra objects in $\cat{Fun}(\cat{d(s)Br}, \cat{Gr}(\bQ\text{-}\cat{mod}))$. As they are concentrated in homological degrees which are multiples of $n$, we can give them an additional weight grading by declaring it to be the homological degree divided by $n$. Furthermore both $Z_n$ and $E_n$ have unique augmentations, by sending all (weighted) partitions of non-empty sets to zero and the empty partition of the empty set to $1$. Up to scaling the homological grading each of these augmented unital commutative algebra objects only depends on the parity of $n$.

For each $n$ there is a (lax symmetric monoidal) epimorphism $E_n \to Z_n$ given by sending a part $(S_\alpha, 0)$ to $S_\alpha$, and parts with strictly positive weight to zero. It is compatible with the augmentations.

After having introduced $E_n$ and $Z_n$ we will now explain that, as commutative algebra objects, they differ only by the attachment of a single commutative algebra cell. As we mentioned above there is a map of commutative algebra objects $E_n \to Z_n$, and hence maps $\bL Q^{\mr{Com}^+}(E_n) \to \bL Q^{\mr{Com}^+}(Z_n)$. Taking the homotopy cofibre of this we may define the relative homology groups $H_p^\mr{Com}(Z_n, E_n)_{q,w}$ which participate in the usual long exact sequence.  We have the following description of this relative homology.

\begin{theorem}\label{thm:relative-com-zn-vs-en}
We have
\[H_p^\mr{Com}(Z_n, E_n)_{q,w}(S) \cong \begin{cases}
\bQ & \text{ if } |S|=1 \text{ and } (p,q,w)=(2,n,1),\\
0 & \text{ otherwise.}
\end{cases}\]
\end{theorem}
\begin{proof}
As the weight is given by the homological grading divided by $n$, we may neglect it. Let us write $\ul{n} \coloneqq \{1,2,\ldots, n\}$ and in particular $\ul{1} = \{1\}$.  There is a map $S(n) \to E_n(\ul{1})$ corresponding to the weighted partition $(\ul{1}, 1)$ of the set $\ul{1}$, with $S(n)$ as in \eqref{eqn:gen-cof}. This is adjoint to a map $\ul{1}_*S(n) \to E_n$ and hence, as the target is a unital commutative algebra object, to a map  $\phi \colon F^{\mr{Com}^+}(\ul{1}_*S(n)) \lra E_n$. Using this we may form the pushout
	\[\begin{tikzcd} 
F^{\mr{Com}^+}(\ul{1}_*S(n)) \rar{\phi} \dar & E_n \dar \\[-2pt]
F^{\mr{Com}^+}(\ul{1}_*D(n+1)) \rar & E'_n
		\end{tikzcd}\]
		of unital commutative algebra objects. As $(\ul{1}, 1)$ maps to to zero in $Z_n(\ul{1})$, there is a factorisation
		\begin{equation}\label{eq:comparison}
		E_n \lra E'_n \lra Z_n.
		\end{equation}
Although $E'_n : \cat{d(s)Br} \to \cat{Ch}$ does not take values in the subcategory $\cat{Gr}(\bQ\text{-}\cat{mod}) \subset \cat{Ch}$, we can nonetheless give it an ``internal grading" and ``weight" by declaring both $\ul{1}_* S(n)$ and $\ul{1}_* D(n+1)$ to have internal degree $n$ and weight 1. Then the maps \eqref{eq:comparison} both preserve these two additional gradings. We will show that the map $E'_n \to Z_n$ is a weak equivalence: as the cofibre of the map $\bL Q^{\mr{Com}^+}(E_n) \to \bL Q^{\mr{Com}^+}(E'_n)$ is $\ul{1}_* (D(n+1)/S(n)) = \ul{1}_* S(n+1)$ by construction (supported in total degree $n+1$, internal degree $n$, and weight $1$) the formula in the statement of the theorem follows.
		
	Pushouts of commutative algebra objects are given by the relative tensor product, and---neglecting the differential---we have
	\begin{align*}
	F^{\mr{Com}^+}(\ul{1}_*D(n+1)) &= F^{\mr{Com}^+}(\ul{1}_*S(n) \oplus \ul{1}_*S(n+1))\\
	&= F^{\mr{Com}^+}(\ul{1}_*S(n)) \otimes F^{\mr{Com}^+}(\ul{1}_*S(n+1))
	\end{align*}
	so that---neglecting the differential---we have
	\[E'_n = E_n \otimes F^{\mr{Com}^+}(\ul{1}_*S(n+1)).\]
	
	\vspace{1ex}
	
\noindent\textbf{Claim} We have $F^{\mr{Com}^+}(\ul{1}_*S(n+1))(T) \cong \det(\bQ^T)^{\otimes n+1}[(n+1)|T|]$, and the functoriality is such that a morphism $(f,m)$  in $\cat{d(s)Br}$ with $m \neq \varnothing$ acts as zero.
	
\begin{proof}[Proof of claim]
By definition of Day convolution we have
\[(\ul{1}_*S(n+1))^{\otimes p}(T) = S(n+1)^{\otimes p} \otimes \cat{d(s)Br}(\{1,2,\ldots, p\}, T)\]
where $\fS_p$ acts diagonally, and so
\[F^{\mr{Com}^+}(\ul{1}_*S(n+1))(T) = \bigoplus_{p \geq 0} S(n+1)^{\otimes p} \otimes_{\fS_p} \cat{d(s)Br}(\{1,2,\ldots, p\}, T).\]

Write $e_{n+1} \in S(n+1)$ for the basis element. If $(f, m) \in \cat{d(s)Br}(\{1,2,\ldots, p\}, T)$ has $m \neq \varnothing$, say with $(x,y) \in m$, then
\begin{align*}
(e_{n+1} \otimes \cdots \otimes e_{n+1}) \otimes_{\fS_p} (f, m) &= (x \, y) \cdot(e_{n+1} \otimes \cdots \otimes e_{n+1}) \otimes_{\fS_p} (x \, y) \cdot(f, m)\\
&= (-1)^{n+1} (e_{n+1} \otimes \cdots \otimes e_{n+1}) \otimes_{\fS_p} (-1)^n (f, m)\\
 &= - (e_{n+1} \otimes \cdots \otimes e_{n+1}) \otimes_{\fS_p} (f, m)
\end{align*}
and so this term vanishes. Hence only the term with $p=|T|$ contributes, giving $F^{\mr{Com}^+}(\ul{1}_*S(n+1))(T) = S(n+1)^{\otimes |T|} \otimes_{\fS_{|T|}} \cat{d(s)Br}(\{1,2,\ldots, |T|\}, T)$. The claim is simply an interpretation of this formula.
\end{proof}

Using the expression \eqref{eqn:formula-day} for Day convolution, neglecting the differential we have
\begin{equation}\label{eq:Eprime}
E'_n(S) = \bigoplus_{S', S''} \cat{Pair}(S' \sqcup S'', S) \otimes_{\fS_{S'} \times \fS_{S''}} E_n(S') \otimes \det(\bQ^{S''})^{\otimes n+1}[(n+1)|S''|].
\end{equation}
We will give an interpretation of this in terms of decorated partitions. Define a \emph{biweighted partition} of a set $S$ to be a partition $\{S_\alpha\}_{\alpha \in I}$ of \emph{a subset of} $S$ along with a weight $g_\alpha \in \{0,1,2,3,\ldots\}$ and a further weight $h_\alpha \in \{0,1\}$ for each part. A biweighted partition is \emph{admissible} if 
\begin{enumerate}[(i)]
\item when $|S_\alpha|=0$ then $g_\alpha+h_\alpha\geq 2$,

\item when $|S_\alpha|=1,2$ then $g_\alpha+h_\alpha\geq 1$.
\end{enumerate}
For such a biweighted partition we let $A \coloneqq \cup_{\alpha \in I} S_\alpha$, $B \coloneqq S \setminus A$, and $J \coloneqq \{\alpha \in I \, | \, h_\alpha = 1\}$.
Define
\[E''_n(S) = \bigoplus_{\substack{\text{admissible biweighted partitions} \\ \{(S_\alpha, g_\alpha, h_\alpha)\}_{\alpha \in I} \text{ of } S}} \det(\bQ^{J}) \otimes \det(\bQ^{B}) \otimes \det(\bQ^{S})^{\otimes n},\]
made into a graded vector space by declaring a biweighted partition to have degree $(n+1)|B| + \sum_{\alpha \in I} n(2 g_\alpha + |S_\alpha|-2) + (2n+1)h_\alpha$.

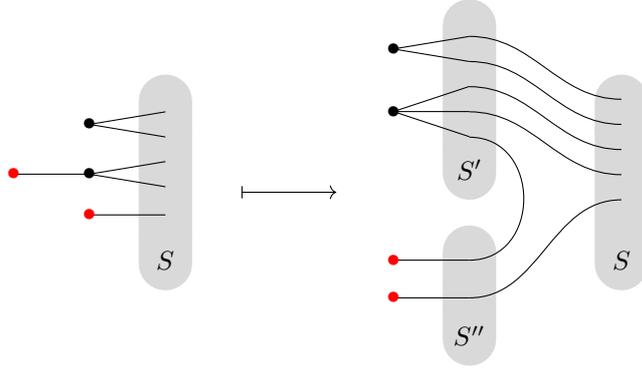
\begin{figure}[h]
\begin{tikzpicture}
	\begin{scope}
	\draw[line width = 20pt,round cap-round cap,black!15!white] (0,.66) -- (0,-2.2);
	\node at (0,-1.8) {$S$};
		\begin{scope}
		\foreach \i in {1,2}
		{
			\draw (-1,0) -- (0,{(\i-1.5)/3});
		}
		\node at (-1,0) {$\bullet$};
		\end{scope}
		\begin{scope}[yshift=-.66cm]
			\foreach \i in {1,2}
			{
				\draw (-1,0) -- (0,{(\i-1.5)/3});
			}
			\node at (-1,0) {$\bullet$};
			\draw (-2,0) -- (-1,0);
			\node at (-2,0) [red] {$\bullet$};
		\end{scope}
		\begin{scope}[yshift=-1.2cm]
			\draw (-1,0) -- (0,0);
			\node at (-1,0) [red] {$\bullet$};
		\end{scope}
	\end{scope}

	\draw[|->] (1,-.9) -- (2.25,-.9);
	
	\begin{scope}[xshift=4cm,yshift=1cm]
		\draw[line width = 20pt,round cap-round cap,black!15!white] (0,.66) -- (0,-2);
		\node at (0,-1.6) {$S'$};
		\begin{scope}
			\foreach \i in {1,2}
			{
				\draw (-1,0) -- (0,{(\i-1.5)/3});
			}
			\node at (-1,0) {$\bullet$};
		\end{scope}
		\begin{scope}[yshift=-.83cm]
			\foreach \i in {1,2,3}
			{
				\draw (-1,0) -- (0,{(\i-2)/3});
			}
			\node at (-1,0) {$\bullet$};
		\end{scope}
	\end{scope}

	\begin{scope}[xshift=4cm,yshift=-2cm]
		\draw[line width = 20pt,round cap-round cap,black!15!white] (0,.66) -- (0,-1.2);
		\node at (0,-.8) {$S''$};
		\draw (-1,0.2) -- (0,.2);
		\draw (-1,-.3) -- (0,-.3);
		\node at (-1,0.2) [red] {$\bullet$};
		\node at (-1,-.3) [red] {$\bullet$};
		\draw (0,.2) to[in=0,out=0,looseness=1.5] (0,{1.83});
	\end{scope}

	\begin{scope}[xshift=6cm]
		\draw[line width = 20pt,round cap-round cap,black!15!white] (0,.66) -- (0,-2.2);
		\node at (0,-1.8) {$S$};
		\begin{scope}
			\foreach \i in {1,...,4}
			{
				\draw (-2,{(\i-3)/3+5/6}) to[out=0,in=180] (0,{(\i-3)/3});
			}
		\end{scope}
		\draw (-2,-2.3) to[out=0,in=180] (0,-1);
	\end{scope}
\end{tikzpicture}
\caption{Intuitive indication of the map $\chi$. Weights $g_\alpha$ are not indicated, and weights $h_\alpha=1$ are indicated by a half-edge with a red end.}\label{fig:biweighted}
\end{figure}

Define a morphism of graded vector spaces
\[\chi \colon E''_n(S) \lra E'_n(S)\]
on  $(\{S_\alpha\}_{\alpha \in I}, g_\alpha, h_\alpha) \otimes (j_1 \wedge \cdots \wedge j_{|J|}) \otimes (b_1 \wedge \cdots \wedge b_{|B|}) \otimes (a_1 \wedge \cdots \wedge a_{|A|} \wedge b_1 \wedge \cdots \wedge b_{|B|})^{\otimes n}$ as follows.  Let $S' \coloneqq A \cup J$ and $S'' \coloneqq  B \cup J$, and let $\phi = (f, m) \colon S' \sqcup S'' \to S$ be the morphism in $\cat{d(s)Br}$ with injection $f \colon S \to S' \sqcup S''$ given by sending $s \in S$ to $s \in S'$ if $s \in A$, and to $s \in S''$ if $s \in B$, and matching $m$ given by pairing each element of $J  \subset S'$ with the same element of $J \subset S''$, putting that of $S'$ first. For $\alpha \in I$ let
\[S'_\alpha \coloneqq \begin{cases}
S_\alpha & \text{if $h_\alpha =0$,}\\
S_\alpha \cup \{\alpha\} & \text{if $h_\alpha =1$,}
\end{cases} \quad \subset S'.\]
Then $\phi$, $(\{S'_\alpha\}_{\alpha \in I}, g_\alpha) \otimes (a_1 \wedge \cdots \wedge a_{|A|} \wedge j_1 \wedge \cdots \wedge j_{|J|})^{\otimes n}$, $(b_1 \wedge \cdots \wedge b_{|B|} \wedge j_1 \wedge \cdots \wedge j_{|J|})^{\otimes n+1}$ 
represents an element of \eqref{eq:Eprime}. Note that permuting the $j_i$ acts by the sign, permuting the $a_i$ acts by the $n$th power of the sign, and permuting the $b_i$ acts via the $(n+1)$st power of the sign, making this map well-defined.

\vspace{1ex}
\noindent\textbf{Claim.} The map $\chi$ is an isomorphism.
\begin{proof}[Proof of claim]
Suppose we are given a morphism $\phi = (f, m) \colon S' \sqcup S'' \to S$, $(\{S'_\alpha\}_{\alpha \in I}, g_\alpha) \otimes (s'_1 \wedge \cdots \wedge s'_{|S'|})^{\otimes n} \in E_n(S')$, and  $(s''_1 \wedge \cdots \wedge s''_{|S''|})^{\otimes n+1} \in \det(\bQ^{S''})^{\otimes n+1}[(n+1)|S''|]$, representing an element of $E'_n(S)$ in the description \eqref{eq:Eprime}. 

Suppose that this data is such that there are ordered pairs $(s'_i, s''_i), (s'_j, s''_j)$ with $i \neq j$, but $s'_i, s'_j \in S'_\alpha$. The permutation $(s'_i \, s'_j)(s''_i \, s''_j)$ gives a morphism $\psi \colon S' \sqcup S'' \to S' \sqcup S''$ such that $\phi \circ \psi = \psi$. Thus the data above is equivalent to the data $\phi$, $(s'_i \, s'_j) \cdot ((\{S'_\alpha\}_{\alpha \in I}, g_\alpha) \otimes (s'_1 \wedge \cdots \wedge s'_{|S'|})^{\otimes n})$, $(s''_i \, s''_j) \cdot (s''_1 \wedge \cdots \wedge s''_{|S''|})^{\otimes n+1}$. That is: the data $\phi$, $(-1)^n (\{S'_\alpha\}_{\alpha \in I}, g_\alpha) \otimes (s'_1 \wedge \cdots \wedge s'_{|S'|})$, and $(-1)^{n+1}(s''_1 \wedge \cdots \wedge s''_{|S''|})^{\otimes n+1}$, and so minus the original element. Thus such elements vanish, and we may suppose that for each $\alpha$ there is at most one pair $(s',s'') \in m$ with $s' \in S'_\alpha$. Using this, it is easy to produce an inverse to the map $\chi$: we set $S_\alpha \coloneqq S'_\alpha \setminus \{s'\}$ and $h_\alpha=1$ if $(s',s'') \in m$ with $s' \in S'_\alpha$, and otherwise set $S_\alpha \coloneqq S'_\alpha$ and $h_\alpha=0$.
\end{proof}

The differential of $E'_n(S)$ is given under the isomorphism $\chi$ by sending the element $(\{S_\alpha\}_{\alpha \in I}, g_\alpha, h_\alpha) \otimes (j_1 \wedge \cdots \wedge j_{|J|}) \otimes (b_1 \wedge \cdots \wedge b_{|B|}) \otimes (a_1 \wedge \cdots \wedge a_{|A|} \wedge b_1 \wedge \cdots \wedge b_{|B|})^{\otimes n}$ of $E''_n(S)$ to
\begin{align*}
&\sum_{\beta \in J \subset I} (\{S_\alpha\}_{\alpha \in I}, g_\alpha + \delta_{\alpha \beta}, h_\alpha - \delta_{\alpha \beta}) \otimes (j_1 \wedge \cdots \wedge j_{|J|})/\beta \\
 &\quad\quad\quad\quad\quad \otimes (b_1 \wedge \cdots \wedge b_{|B|}) \otimes (a_1 \wedge \cdots \wedge a_{|A|} \wedge b_1 \wedge \cdots \wedge b_{|B|})^{\otimes n}\\
 &+ \sum_{b \in B} (\{S_\alpha\}_{\alpha \in I \cup \{b\}}, g_\alpha, h_\alpha) \otimes (j_1 \wedge \cdots \wedge j_{|J|})\\
&\quad\quad\quad\quad\quad \otimes (b_1 \wedge \cdots \wedge b_{|B|})/b \otimes (a_1 \wedge \cdots \wedge a_{|A|} \wedge b_1 \wedge \cdots \wedge b_{|B|})^{\otimes n}
\end{align*}
where $S_b = \{b\}$, $g_b = 1$, $h_b=0$. In particular, the partition $S = \bigcup_{\alpha} S_\alpha \cup \bigcup_{b \in B} \{b\}$ of $S$ is preserved by the differential, so we recognise the chain complex $E''_n(S)$ as a direct sum of complexes, one for each partition of $S$. Furthermore, we recognise the chain complex for the partition $S = \bigcup_{\alpha} S_\alpha \cup \bigcup_{b \in B} \{b\}$ as the tensor product of chain complexes, one for each part of this partition. 

\begin{enumerate}[(i)]
\item That corresponding to the part $S_\alpha$ is given by
\[\bigoplus_{g_\alpha + 1 \geq r_\alpha} (S_\alpha, g_\alpha, 1) \otimes \alpha \otimes 1  \overset{d}\lra \bigoplus_{g_\alpha \geq r_\alpha} (S_\alpha, g_\alpha, 0) \otimes 1 \otimes 1 \]
(tensored with $\det(\bQ^{S_\alpha})^{\otimes n}$) with $d((S_\alpha, g_\alpha, 1) \otimes \alpha \otimes 1) = (S_\alpha, g_\alpha+1, 0) \otimes 1 \otimes 1$, where $r_\alpha$ is 2 if $|S_\alpha|$ is 0, is 1 is $|S_\alpha|$ is 1 or 2, and is 0 otherwise. The homology of this complex is given by $\bigoplus_{0 \geq r_\alpha} (S_\alpha, 0, 0) \otimes 1 \otimes 1$, i.e.\ is 1-dimensional as long as $|S_\alpha| > 2$ and zero otherwise, recalling the requirement $g_\alpha \geq 0$.

\item That corresponding to the part $\{b\}$ with $b \in B$ is given by
\[(\varnothing, -, -) \otimes 1 \otimes s  \underset{\sim}{\overset{d}\lra} (\{b\}, 1, 0) \otimes 1 \otimes 1\]
(tensored with $\det(\bQ^{b})^{\otimes n}$) so is acyclic.
\end{enumerate}
Thus the homology of $E_n''(S)$ is supported on those partitions with $B=\varnothing$, $|S_\alpha| > 2$ and $g_\alpha=h_\alpha=0$, and we recognise it as being isomorphic to $Z_n(S)$. By considering the maps involved (or by counting dimensions) it follows that the map $E'_n(S) \to Z_n(S)$ is an isomorphism on homology, so $E'_n \to Z_n$ is a weak equivalence.
\end{proof}

\begin{corollary}\label{cor:EnZnEquiKoszul}
$E_n$ is Koszul if and only if $Z_n$ is.
\end{corollary}
\begin{proof}
In the portion of the long exact sequence
	\[\begin{tikzcd}
& H_2^\mr{Com}(E_n)_{n,1}(\ul{1})\rar & H_2^\mr{Com}(Z_n)_{n,1}(\ul{1}) \rar \ar[draw=none]{d}[name=X, anchor=center]{} & H_2^\mr{Com}(Z_n, E_n)_{n,1}(\ul{1}) = \bQ
\ar[rounded corners,
to path={ -- ([xshift=2ex]\tikztostart.east)
	|- (X.center) \tikztonodes
	-| ([xshift=-2ex]\tikztotarget.west)
	-- (\tikztotarget)}]{dll} \\
& H_1^\mr{Com}(E_n)_{n,1}(\ul{1}) \rar & H_1^\mr{Com}(Z_n)_{n,1}(\ul{1}) \rar & \cdots
\end{tikzcd}\]
we must show that the connecting map is injective. By the proof of the previous theorem, this connecting map sends the generator to the map
\[\ul{1}_*S(n) \overset{\phi}\lra \overline{E_n} \lra \bL Q^\mr{Com}(E_n),\]
a class in $\mr{AQ}_n(E_n)_{n,1}(\ul{1}) = H_1^\mr{Com}(E_n)_{n,1}(\ul{1})$. As $1$ is the lowest weight in which $\overline{E_n}$ has nontrivial elements, elements of this weight cannot be decomposable: thus as $\phi$ is nontrivial the connecting map is too.
\end{proof}

\subsection{Relation to Torelli Lie algebras}\label{sec:realisation-torelli}

The relation between $E_1$ and the Torelli Lie algebra was already indicated in \cite[Remark 8.4]{KR-WTorelli}, but we explain it again here.

\begin{notation}We write $\kappa_{e^j} \coloneqq (\varnothing, j) \in E_n(\varnothing)$.\end{notation}

\begin{theorem}\label{thm:TorRelation}The realisation 
\[K^\vee \otimes^{\cat{dsBr}} E_1/(\kappa_{e^2}) \in \cat{Gr}(\cat{Rep}(\mr{Sp}_{2g}(\bZ)))^\bN\]
is a commutative algebra object, and as long as $g \geq 4$ it agrees  in degrees $* \leq g$ with the quadratic dual of the Lie algebra object
\[\mr{Gr}^\bullet_\mr{LCS} \ft_{g,1}  \in \cat{Gr}(\cat{Rep}(\mr{Sp}_{2g}(\bZ)))^\bN.\]
\end{theorem}

\begin{proof}
In the language of \cite[Section 5]{KR-WTorelli} we have $E_1 = \mathcal{P}(-;\mathcal{V})'_{\geq 0} \otimes \det$ with $\mathcal{V}=\bQ[e]$ having $|e|=2$, and so $K^\vee \otimes^{\cat{dsBr}} E_1$ is the object denoted $R^\mathcal{V}$ there. By the discussion of \cite[Section 8.2]{KR-WTorelli} (especially Remark 8.4), and using the notation of Section \ref{sec:CohTg}, there is a map of graded commutative algebras
\[\frac{\Lambda^*[\Lambda^3 V_1[1]]}{((\ref{eq:IH}))} \lra R^\cV\]
which is an isomorphism in a stable range of degrees. As in the proof of Lemma \ref{lem:MMMrels} the element $\Theta = \sum_{i,j,k} \kappa_1(a_i \otimes a_j \otimes a_k) \cdot \kappa_1(a_i^\# \otimes a_j^\# \otimes a_k^\#)$ 
corresponds to $-\kappa_{e^2}$. As the weight grading on $E_1$ coincides with its homological grading, there is a map 
\[\frac{\Lambda^*[\Lambda^3 V_1[1,1]]}{((\ref{eq:IH}), (\text{\ref{eq:Theta}}))} \lra K^\vee \otimes^{\cat{dsBr}} E_1/(\kappa_{e^2})\]
which is an isomorphism in a stable range of degrees, and the domain is the quadratic dual of $\mr{Gr}^\bullet_\mr{LCS} \ft_{g,1}$ as long as $g \geq 4$ by the discussion in Section \ref{sec:CohTg}. Finally, the stability range is determined in \cite[Section 9.5]{KR-WTorelli}: this map is an isomorphism in degrees $* \leq g$.
\end{proof}

This yields a relation between Koszulness of $\GrLCS  \ft_{g,1}$ and of $E_1$.

\begin{proposition}\label{prop:E1KoszImpliesStablyKosz}
If $E_1$ is Koszul, then $K^\vee \otimes^{\cat{dsBr}} E_1/(\kappa_{e^2})$ is Koszul in weight $ \leq \tfrac{1}{3} g$.
\end{proposition}
\begin{proof}
Note that $E_1(\varnothing)$ is the polynomial algebra on the elements $\kappa_{e^j} = (\varnothing, j)$ for $j > 1$, and each $E_1(S)$ is a free $E_1(\varnothing)$-module, so the element $\kappa_{e^2} = (\varnothing,2) \in E_1(\varnothing)_{2,2}$ is not a zerodivisor in the sense of Section \ref{sec:Regular}, and hence we may apply Corollary \ref{cor:RegSeq}. Recalling the definition $\mr{AQ}_*({R}) = H_*(\bL Q^\mr{Com}(\overline{R}))$, 
there is a long exact sequence
\[\cdots \lra (\varnothing)_*(\bQ\{\kappa_{e^2}\}[2]) \lra \mr{AQ}_*(E_1) \lra \mr{AQ}_*(E_1/(\kappa_{e^2})) \lra \cdots.\]
As $\kappa_{e^2}$ is decomposable in $E_1$ by the formula for $\text{\eqref{eq:Theta}}$, taking internal gradings and weight into account this gives short exact sequences
\[0 \to H_p^{\mr{Com}}(E_1)_{q,w}(S) \to H_p^{\mr{Com}}(E_1/(\kappa_{e^2}))_{q,w}(S) \to \begin{cases}
\bQ & (p,q,w)=(2,2,2),\ S=\varnothing\\
0 & \text{ otherwise}
\end{cases} \to 0,\]
so $E_1$ is Koszul if and only if $E_1/(\kappa_{e^2})$ is. 

Observe that the weight $w$ part of $E_1/(\kappa_{e^2})$ is supported on sets of size $\leq 3w$, so by \cref{lem:realisation-h-com} we have
\[K^\vee \otimes^{\cat{dsBr}} H^\mr{Com}_p(E_1/(\kappa_{e^2}))_{q,w} \cong H^\mr{Com}_p(K^\vee \otimes^{\cat{dsBr}} E_1/(\kappa_{e^2}))_{q,w}\] 
as long as $w \leq \tfrac{1}{3}g$. Now if $E_1/(\kappa_{e^2})$ is Koszul then $H^\mr{Com}_{p}(E_1/(\kappa_{e^2}))_{q,w}=0$ for $p \neq w$, and it follows that $H^\mr{Com}_{p}(K^\vee \otimes^{\cat{dsBr}} E_1/(\kappa_{e^2}))_{q,w}=0$ for $p \neq w$ and  $w \leq \tfrac{1}{3} g$.
\end{proof}

\begin{corollary}\label{cor:E1Suffices}
If $E_1$ is Koszul then $\GrLCS \ft_{g,1}$ is Koszul in weight $\leq \tfrac{1}{3}g$.
\end{corollary}

\begin{proof}
If $E_1$ is Koszul then, by Proposition \ref{prop:E1KoszImpliesStablyKosz}, $K^\vee \otimes^{\cat{dsBr}} E_1/(\kappa_{e^2})$ is Koszul in weight $ \leq \tfrac{1}{3} g$. By Theorem \ref{thm:TorRelation}, $K^\vee \otimes^{\cat{dsBr}} E_1/(\kappa_{e^2})$ agrees with the quadratic dual of the Lie algebra $\GrLCS \ft_{g,1}$ in weight $ \leq g$, so by \cref{lem:koszul-lie-vs-com} this Lie algebra is Koszul in weight $ \leq \tfrac{1}{3}g$.
\end{proof}

For any odd $n \geq 1$ the commutative algebra object $E_n$ is isomorphic to $E_1$ with its homological grading scaled by $n$. Thus $E_1$ is Koszul if and only if $E_n$ is Koszul, and using \cref{cor:E1Suffices} and \cref{prop:koszul} this is what we shall show to prove \cref{athm:main} and hence \cref{thm:mainPrime}. We will do so by showing that the $Z_n$ are Koszul and appealing to \cref{cor:EnZnEquiKoszul}. Koszulness of the $Z_n$ has its own applications, to the rational homotopy Lie algebra of diffeomorphism groups of high-dimensional manifolds: we will explain these applications in \cref{sec:HighDimApplications}.

\section{High-dimensional manifolds: starting the proof of \cref{athm:main}} \label{sec:positive-arity}  
The goal of this section is to prove the following theorem, which \emph{almost} proves that the objects $Z_n, E_n \in \cat{Alg}_{\cat{Com}^+}^\mr{augm}(\cat{Fun}(\cat{d(s)Br}, \cat{Ch}^\bN))$ are Koszul.

\begin{theorem}\label{thm:KoszulPositive}
	If $q \neq n \cdot p$ then the functors
	\[H^\mr{Com}_{p}(Z_n)_{q,w} \colon \cat{d(s)Br} \lra \bQ\text{-}\cat{mod}\qquad \text{and} \qquad H^\mr{Com}_{p}(E_n)_{q,w} \colon \cat{d(s)Br} \lra \bQ\text{-}\cat{mod}\]
	vanish when evaluated on non-empty sets.
\end{theorem}
As $Z_n$ and $E_n$ have a weight grading which is equal to the homological grading divided by $n$, the above is equivalent to the vanishing on non-empty sets of $H^\mr{Com}_{p}(Z_n)_{q,w}$ and $H^\mr{Com}_{p}(E_n)_{q,w}$ for $p \neq w$, i.e.\ to Koszulness, but only on non-empty sets. Because of this, we can and will refrain from mentioning the weight grading from now on.

By \cref{cor:EnZnEquiKoszul} it is enough to prove \cref{thm:KoszulPositive} for $Z_n$, which we will do by relating $Z_n$ to the cohomology of moduli spaces of certain high-dimensional manifolds with framings. We have already developed these results in \cite{KR-WDisks}, relying on results from \cite{KR-WTorelli,KR-WAlg}, and it will suffice to merely recall them here.

\subsection{Torelli groups of high-dimensional manifolds}\label{sec:torelli-groups} For the remainder of this section we suppose that $2n \geq 6$. We will use the $2n$-manifolds 
\[W_{g,1} \coloneqq D^{2n} \# (S^n \times S^n)^{\# g}.\]

A \emph{tangential structure} for $2n$-manifolds can equivalently be described by a map $\theta \colon B \to B\mr{O}(2n)$, or by a $\mr{GL}_{2n}(\bR)$-space $\Theta$ (see \cite[Section 4.5]{grwsurvey}). Here we use the former, as in \cite{KR-WTorelli,KR-WAlg}, though we cite extensively from \cite{KR-WDisks}, which uses the latter. 

Fix a tangential structure $\theta \colon B \to B\mr{O}(2n)$. A \emph{$\theta$-structure} on $W_{g,1}$ is a map of vector bundles $\ell \colon TW_{g,1} \to \theta^* \gamma_{2n}$, where $\gamma_{2n}$ is the universal bundle over $B\mr{O}(2n)$. We will fix a boundary condition $\ell_\partial \colon TW_{g,1}|_{\partial W_{g,1}} \to \theta^*\gamma_{2n}$ and consider only those $\theta$-structures which extend it; we let $\mr{Bun}_\partial(TW_{g,1},\theta^*\gamma_{2n})$ denote the space of such bundle maps. The topological group $\Diff_\partial(W_{g,1})$ of diffeomorphisms of $W_{g,1}$ which fix its boundary pointwise (in the $C^\infty$-topology) acts on it through the derivative. We define
\[B\Diff^\theta_\partial(W_{g,1}) \coloneqq \mr{Bun}_\partial(TW_{g,1},\theta^*\gamma_{2n})\sslash \Diff_\partial(W_{g,1}).\]

We will only use tangential structures which satisfy the assumptions of Section 8 of \cite{KR-WAlg}. By Lemma 8.5 (i) of \cite{KR-WAlg}, there then exists up to homotopy a unique orientation-preserving boundary condition $\ell_\partial \colon TW_{g,1}|_{\partial W_{g,1}} \to \theta^* \gamma_{2n}$ which extends to $W_{g,1}$. For a $\theta$-structure $\ell$ on $W_{g,1}$, we denote by $B\Diff^\theta_\partial(W_{g,1})_\ell$ the path-component of $B\Diff^\theta_\partial(W_{g,1})$ which contains it. 

The intersection product endows $H_n(W_{g,1};\bZ)$ with a non-degenerate $(-1)^n$-symmetric form. Every diffeomorphism of $W_{g,1}$ induces an automorphism of this form, so we get a homomorphism 
\[\pi_1(B\Diff_\partial(W_{g,1})) \lra G_g = \begin{cases} 	\mr{Sp}_{2g}(\bZ) & \text{if $n$ is odd,} \\
	\mr{O}_{g,g}(\bZ) & \text{if $n$ is even.}\end{cases}\]
We will let $\smash{G^{\theta,[\ell]}_g} \subset G_g$ denote the image of the composition
\[\pi_1(B\Diff^\theta_\partial(W_{g,1})_\ell) \lra \pi_1(B\Diff_\partial(W_{g,1})) \lra G_g.\]
It is a finite index subgroup by \cite[Lemma 8.9]{KR-WAlg}. There is an induced map $B\Diff^\theta_\partial(W_{g,1})_\ell \to BG^{\theta,[\ell]}_g$, and we define
\[B\Tor^\theta_\partial(W_{g,1})_{\ell} \coloneqq \mr{hofib}\left[B\Diff^\theta_\partial(W_{g,1})_\ell \to BG^{\theta,[\ell]}_g\right].\]

\subsection{$Z_n$ and high-dimensional manifolds with framings} \label{sec:framed} For $Z_n$, we will use the above constructions with tangential structure $\theta_\mr{fr} \colon E\mr{O}(2n) \to B\mr{O}(2n)$ to obtain Torelli groups $B\mr{Tor}^\mr{fr}_\partial(W_{g,1})_\ell$ with framings. From this, we developed the commutative diagram \cite[(15)]{KR-WDisks} with rows and columns fibration sequences
\begin{equation*}
\begin{tikzcd}[column sep=.3in,row sep=.2in]X_1{}(g) \rar \dar & \overline{B\Tor}^\mr{fr}_\partial(W_{g,1})_\ell \rar \dar & X_0{} \dar \\ 
		A_1{}(g) \rar \dar & \overline{B\Diff}^\mr{fr}_\partial(W_{g,1})_\ell \rar{\overline{\alpha}^\mr{fr}} \dar  & \overline{\Omega^\infty_0 \mathbf{MT}\theta_{\mr{fr}}} \dar \\
		A_2{}(g) \rar & B\overline{G}{}^{\mr{fr},[\ell]}_g \rar & \overline{\overline{\Omega^\infty_0 \mathbf{KH}}} 
\end{tikzcd}
\end{equation*}
with overlines denoting a passage to a finite cover or finite subgroup which has no effect on rational cohomology (we refer to loc.~cit.~for definitions and details). The proof  of \cref{thm:KoszulPositive} uses the following facts about $X_1(g)$, which combines Lemma 4.9, Theorem 4.11, Corollary 4.20, and Remark 6.5 of \cite{KR-WDisks}.

\begin{theorem}\label{thm:pos-arity-framing-input} Fix $2n \geq 6$.
	\begin{enumerate}[(i)]
		\item \label{enum:pos-arity-framing-input-nilpotent} The space $X_1(g)$ is nilpotent and of finite type.
		\item \label{enum:pos-arity-framing-input-action} The action of $\pi_1(A_2(g))$ on $H^*(X_1(g);\bQ)$ and $\pi_*(X_1(g)) \otimes \bQ$ factors over $\overline{G}{}^{\mr{fr},[\ell]}_g$.
		\item \label{enum:pos-arity-framing-input-cohomology} There is a map $K^\vee \otimes^{\cat{d(s)Br}} Z_n \to H^*(X_1(g);\bQ)$ of $\overline{G}{}^{\mr{fr},[\ell]}_g$-representations which is an isomorphism in a range of homological degrees tending to $\infty$ as $g \to \infty$.
		\item \label{enum:pos-arity-framing-input-homotopy} In a range tending to $\infty$ as $g \to \infty$, in degrees $* \leq \tfrac{n(n-3)}{2}$ the non-trivial $\overline{G}{}^{\mr{fr},[\ell]}_g$-representations in $\pi_{*+1}(X_1(g)) \otimes \bQ$ vanish except with $* = r(n-1)$ for some $r \geq 1$.\qed
	\end{enumerate}
\end{theorem}

\subsection{Proof of Theorem \ref{thm:KoszulPositive}}\label{sec:proof-KoszulPositive}We first prove \cref{thm:KoszulPositive} for $Z_n$, using the above theorem as input. As a consequence of \cref{thm:pos-arity-framing-input} \eqref{enum:pos-arity-framing-input-nilpotent}, there is a strongly convergent unstable rational Adams spectral sequence
\[E^2_{s,t} = H^\mr{Com}_{s}(H^*(X_1(g);\bQ))_t \Longrightarrow \mr{Hom}(\pi_{t-s}(\Omega X_1(g)),\bQ),\]
with $d^r$-differential of bidegree $(-r,-r+1)$ which, by naturality and \cref{thm:pos-arity-framing-input} \eqref{enum:pos-arity-framing-input-action}, is a spectral sequence of $\overline{G}{}^{\mr{fr},[\ell]}_g$-representations.

By \cref{thm:pos-arity-framing-input} \eqref{enum:pos-arity-framing-input-cohomology}, in a stable range the rational cohomology of $X_1(g)$ is concentrated in degrees which are multiples of $n$. Thus in a stable range the groups $E^2_{s,t}$ are supported along the lines $t = r \cdot n$ with $r \in \bN$ and by the argument of Lemma \ref{lem:KoszEasyEstimate} we have $E^2_{s,rn}=0$ for $r < s$. Thus $E^2_{s,t}$ with $t=r\cdot n$ is nontrivial only for $1 \leq s \leq r$, contributing to degrees
\[r\cdot n-r+1 \leq t-s+1 \leq r\cdot n.\]
Furthermore, as long as these ranges of total degrees are separated from each other there can be no differentials between the corresponding lines: in particular, as long as $n \geq r-2$ the spectral sequence must collapse in degrees $t-s+1 \leq r \cdot n$. Thus as long as $r \leq n+2$, there is an isomorphism of $\overline{G}{}^{\mr{fr},[\ell]}_g$-representations in a stable range
\[H^\mr{Com}_{s}(H^*(X_1(g);\bQ))_{rn} \cong \mr{Hom}(\pi_{rn-s+1}(X_1(g)) , \bQ).\]
By \cref{thm:pos-arity-framing-input} \eqref{enum:pos-arity-framing-input-homotopy}, in degrees $* \leq \tfrac{n(n-3)}{2}$ the non-trivial representations on the right side vanish except when $r=s$. Hence the same is true for the left side.

By \cref{thm:pos-arity-framing-input} \eqref{enum:pos-arity-framing-input-cohomology} there is a map
\[K^\vee \otimes^{\cat{d(s)Br}} Z_n\lra H^*(X_1(g);\bQ)\]
of $\overline{G}{}^{\mr{fr},[\ell]}_g$-representations which is an isomorphism in a stable range. By \cref{lem:realisation-h-com} 
\[K^\vee \otimes^{\cat{d(s)Br}} H^\mr{Com}_p(Z_n)_q \cong H^\mr{Com}_p(K^\vee \otimes^{\cat{d(s)Br}} Z_n)_q\]
as long as $\tfrac{q}{n} \leq \tfrac{1}{3} g$, which is satisfied by taking $g$ large enough. By the vanishing result established above, for all large enough $g$ and for $q \leq \smash{\tfrac{n(n-3)}{2}}$, this is a trivial representation whenever $q \neq n\cdot p$. By the discussion in \cref{sec:orth-symp-rep-theory}, the same is true as $\mr{O}_\epsilon(H(g))$-representations.

By Lemma \ref{lem:partial-inverse-to-realisation}, if $H^\mr{Com}_p(Z_n)_q(S)$ were not zero for a non-empty finite set $S$ then for $g \geq |S|$ so would be $[H_{[S]} \otimes (K^\vee \otimes^{\cat{d(s)Br}} H^\mr{Com}_p(Z_n)_q]^{\mr{O}_\epsilon(H(g))}$. As $H_{[S]}$ for a non-empty set $S$ is a direct sum of irreducibles that are \emph{not} equal to the trivial representation, this implies that $K^\vee \otimes^{\cat{d(s)Br}} H^\mr{Com}_p(Z_n)_q$ would contain a non-trivial subrepresentation. Using the above, we conclude that
\[H^\mr{Com}_{p}(Z_n)_{q}(S) = 0 \text{ for } S \neq \varnothing \text{ and } q \neq n \cdot p \text{ as long as } q \leq \tfrac{n(n-3)}{2}.\]
We now use that for all $n$'s of the same parity the commutative algebra objects $Z_n$ are isomorphic up to a linear scaling of degrees. As the bound ``$q \leq\tfrac{n(n-3)}{2}$" scales quadratically with $n$ it can therefore be ignored, yielding Theorem \ref{thm:KoszulPositive} for $Z_n$.

\vspace{1em}

To deduce Theorem \ref{thm:KoszulPositive} for $E_n$, we invoke \cref{cor:EnZnEquiKoszul}.

\subsection{$E_n$ and high-dimensional manifolds with Euler structure}\label{sec:euler} Instead of using \cref{cor:EnZnEquiKoszul} to deduce \cref{thm:KoszulPositive} for $E_n$ from $Z_n$, one may prove it directly using a custom tangential structure which mimics in high dimensions the moduli spaces of surfaces. We explain this in this extended remark, as it may be of independent interest. It was in fact our original approach to \cref{athm:main}.

The Pontrjagin classes $p_j \in H^{4j}(B\mr{SO}(2n);\bZ)$ for $1 \leq j < n$ induce a map
\[B\mr{SO}(2n) \lra \prod_{1 \leq j<n} K(\bZ,4j).\]
We denote its homotopy fibre by $\widetilde{B}^e$, which has $n$-connected cover $B^e \coloneqq \widetilde{B}^e\langle n \rangle$.

\begin{definition}The \emph{Euler tangential structure} is the map $\theta_e \colon B^e \to B\mr{O}(2n)$.
\end{definition}

The name is justified by the calculation $H^*(B^e;\bQ) = \bQ[e]$ where $e$ denotes the Euler class, of degree $2n$. By construction $B^e$ is $n$-connected, and since $B\mr{SO}(2n)$ has finitely-generated homotopy groups, so does $B^e$. Thus this tangential structure satisfies the hypotheses for Section 8 of \cite{KR-WAlg}. From \cref{sec:torelli-groups}, we obtain $B\Tor^{\theta_e}_\partial(W_{g,1})$ which we shorten to $B\Tor^e_\partial(W_{g,1})$. A straightforward adaptation of the techniques of \cite{KR-WTorelli,KR-WDisks} yields a homotopy-commutative diagram
\begin{equation*}
\begin{tikzcd}[column sep=.3in,row sep=.15in]X_1^e{}(g) \rar \dar & \overline{B\Tor}^e_\partial(W_{g,1})_\ell \rar \dar & X_0^e \dar \\ 
		A_1^e(g) \rar \dar & \overline{B\Diff}^e_\partial(W_{g,1})_\ell \rar{\overline{\alpha}^e} \dar  & \overline{\Omega^\infty_0 \mathbf{MT}\theta_e} \dar \\
		A_2^e(g) \rar & B\overline{G}{}^{e,[\ell]}_g \rar & \overline{\overline{\Omega^\infty_0 \mathbf{KH}}} 
\end{tikzcd}
\end{equation*}
with rows and columns fibration sequences, and an analogue of \cref{thm:pos-arity-framing-input} sufficient to directly prove \cref{thm:KoszulPositive} for $E_n$:

\begin{theorem}\label{thm:pos-arity-euler-input} \
	\begin{enumerate}[(i)]
		\item The space $X_1^e(g)$ is nilpotent and of finite type.
		\item \label{enum:pos-arity-euler-input-action} The action of $\pi_1(A^e_2(g))$ on $H^*(X^e_1(g);\bQ)$ and $\pi_*(X^e_1(g)) \otimes \bQ$ factors over $\overline{G}{}^{e,[\ell]}_g$.
		\item \label{enum:pos-arity-euler-input-cohomology} There is a map
		\[\frac{K^\vee \otimes^{\cat{d(s)Br}} E_n}{(\kappa_{e^{j}} \mid j > 1)} \lra H^*(X^e_1(g);\bQ)\]
		of $\overline{G}{}^{e,[\ell]}_g$-representations which is an isomorphism in a range of homological degrees tending to $\infty$ as $g \to \infty$.
		\item \label{enum:pos-arity-euler-input-homotopy} In a stable range, in degrees $* \leq \tfrac{n(n-3)}{2}$ the non-trivial $\overline{G}{}^{e,[\ell]}_g$-representations in $\pi_{*+1}(X^e_1(g)) \otimes \bQ$ vanish except with $* = r(n-1)$ for some $r \geq 1$.\qed
	\end{enumerate}
\end{theorem}

\section{Graph complexes: finishing the proof of \cref{athm:main}} Our strategy to complete the proof of \cref{thm:mainPrime} is as follows. Firstly, we will establish graph complex models $RB^{Z_n}_\mr{conn}$ and $RB^{E_n}_\mr{conn}$ for $\bL Q^\mr{Com}(Z_n)$ and $\bL Q^\mr{Com}(E_n)$, in terms of graphs with red and black edges whose black subgraph is connected, equipped with a certain differential. Once these models have been established, there are several ways in which one can proceed: we have chosen to proceed in the most self-contained way, but see \cref{sec:alternative-graph-proof} for alternatives using vanishing results for graph complexes from the literature. In this section we will disregard the weight grading on $E_n$ and $Z_n$.

\subsection{Red and black graphs}

We will first describe explicit cofibrant replacements for the augmented unital commutative algebra objects $Z_n, E_n \in \cat{Alg}_{\cat{Com}^+}^\mr{augm}(\cat{Fun}(\cat{d(s)Br}, \cat{Ch}))$.

\begin{definition}\
\begin{enumerate}[(i)]
		\item A \emph{graph} $\Gamma$ is a tuple $(V, H, \iota, \partial)$ of a set $V$ of vertices, a set $H$ of half-edges, a function $\partial\colon H \to V$, and an involution $\iota$ on $H$. The \emph{edges} $E$ of $\Gamma$ is the set of free orbits of $\iota$, and the \emph{legs} $L$ is the set of trivial orbits. For a graph $\Gamma$ we will write $V(\Gamma), H(\Gamma), E(\Gamma), L(\Gamma), \iota_\Gamma$, and $\partial_\Gamma$ for its associated data. 
		\item An \emph{isomorphism of graphs} $\Gamma \leadsto \Gamma'$ is a pair of bijections between their sets of vertices and half-edges which intertwine the $\iota$'s and $\partial$'s.
\end{enumerate}
\end{definition}

\begin{definition}If $e = \{h_1, h_2\} \in E(\Gamma)$ is an edge, the graph $\Gamma' = \Gamma/e$ \emph{obtained by contracting the edge $e$} has vertices $V(\Gamma') \coloneqq V(\Gamma)/(\partial_\Gamma(h_1) \sim \partial_\Gamma(h_2))$, half-edges $H(\Gamma') \coloneqq H(\Gamma) \setminus \{h_1, h_2\}$, boundary map
	\[\partial_{\Gamma'}\colon H(\Gamma) \setminus \{h_1, h_2\} \overset{\mr{inc}}\lra H(\Gamma) \overset{\partial_{\Gamma}}\lra V(\Gamma) \xrightarrow{\mr{quot}} V(\Gamma)/(\partial_\Gamma(h_1) \sim \partial_\Gamma(h_2)),\]
	and involution $\iota_{\Gamma'} = \iota_{\Gamma}\vert_{H(\Gamma')}$.
\end{definition}

\begin{definition}\
	\begin{enumerate}[(i)]
		\item A \emph{weighted red and black graph} $\Gamma$ is a graph $(V, H, \iota, \partial)$ along with a function $c = c_\Gamma\colon E(\Gamma) \to \{R,B\}$ colouring each edge either red or black, and a function $w\colon V(\Gamma) \to \bN$ assigning to each vertex a \emph{weight}\footnote{This is unrelated to the ``weight grading" discussed earlier: no confusion should arise, as we do not refer to the weight grading in this section.}.
			\item We write $E_R(\Gamma) \coloneqq c_\Gamma^{-1}(R)$ and $E_B(\Gamma) \coloneqq c_\Gamma^{-1}(B)$ for the sets of \emph{red edges} and \emph{black edges} respectively. The \emph{black subgraph} of $\Gamma$ is the subgraph consisting of all vertices and the black edges. 
		\item An \emph{isomorphism of weighted red and black graph}s $\Gamma \leadsto \Gamma'$ is an isomorphism of graphs which commutes with the colour and weight functions. 
		\item A weighted red and black graph is \emph{admissible} if for each $v \in V(\Gamma)$ we have $2 w(v) + \# \partial^{-1}(v) \geq 3$.
	\end{enumerate}
\end{definition}

\begin{definition}
	If $e \in E_B(\Gamma)$, we let $\Gamma/e$ be the weighted red and black graph obtained by contracting the edge $e = \{h_1, h_2\}$, taking the induced colouring, and:
\begin{enumerate}[(i)]	
\item if $\partial (h_1) \neq \partial (h_2)$ then giving the new vertex weight $w(\partial (h_1)) + w(\partial (h_2))$;

\item if $\partial (h_1) = \partial (h_2)$ then giving the new vertex weight $w(\partial(h_1))+1$.
\end{enumerate}
Similarly let $\Gamma{\setminus}e$ be the red and black graph obtained by turning the black edge $e$ red.
\end{definition}

Note that if $\Gamma$ is admissible, so are $\Gamma/e$ and $\Gamma \backslash e$. 
Recall that $\det(V) \coloneqq \Lambda^{\dim(V)}(V)$. Define 1-dimensional vector spaces
\begin{equation}\label{eqn:orientations}\mathfrak{K}_B(\Gamma) \coloneqq \det(\bQ^{E_B(\Gamma)}) \quad\quad \text{ and } \quad\quad \mathfrak{L}(\Gamma) \coloneqq \det(\bQ^{L(\Gamma)}).\end{equation}
Isomorphisms of weighted red and black graphs induce isomorphisms of these vector spaces.

If $e \in E_B(\Gamma)$ there are maps $-/e\colon \mathfrak{K}_B(\Gamma) \to \mathfrak{K}_B(\Gamma/e)$ and $- {\setminus} e\colon \mathfrak{K}_B(\Gamma) \to \mathfrak{K}_B(\Gamma {\setminus} e)$ given by contracting with $e$. As contracting or colouring an edge does not change the set of legs, there are also canonical identifications $\mathfrak{L}(\Gamma) \overset{\sim}\to \mathfrak{L}(\Gamma/e)$ and $\mathfrak{L}(\Gamma) \overset{\sim}\to \mathfrak{L}(\Gamma {\setminus} e)$.

\begin{definition}\label{def:RBEn} For an admissible weighted red and black graph $\Gamma$ we set
	\[\deg(\Gamma) \coloneqq n\left(\underset{v \in V(\Gamma)}{\textstyle \sum}(2 w(v){+}\mr{val}(v){-}2)\right) + \#E_B(\Gamma)\]
and define a graded vector space
\[RB^{E_n}(S) \coloneqq \left(\bigoplus_{\substack{\Gamma=(V, F, \iota, \partial, c, w)\\\ell\colon L(\Gamma) \overset{\sim}\to S}} \mathfrak{K}_B(\Gamma) \otimes \mathfrak{L}(\Gamma)^{\otimes n} [\deg(\Gamma)] \right)/\sim\]
where the sum is over all admissible weighted red and black graphs with legs identified with $S$, and $\sim$ is the equivalence relation induced by isomorphism of such graphs. 

We define a differential by summing over all ways of contracting a black edge or turning a black edge red, with certain signs. More formally, we define $d=d_\text{con}+d_\text{col}$ on this graded vector space by 
\begin{align*}
d_\text{con}(\Gamma, \ell, \nu, \omega^{\otimes n}) &= \sum_{e \in E_B(\Gamma)}  (\Gamma/e, \ell, \nu/e, \omega^{\otimes n})\\
d_\text{col}(\Gamma, \ell, \nu, \omega^{\otimes n}) &= -\sum_{e \in E_B(\Gamma)}  (\Gamma {\setminus} e, \ell, \nu {\setminus} e, \omega^{\otimes n}).
\end{align*}

We consider these complexes as defining a functor $RB^{E_n}\colon \cat{d(s)Br} \to \cat{Ch}$, by letting the morphism $[(f, m_S)] \in \cat{d(s)Br}(S, T)$ induce the map $RB^{E_n}(S) \to RB^{E_n}(T)$ given by creating (oriented) red edges according to the (ordered) matching $m_S$, and then relabelling the remaining legs according to $f$. (This is well-defined as a functor on $\mathsf{d(s)Br}$ because of the twisting by $\mathfrak{L}(\Gamma)^{\otimes n}$.) Disjoint union makes $RB^{E_n}$ into a commutative algebra object.
\end{definition}

\begin{figure}
	\begin{tikzpicture}
	\begin{scope}
		\node at (-1,0) {$\bullet$};
		\node at (1,0) {$\bullet$};
		\draw (1,0) -- (1.2,.3);
		\node at (1.3,.5) {$1$};
		\draw (1,0) -- (1.2,-.3);
		\node at (1.3,-.5) {$2$};
		\node at (-1.3,0) {$g_\alpha$};
		\node at (1.4,0) {$g_\beta$};
		\draw (-1.5,0) circle (.5cm);
		\draw [red] (-1,0) to[out=50,in=130,looseness=1.5] (1,0);
		\draw (-1,0) to[out=-50,in=-130,looseness=1.5] (1,0);
		\draw (-1,0) -- (1,0);
	\end{scope}
	
	\draw [|->] (-1,-1) -- (-2,-2);
	\node at (-1.5,-1.4) [left] {$d_\mr{col}$}; 
	\draw [|->] (1,-1) -- (2,-2);
	\node at (1.5,-1.4) [right] {$d_\mr{con}$}; 
	
	\begin{scope}[xshift=-6cm,yshift=-3cm]
		\node at (-1,0) {$\bullet$};
		\node at (1,0) {$\bullet$};
		\draw (1,0) -- (1.2,.3);
		\node at (1.3,.5) {$1$};
		\draw (1,0) -- (1.2,-.3);
		\node at (1.3,-.5) {$2$};
		\node at (-1.3,0) {$g_\alpha$};
		\node at (1.4,0) {$g_\beta$};
		\draw (-1.5,0) circle (.5cm);
		\draw [red] (-1,0) to[out=50,in=130,looseness=1.5] (1,0);
		\draw (-1,0) to[out=-50,in=-130,looseness=1.5] (1,0);
		\draw (-1,0) -- (1,0);
	\end{scope}

	\node at (-4,-3) {$+$};

	\begin{scope}[xshift=-1.55cm,yshift=-3cm]
		\node at (-1,0) {$\bullet$};
		\node at (1,0) {$\bullet$};
		\draw (1,0) -- (1.2,.3);
		\node at (1.3,.5) {$1$};
		\draw (1,0) -- (1.2,-.3);
		\node at (1.3,-.5) {$2$};
		\node at (-1.3,0) {$g_\alpha$};
		\node at (1.4,0) {$g_\beta$};
		\draw [red] (-1.5,0) circle (.5cm);
		\draw [red] (-1,0) to[out=50,in=130,looseness=1.5] (1,0);
		\draw (-1,0) to[out=-50,in=-130,looseness=1.5] (1,0);
		\draw (-1,0) -- (1,0);
	\end{scope}
	
	\begin{scope}[xshift=4cm,yshift=-3cm]
		\node at (-1,0) {$\bullet$};
		\node at (1,0) {$\bullet$};
		\draw (1,0) -- (1.2,.3);
		\node at (1.3,.5) {$1$};
		\draw (1,0) -- (1.2,-.3);
		\node at (1.3,-.5) {$2$};
		\node at (-1.1,0) [left] {$g_\alpha{+}1$};
		\node at (1.4,0) {$g_\beta$};
		\draw (-1,0) [red] to[out=50,in=130,looseness=1.5] (1,0);
		\draw (-1,0) to[out=-50,in=-130,looseness=1.5] (1,0);
		\draw (-1,0) -- (1,0);
	\end{scope}
	\end{tikzpicture}
	\caption{A weighted red and black graph with two vertices, of weight $g_\alpha$ and $g_\beta$ respectively, and two legs, labelled by $\ul{2}$. We have also indicated the terms $d_\mr{col}$ and $d_\mr{con}$ of the differential, assuming $n$ is even (if $n$ were odd, then the graph would be zero by symmetry of the legs). In $d_\mr{col}$, the reader maybe expected two terms obtained by turning red one of the two black edges connecting the two vertices, by these yield the same graph with opposite sign and hence cancel. The same happens in $d_\mr{con}$, where two terms obtained by collapsing one of the two black edges connecting the two vertices cancel.}
	\label{fig:rbgraph-example}
\end{figure}
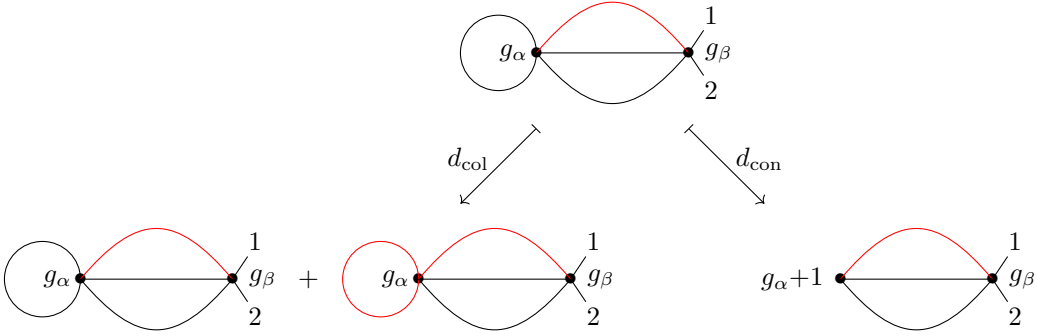

We let $RB^{Z_n}(S)$ be the quotient of $RB^{E_n}(S)$ by those weighted red and black graphs having some vertex of weight $>0$. These assemble into a functor $RB^{Z_n}\colon \cat{d(s)Br} \to \cat{Ch}$ which is a quotient of $RB^{E_n}$, having an induced augmented unital commutative algebra structure. Concretely $RB^{Z_n}(S)$ is given by weighted red and black graphs having all weights zero (equivalently, no weights), and the differential is given by the same formula with the understanding that $\Gamma/e$ is zero when $e$ is a loop (as then $\Gamma/e$ has a vertex of weight $>0$).

\begin{proposition}\label{prop:RBIsModel}
There are weak equivalences of unital commutative algebras
\[RB^{Z_n} \overset{\sim}\lra Z_n\quad\quad\text{ and }\quad\quad RB^{E_n} \overset{\sim}\lra E_n,\]
and $RB^{Z_n}$ and $RB^{E_n}$ are cofibrant in $\cat{Alg}_{\cat{Com}^+}^\mr{augm}(\cat{Fun}(\cat{d(s)Br}, \cat{Ch}))$. 
\end{proposition}
\begin{proof}
Let $\phi\colon RB^{E_n} \to E_n$ be the map which
\begin{enumerate}[(i)]
\item kills weighted red and black graphs having any black edges, and

\item assigns to a graph having only red edges the induced partition of its set of legs given by the connected components of the graph, with the part corresponding to the connected subgraph $\Gamma_\alpha$ assigned weight $\left(\sum_{v \in V(\Gamma_\alpha)} w(v)\right) + 1 - \chi(\Gamma_\alpha)$, i.e.\ the sum of the weights of its vertices and its genus.
\end{enumerate}
This is a map of commutative algebra objects. It is also---when evaluated on any $S$---a map of chain complexes, as follows. If $(\Gamma, \ell, \omega^{\otimes n}, \nu)$ is such that $\Gamma$ has at least two black edges then every term of $d(\Gamma, \ell, \omega^{\otimes n}, \nu)$ has a black edge and so vanishes under $\phi$; $(\Gamma, \ell, \omega^{\otimes n}, \nu)$ does too. If $(\Gamma, \ell, \omega^{\otimes n}, \nu)$ is such that $\Gamma$ has a single black edge $e$ then 
\[d(\Gamma, \ell, \nu, \omega^{\otimes n}) = (\Gamma/e, \ell, 1, \omega^{\otimes n}) - (\Gamma {\setminus} e, \ell, 1,\omega^{\otimes n}),\]
and $\Gamma/e$ and $\Gamma{\setminus} e$ induce the same partition on $S$ with the same weights so this vanishes on applying $\phi$; $(\Gamma, \ell, \omega^{\otimes n}, \nu)$ does too.

To check that $\phi$ is a weak equivalence, we consider the filtration given by letting $F^p RB^{E_n}(S)$ be spanned by those graphs with $\leq p$ edges. As $d_\text{con}$ reduces the number of edges and $d_\text{col}$ preserves the number of edges, we have
\[\mr{Gr} (RB^{E_n}(S), d=d_\text{con}+d_\text{col}) \cong (RB^{E_n}(S), d_\text{col}).\]
This complex splits over isomorphism classes of weighted graphs $\Gamma$ as a sum of chain complexes $\mr{Colour}_*(\Gamma)_{\mr{Aut}(\Gamma)}$, where $\mr{Colour}_*(\Gamma)$ has basis the set of subgraphs of $\Gamma$, considered as those edges to be coloured red: such a subgraph has degree given by $n\left(\sum_{v \in V(\Gamma)}(2w(v)+\mr{val}(v)-2)\right) + \# E_B(\Gamma)$, and the differential sums over all ways of adding a single edge to the red subgraph. In other words, up to a shift of degrees $\mr{Colour}_*(\Gamma)$ is the \emph{reduced} simplicial chain complex of the simplex $\Delta^{E(\Gamma)}$ on the set of edges of $\Gamma$. It is therefore acyclic if $\Gamma$ contains any edges, and has homology $\bQ[n\sum_{v \in V(\Gamma)}(2w(v)+\mr{val}(v)-2)]$ otherwise. As it is supported on graphs with no edges, the differential $d_\text{con}$ has no effect and so the spectral sequence for this filtration collapses, to give
\[H_*(RB^{E_n}(S), d) \cong \bigoplus_{\substack{\text{weighted graphs } \Gamma \\ \text{ with no edges}}} \bQ[n\sum_{v \in V(\Gamma)}(2w(v)+\mr{val}(v)-2)],\]
and with this description we observe that $\phi_*\colon H_*(RB^{E_n}(S), d) \to E_n$ is an isomorphism as required. The argument for $Z_n$ is completely parallel.



Neglecting for a moment the differential, let $X \subset RB^{E_n}$ be the subobject spanned by those weighted red and black graphs whose black subgraphs are connected (we take this to also mean non-empty). The commutative unital algebra structure on $RB^{E_n}$ induces a map
\begin{equation}\label{eq:BRFreeOnConn}
F^{\mr{Com}^+}(X) = \bigoplus_{k \geq 0} (X^{\otimes k})_{\fS_k} \lra RB^{E_n}
\end{equation}
which we claim is an isomorphism. By definition of Day convolution, we have
\[(X^{\otimes k})(S) = \colim_{\substack{f: S_1 \sqcup \cdots \sqcup S_k \to S \\ \in \mathsf{d(s)Br}}} X(S_1) \otimes \cdots \otimes X(S_k).\]
Spelling this out, it is the space of weighted red and black graphs with legs $S$ whose black subgraph has precisely $k$ components, and the components are ordered; $(X^{\otimes k})_{\fS_k}(S)$ is then the same without the ordering of components. We therefore see that \eqref{eq:BRFreeOnConn} is indeed an isomorphism.

Reincorporating the differential, we have just shown that $RB^{E_n}$ is quasi-free, and it is also non-negatively graded, so a standard induction over skeleta of $X$ shows that it is cofibrant, cf.\ \cite[Proposition B.6.6]{LodayVallette}. The same goes for $RB^{Z_n}$.
\end{proof}

Let $RB^{E_n}_\mr{conn}$ be the quotient of $RB^{E_n}$ by the unit and those weighted red and black graphs whose black subgraph is not connected, and $RB^{Z_n}_\mr{conn}$ be the analogous quotient of $RB^{Z_n}$. As the red and black graphs whose black subgraph is not connected are precisely the decomposables, we find:

\begin{corollary}
We have $\bL Q^\mr{Com}(Z_n) \simeq RB^{Z_n}_\mr{conn}$ and $\bL Q^\mr{Com}(E_n) \simeq RB^{E_n}_\mr{conn}$.\qed
\end{corollary}

Following Section \ref{sec:Gradings}, as $Z_n$ and $E_n$ have trivial differential the homology groups $\mr{AQ}_*(Z_n)$ and $\mr{AQ}_*(E_n)$ are equipped with an additional internal grading. This may be implemented in terms of the resolutions $RB^{Z_n}$ and $RB^{E_n}$, and so in the models $RB^{Z_n}_\mr{conn}$ and $RB^{E_n}_\mr{conn}$, by giving a (weighted) red and black graph $\Gamma$ \emph{internal degree}
\[\mr{deg_{int}}(\Gamma) \coloneqq n\left(\sum_{v \in V(\Gamma)}(2 w(v)+\mr{val}(v)-2)\right).\]
It is indeed preserved by the differential, and by the quasi-isomorphisms to $Z_n$ and $E_n$. As the total degree of $\Gamma$ is
$ n\left(\sum_{v \in V(\Gamma)}(2 w(v)+\mr{val}(v)-2)\right) + \#E_B(\Gamma)$, its \emph{Harrison degree} is then $\# E_B(\Gamma) + 1$.

\subsection{Black graphs}

\begin{definition}\label{def:Gn}
Let $G^{E_n}(S)$ be the quotient of $RB^{E_n}_\mr{conn}(S)$ by  the subcomplex spanned by those red and black graphs having a nonzero number of red edges, and $G^{Z_n}(S)$ be the analogous quotient of $RB^{Z_n}_\mr{conn}(S)$.
\end{definition}

By the discussion above these complexes inherit an internal grading, giving homology groups $H_{p+q}(G^{E_n}(S))_q$ and $H_{p+q}(G^{Z_n}(S))_q$ in total degree $p+q$ and internal degree $q$. Moreover, $G^{E_n}(S)$ and $G^{Z_n}(S)$ are functorial with respect to bijections of $S$, and may be organised into objects $G^{E_n}$ and $G^{Z_n}$ of $\cat{Fun}(\cat{FB},\cat{Ch})$. 
In Section \ref{sec:ClassicalGraphCxes} we explain how these are related to other graph complexes in the literature.

Neglecting the differential, every red and black graph whose black subgraph is connected can, tautologically, be obtained from a connected black graph by attaching some red edges. This gives the decompositions
\begin{equation}\label{eq:RBdecomp}
\begin{aligned}
RB^{Z_n}_\mr{conn}(-) &\cong \bigoplus_{n \geq 0} G^{Z_n}(\ul{n}) \otimes_{\fS_n} \cat{d(s)Br}(\ul{n}, -)\\
RB^{E_n}_\mr{conn}(-) &\cong \bigoplus_{n \geq 0} G^{E_n}(\ul{n}) \otimes_{\fS_n} \cat{d(s)Br}(\ul{n}, -)
\end{aligned}
\end{equation}
as functors to graded vector spaces, where we recall $\ul{n} \coloneqq \{1,\ldots,n\}$. 


The inclusion $i \colon \mathsf{FB} \to \mathsf{d(s)Br}$ of the ($\bQ$-linearised) category of finite sets and bijections has a retraction $r \colon \mathsf{d(s)Br} \to \mathsf{FB}$, which is the identity on objects, and on morphisms $r \colon  \mathsf{d(s)Br}(A,B) \to \mathsf{FB}(A,B)$ is the identity if $|A|=|B|$, and is the zero map otherwise. Left Kan extension gives a functor
\[r_* \colon \mathsf{Fun}(\mathsf{d(s)Br}, \mathsf{Ch}) \lra \mathsf{Fun}(\mathsf{FB}, \mathsf{Ch}).\]
In \cref{lem:functor-model-cat} we discussed the model structure on the domain, and the codomain has a similar model structure, such that $r_* \dashv r^*$ is a Quillen adjunction, and in particular $r_*$ admits a total left derived functor $\bL r_*$. Unwinding definitions shows that $G^{E_n} = r_*(RB_{conn}^{E_n})$ and $G^{Z_n} = r_*(RB_{conn}^{Z_n})$, and we first argue that these are in fact also derived left Kan extensions.

\begin{lemma}
The natural maps
\begin{align*}
\bL r_* (RB_{conn}^{E_n}) &\lra r_* (RB_{conn}^{E_n}) = G^{E_n},\\
\bL r_* (RB_{conn}^{Z_n}) &\lra r_* (RB_{conn}^{Z_n}) = G^{Z_n}
\end{align*}
are equivalences.
\end{lemma}
\begin{proof}We consider the case of $E_n$; that of $Z_n$ is identical. Filtering $RB^{E_n}_{conn}(-)$ by its grading, we obtain compatible filtrations of the two sides. This gives a map of spectral sequences, which are strongly convergent as the values of ${RB}_{conn}^{E_n}$ are non-negatively graded.  As $r_*$ preserves exact sequences 
the map of $E^1$-pages is
\[\bL r_* (\mr{Gr}(RB_{conn}^{E_n}))(-) \lra r_* (\mr{Gr}(RB_{conn}^{E_n}))(-).\]
Thus it suffices to show that this is an equivalence. By \eqref{eq:RBdecomp} we have $\mr{Gr}(RB_{conn}^{E_n})(-) = i_*(G^{E_n})(-)$, where the latter has trivial differential. As we are working over a field of characteristic zero all objects of $\mathsf{Fun}(\mathsf{FB}, \mathsf{Ch})$ are cofibrant, so $\bL i_*(G^{E_n})(-) \overset{\sim}\to i_*(G^{E_n})(-)$. Using that $r_* \circ i_* = (r \circ i)_* = \mr{id}$ and similarly $\bL r_* \circ \bL i_* \simeq \mr{id}$, the claim follows.
\end{proof}

Let us analyse $\bL r_*(A)$ for $A \colon \mathsf{d(s)Br} \to \mathsf{Ch}$ using the colimit description of Kan extension and the bar construction 
model for homotopy colimits. 
 This is the simplicial object with $s$-simplices
\[[s] \longmapsto \bigoplus_{S_0, \ldots, S_s \in \mathsf{d(s)Br}} A(S_0) \otimes \mathsf{d(s)Br}(S_0, S_1) \otimes \cdots \otimes \mathsf{d(s)Br}(S_{s-1}, S_s) \otimes \mathsf{FB}(r(S_s), -)\]
(implicitly replacing $\mathsf{d(s)Br}$ by a skeletal subcategory), and face maps given by composing morphisms in $\mathsf{d(s)Br}$, acting with them on $A(-)$, or applying $r \colon \mathsf{d(s)Br} \to \mathsf{FB}$ to them and composing in $\mathsf{FB}$; its geometric realisation models $\bL r_*(A)$. Filtering this geometric realisation by skeleta gives a spectral sequence with
\[E^1_{s,t} = \bigoplus_{S_0, \ldots, S_s \in \mathsf{d(s)Br}} H_{t}(A)(S_0) \otimes  \mathsf{d(s)Br}(S_0, S_1) \otimes \cdots \otimes \mathsf{d(s)Br}(S_{s-1}, S_s) \otimes \mathsf{FB}(r(S_s), -),\]
converging strongly to $H_{s+t}(\bL r_*(A))$ as the values of $A$ are by definition non-negatively graded. The $d^1$-differential is given by the alternating sum of the face maps. 

If $A$ also has an additional grading, then this gets carried along too. Applied to $A \coloneqq RB_{conn}^{E_n}$ with its additional internal grading this gives a spectral sequence
\begin{equation}\label{eq:KanExtSS}
\begin{tikzcd}
\bigoplus\limits_{\mathclap{\substack{S_0, \ldots, S_s \\ \in \mathsf{d(s)Br}}}} H_{t-q+1}^\mr{Com}(E_n)_q(S_0) \otimes  \mathsf{d(s)Br}(S_0, S_1) \otimes \cdots \otimes \mathsf{d(s)Br}(S_{s-1}, S_s) \otimes \mathsf{FB}(r(S_s), T) \arrow[Rightarrow,start anchor={[yshift=15pt]}]{d} \\[-5pt]
	 H_{s+t}(G^{E_n}(T))_q,
\end{tikzcd}
\end{equation}
where the top expression is $E^1_{s,t,q}$, and similarly for $RB_{conn}^{Z_n}$.

\subsection{Finishing the proof of \cref{athm:main}}

\begin{lemma}\label{lem:BlackGphVanishing}
$H_{p+q}(G^{Z_n}(T))_q$ and $H_{p+q}(G^{E_n}(T))_q$ vanish if $n \cdot (p+1) < q$ and $T \neq \varnothing$.
\end{lemma}
\begin{proof}We consider the case of $E_n$; that of $Z_n$ is identical. We apply the spectral sequence \eqref{eq:KanExtSS} for a fixed $q$ and evaluated at $T \neq \varnothing$, which starts from
\[E^1_{s,t,q} = \bigoplus\limits_{\mathclap{\substack{S_0, ..., S_s \\ \in \mathsf{d(s)Br}}}} H_{t-q+1}^\mr{Com}(E_n)_q(S_0) \otimes  \mathsf{d(s)Br}(S_0, S_1) \otimes \cdots \otimes \mathsf{d(s)Br}(S_{s-1}, S_s) \otimes \mathsf{FB}(r(S_s), T)\]
and converges to $H_{s+t}(G^{E_n}(T))$. For the latter tensor factors to be non-zero we must have $|S_0| \geq |S_1| \geq \cdots \geq |S_s|=|T|$, and so if $T \neq \varnothing$ then $S_0 \neq \varnothing$. In this case, for the first tensor factor to be non-zero we must have $q = n \cdot (t-q+1)$ by Theorem \ref{thm:KoszulPositive}. Thus if $q \neq n \cdot (t-q+1)$ then $E^1_{s,t,q}=0$.

In particular if $t < \tfrac{n+1}{n}q-1$ then $E^1_{s,t,q}=0$, so $H_{s+t}(G^{E_n}(T))_q=0$ if $s+t < \tfrac{n+1}{n}q-1$. Writing $s+t = p+q$ and rearranging, we see that $H_{p+q}(G^{E_n}(T))_q=0$ for $n \cdot(p+1)< q$ as claimed.
\end{proof}

The following lemma is due to Turchin--Willwacher \cite[Section 6.1]{TW} for $Z_n$ and Chan--Galatius--Payne \cite[Theorem 1.7]{CGP2} for $E_n$. As its proof is quite elementary we give it, following \cite[Section 5.3]{CGP2}.

\begin{lemma}[Turchin--Willwacher, Chan--Galatius--Payne] \label{lem:inj-empty-from-1}
There are injections
\begin{align*}
H_{p+q}(G^{Z_n}(\varnothing))_q &\lra H_{p+q+n}(G^{Z_n}(\ul{1}))_{q+n},\\
H_{p+q}(G^{E_n}(\varnothing))_q &\lra H_{p+q+n}(G^{E_n}(\ul{1}))_{q+n}.
\end{align*}
\end{lemma}

\begin{proof}We explain the proof for $G^{E_n}$; the proof for $G^{Z_n}$ is identical without weights.
	
For $q \leq 0$ there is nothing to prove: the admissibility condition for graphs in $G^{E_n}(\varnothing)$ implies that $2 w(v)+\mr{val}(v) -2 \geq 1$ for all vertices $v$, so the internal grading satisfies $q = n\sum_{v \in V(\Gamma)} (2w(v)+\mr{val}(v)-2) > 0$. 

We will define maps of chain complexes
\[G^{E_n}_*(\varnothing) \overset{t}\lra G^{E_n}_{*+n}(\ul{1}) \overset{\pi}\lra G^{E_n}_*(\varnothing)\]
whose composition preserves the internal grading and on the summand with internal grading $q$ is given by multiplication by $q$. In particular it is an isomorphism for $q>0$. Informally, $t$ sums over all ways to add a single leg and $\pi$ deletes the leg.

Suppose we are given a weighted black graph $(\Gamma,\ell,\nu,\omega^{\otimes n}) \in G^{E_n}(\varnothing)$; $\Gamma$ is the graph with $w$ the weights of its vertices, $\ell$ the labelling of its legs, and $\nu$ and $\omega^{\otimes n}$ are orientations in $\mathfrak{K}_B(\Gamma)$ and $\mathfrak{L}(\Gamma)^{\otimes n}$ respectively, as in \eqref{eqn:orientations}. For each vertex $v \in V(\Gamma)$ adding a leg at $v$ yields a graph $\Gamma_v$ with $V(\Gamma_v) = V(\Gamma)$, $E_B(\Gamma_v) = E_B(\Gamma)$, and $L(\Gamma_v) = \{l\}$. Using the same $w$ and $\nu$, $\ell_v \colon \{l\} \to \{1\}$ the unique bijection, and $\omega_v$ induced from $\omega$ by $- \wedge l \colon \det(\bQ^{L(\Gamma)}) \overset{\sim}\to \det(\bQ^{L(\Gamma_v)})$, 
we obtain $(\Gamma_v,\ell_v,\nu,\omega_v^{\otimes n}) \in G^{E_n}(\ul{1})$. Since the valence of $v$ has increased by $1$, $\mr{deg}_\mr{int}(\Gamma_v) = \mr{deg}_\mr{int}(\Gamma)+n$ and the same is true for the total degree since the number of black edges did not change. Introducing the notation $\chi(v) \coloneqq n(2w(v)+\mr{val}(v)-2)$, then $t$ is defined by
\begin{align*} t \colon G^{E_n}_{p+q}(\varnothing) & \lra G^{E_n}_{p+q+n}(\ul{1}) \\
	(\Gamma,\ell,\nu,\omega^{\otimes n}) &\longmapsto \sum_{v \in V(\Gamma)} \chi(v) \cdot (\Gamma_v,,\ell_v,\nu,\omega_v^{\otimes n}).\end{align*}
To verify $t$ is a chain map, we need to show it commutes with $d_\mr{con}$ (since $d_\mr{col}$ vanishes upon taking the quotient by graphs with at least one red edge). This follows by observing that if an edge $e$ connecting vertices $v'$ and $v''$ (these can be equal, when $e$ is a loop) is collapsed to a vertex $v$, then $(\Gamma/e)_v = \Gamma_{v'}/e = \Gamma_{v''}/e$ and $\chi(v) = \chi(v')+\chi(v'')$.

Suppose we are given a weighted black graph $(\Gamma,\ell,\nu,\omega^{\otimes n}) \in G^{E_n}(\ul{1})$. Removing the unique leg $l$ from $\Gamma$ yields a graph $\Gamma^l$ with $V(\Gamma^l) = V(\Gamma)$, $E_B(\Gamma^l) = E_B(\Gamma)$, and $L(\Gamma^l) = \varnothing$. Using the same $w$ and $\nu$, $\ell^l$ the unique map between empty sets, and $\omega^l$ induced from $\omega$ by the map $\det(\bQ^{L(\Gamma)}) \smash{\overset{\sim}\to} \det(\bQ^{L(\Gamma^l)})$ that contracts with $l$,
we obtain $(\Gamma^l,\ell^l,\nu,(\omega^l)^{\otimes n}) \in G^{E_n}(\varnothing)$. Since the valence of $v_l \coloneqq \partial(l)$ has decreased by $1$, $\mr{deg}_\mr{int}(\Gamma^l) = \mr{deg}_\mr{int}(\Gamma)-n$ and the same is true for the total degree since the number of black edges did not change. This weighted graph is admissible if and only if $2 w(v_l) + \mr{val}(v_l) \geq 3$. Then define
\begin{align*} \pi \colon G^{E_n}_{p+q+n}(\ul{1}) & \lra G^{E_n}_{p+q}(\varnothing) \\
	(\Gamma,\ell,\nu,\omega^{\otimes n}) &\longmapsto \begin{cases} (\Gamma^l,\ell^l,\nu,(\omega^l)^{\otimes n}) & \text{if $2 w(v_l) + \mr{val}(v_l) \geq 3$,} \\
	0 & \text{otherwise.}\end{cases}\end{align*}
To verify that $\pi$ is a chain map, we observe that $(\Gamma/e)^l = \Gamma^l/e$. A subtlety is that deleting a leg may make a graph inadmissible; this occurs either when (a) the leg $l$ is at a trivalent vertex of weight $0$, or (b) the leg $l$ is at a univalent vertex with weight 1. The case (a) is verified in \cref{fig:inadm}, and case (b) is left to the reader.

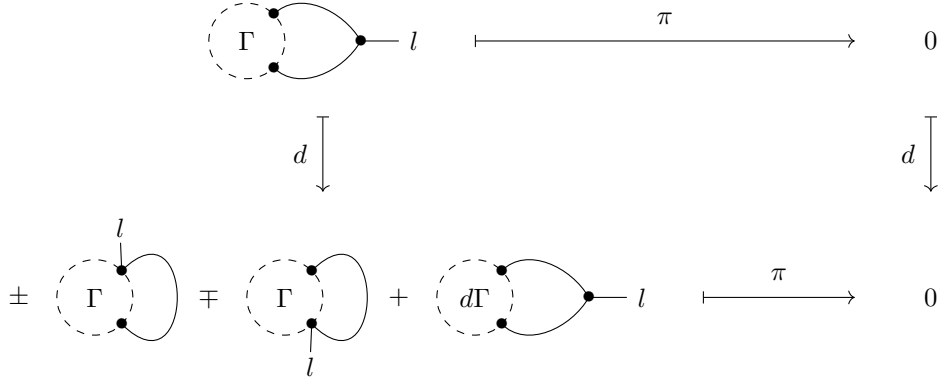
\begin{figure}
	\begin{tikzpicture}
	\begin{scope}
		\draw [dashed] (0,0) circle (.5cm);
		\node at (0,0) {$\Gamma$};
		\node at (45:.5) {$\bullet$};
		\draw (45:.5) to[out=45,in=120] (1.5,0);
		\node at (-45:.5) {$\bullet$};
		\draw (-45:.5) to[out=-45,in=-120] (1.5,0);
		\node at (1.5,0) {$\bullet$};
		\draw (1.5,0) -- (2,0);
		\node at (2.2,0) {$l$};
	\end{scope}

	\draw [|->] (3,0) -- (8,0);
	\node at (5.5,.3) {$\pi$};
	\node at (9,0) {$0$};
	\draw [|->] (1,-1) -- (1,-2);
	\node at (.7,-1.5) {$d$};
	\draw [|->] (9,-1) -- (9,-2);
	\node at (8.7,-1.5) {$d$};	
	\draw [|->] (6,-3.4) -- (8,-3.4);
	\node at (7,-3.1) {$\pi$};
	\node at (9,-3.4) {$0$};
	
	\begin{scope}[yshift=-3.4cm,xshift=-2cm]
		\node at (-1,0) {$\pm$};
		\draw [dashed] (0,0) circle (.5cm);
		\node at (0,0) {$\Gamma$};
		\node at (45:.5) {$\bullet$};
		\node at (-45:.5) {$\bullet$};
		\draw (45:.5) to[out=45,in=-45,looseness=5] (-45:.5);
		\draw (45:.5) -- (65:.8);
		\node at (65:.8) [yshift=.2cm] {$l$};
	\end{scope}
	
	\begin{scope}[yshift=-3.4cm,xshift=.5cm]
		\node at (-1,0) {$\mp$};
		\draw [dashed] (0,0) circle (.5cm);
		\node at (0,0) {$\Gamma$};
		\node at (45:.5) {$\bullet$};
		\node at (-45:.5) {$\bullet$};
		\draw (45:.5) to[out=45,in=-45,looseness=5] (-45:.5);
		\draw (-45:.5) -- (-65:.8);
		\node at (-65:.8) [yshift=-.2cm] {$l$};
	\end{scope}

	\begin{scope}[xshift=3cm,yshift=-3.4cm]
		\draw [dashed] (0,0) circle (.5cm);
		\node at (-1,0) {$+$};
		\node at (0,0) {$d\Gamma$};
		\node at (45:.5) {$\bullet$};
		\draw (45:.5) to[out=45,in=120] (1.5,0);
		\node at (-45:.5) {$\bullet$};
		\draw (-45:.5) to[out=-45,in=-120] (1.5,0);
		\node at (1.5,0) {$\bullet$};
		\draw (1.5,0) -- (2,0);
		\node at (2.2,0) {$l$};
	\end{scope}
	\end{tikzpicture}
	\caption{This figure illustrates that $\pi$ is a chain map when deleting the unique leg creates inadmissible graphs, in case (a). Weights are not displayed.}
	\label{fig:inadm}
\end{figure}

The composition $\pi \circ t$ is given on the element $(\Gamma,\ell,\nu,\omega^{\otimes n})$ by multiplication by $\sum_{v \in V(\Gamma)} \chi(v) = \mr{deg}_\mr{int}(\Gamma)$; the result follows.
\end{proof}

Combining this with Lemma \ref{lem:BlackGphVanishing} shows that the conclusion of that lemma holds for $T=\varnothing$ too, with a slightly better range:

\begin{corollary}\label{cor:BlackGphVanishingZero}
$H_{p+q}(G^{Z_n}(\varnothing))_q$ and $H_{p+q}(G^{E_n}(\varnothing))_q$ vanish if $n \cdot p < q$.
\end{corollary}

\begin{proof}By \cref{lem:inj-empty-from-1}, $H_{p+q}(G^{Z_n}(\varnothing))_q$ injects into $H_{p+q+n}(G^{Z_n}(\ul{1}))_{q+n}$ for $q \neq 0$. By \cref{lem:BlackGphVanishing} for $T = \ul{1}$, this vanishes as long as $n(p+1)<q+n$. Rearranging gives the result.
\end{proof}

Using this we can reverse the logic of the proof of Lemma \ref{lem:BlackGphVanishing} to access $H_{p}^\mr{Com}(Z_n)_q(\varnothing)$ and $H_{p}^\mr{Com}(E_n)_q(\varnothing)$.

\begin{theorem}\label{thm:koszulity-zn-en}
$H^\mr{Com}_{p}(Z_n)_q(\varnothing)$ and $H^\mr{Com}_{p}(E_n)_q(\varnothing)$ vanish if $q \neq n \cdot p$.
\end{theorem}
\begin{proof}
We consider the case of $E_n$; that of $Z_n$ is identical. Consider the spectral sequence of \eqref{eq:KanExtSS} for $T=\varnothing$. This has
\[E^1_{s,t,q} = \bigoplus\limits_{\mathclap{\substack{S_0, ..., S_s \\ \in \mathsf{d(s)Br}}}} H_{t-q+1}^\mr{Com}(E_n)_q(S_0) \otimes  \mathsf{d(s)Br}(S_0, S_1) \otimes \cdots \otimes \mathsf{d(s)Br}(S_{s-1}, S_s) \otimes \mathsf{FB}(r(S_s), \varnothing)\]
and converges strongly to $H_{s+t}(G^{E_n}(\varnothing))_q$. The terms with $S_0 = \varnothing$ form a subcomplex $D_{*,t,q}$ of $(E^1_{*,t,q}, d^1)$, as in this case for the latter tensor factors to be non-zero we must have $\varnothing = S_0 = S_1 = \cdots = S_s$, and $d^1$ preserves this property. Furthermore, this subcomplex is
\[H_{t-q+1}^\mr{Com}(E_n)_q(\varnothing) \overset{0}\leftarrow H_{t-q+1}^\mr{Com}(E_n)_q(\varnothing) \overset{\mr{id}}\leftarrow H_{t-q+1}^\mr{Com}(E_n)_q(\varnothing) \overset{0}\leftarrow H_{t-q+1}^\mr{Com}(E_n)_q(\varnothing) \overset{\mr{id}}\leftarrow\cdots\]
so is acyclic in degrees $s>0$. The quotient complex $E^1_{*,t,q}/D_{*,t,q}$ is given by the analogous formula to the above, but only summing over $S_0 \neq \varnothing$. In this case, for $H_{t-q+1}^\mr{Com}(E_n)_q(S_0)$  to be non-zero we must have $q = n \cdot (t-q+1)$. In total we find that $E^2_{0,t,q} = H_{t-q+1}^\mr{Com}(E_n)_q(\varnothing)$ and if $q \neq n \cdot (t-q+1)$ and $s > 0$ then $E^2_{s,t,q}=0$. 

We combine these properties as follows. By Lemma \ref{lem:KoszEasyEstimate} we already know that the claimed vanishing occurs for $q < n \cdot p$, so suppose that $q > n \cdot p$ and consider $H_{p}^\mr{Com}(E_n)_q(\varnothing) = E^2_{0,p+q-1,q}$. There are no differentials leaving this position in the spectral sequence. Differentials arriving at it come from $E^r_{r,p+q-r,q}$ with $r \geq 2$, but for this to be nontrivial we must have $q=n \cdot (p-r+1) \leq n \cdot p$, which is not possible by our assumption that $q > n \cdot p$. Thus no differentials can enter this position either, and so 
\[H_{p}^\mr{Com}(E_n)_q(\varnothing) \cong E^\infty_{0,p+q-1,q}.\]
But $E^\infty_{0,p+q-1,q}$ is a filtration quotient of $H_{p+q-1}(G^{E_n}(\varnothing))_q$, which by Corollary \ref{cor:BlackGphVanishingZero} vanishes for $n\cdot (p-1) < q$, so in particular for $q> n \cdot p$ which is what we assumed.
\end{proof}

\begin{remark}
The reader may wonder why we did not apply a transfer argument as in \cref{lem:inj-empty-from-1} directly to $RB^{Z_n}_{conn}$ and $RB^{E_n}_{conn}$ to deduce \cref{thm:koszulity-zn-en} from the corresponding result with $\ul{1}$ in place of $\varnothing$. The reason is that the analogue of the map $\pi$, which removes the unique leg from a red-and-black graph whose legs are labelled by $\ul{1}$, is \emph{not} a chain map in this case. We invite the reader to verify this for the following red-and-black graph:
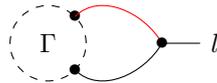
\[\begin{tikzpicture}
	\begin{scope}
		\draw [dashed] (0,0) circle (.5cm);
		\node at (0,0) {$\Gamma$};
		\node at (45:.5) {$\bullet$};
		\draw [red] (45:.5) to[out=45,in=120] (1.5,0);
		\node at (-45:.5) {$\bullet$};
		\draw (-45:.5) to[out=-45,in=-120] (1.5,0);
		\node at (1.5,0) {$\bullet$};
		\draw (1.5,0) -- (2,0);
		\node at (2.2,0) {$l$};
	\end{scope}
\end{tikzpicture}\]
\end{remark}

Recalling that the internal degree is $n$ times the weight, \cref{thm:KoszulPositive} and \cref{thm:koszulity-zn-en} together show that the $Z_n$ and $E_n$ are all Koszul, and hence by \cref{cor:E1Suffices} imply that $\GrLCS \ft_{g,1}$ is Koszul in gradings $\leq \tfrac{g}{3}$, and hence by \cref{prop:koszul} imply \cref{thm:mainPrime}.

\section{Applications of Koszulness} 

In this section we give two applications of \cref{athm:main}, to high-dimensional manifolds and graph complexes respectively.

\subsection{Applications to high-dimensional manifolds}\label{sec:HighDimApplications} In this subsection we follow \cref{sec:positive-arity} in refraining from mentioning the weight grading (as it coincides with the homological grading divided by $n$). Given the Koszulness of $Z_n$, we can now revisit \cref{sec:framed} and determine the rational homotopy groups of $X_1(g)$ in a stable range of degrees. Let $2n \geq 6$ and recall that by \cref{thm:pos-arity-framing-input} \eqref{enum:pos-arity-framing-input-cohomology} the construction of twisted MMM-classes yields an algebra homomorphism
\[\psi \colon K^\vee \otimes^{\cat{d(s)Br}} Z_n \lra H^*(X_1(g);\bQ)\]
which is an isomorphism of $\overline{G}{}^{\mr{fr},[\ell]}_g$-representations in a stable range; in fact, $* \leq \frac{g-3}{2}$ suffices by \cite[Section 9.2]{KR-WTorelli}. The discussion in \cite[Section 5]{KR-WTorelli} gives in degrees $* \leq n \cdot g$ (cf.~\cref{thm:TorRelation}) an expression for the left side as a quadratic commutative algebra:
\[\Lambda^*[W[n]]/(R') \lra K^\vee \otimes^{\cat{d(s)Br}} Z_n \qquad \text{with} \quad W \coloneqq \begin{cases}V_{1^3} & \text{if $n$ is odd,} \\
	V_3 & \text{if $n$ is even,}\end{cases}\]
and $R' \coloneqq \langle \eqref{eq:IH} \rangle \leq \Lambda^2[W[n]]$. We write $R'^\perp \subset \Lambda^2[W[n]]^\vee \cong \Lambda^2[W[n]]$ for its annihilator.

\begin{remark}\label{rem:homotopy-x1g-explicit}
	When $n$ is odd, as long as $g \geq 6$ we have the decomposition
	\[\Lambda^2[V_{1^3}] = V_0 + V_{1^2} + V_{1^4} + V_{1^6} + V_{2^2} + V_{2^2,1^2}.\]
	into irreducible $\mr{Sp}_{2g}(\bZ)$-representations. The vectors \eqref{eq:IH} generate the representation $V_0 + V_{1^2} + V_{2^2}$ (c.f.~\cref{rem:tg1-explicit}), so $R'^\perp \cong  V_{1^4}+V_{1^6} + V_{2^2,1^2}$. For $n$ even, one takes the transpose \cite[Remark 4.23]{KR-WDisks}.
\end{remark}

\begin{theorem}\label{thm:homotopy-x1g-koszul} 
In degrees $* \leq \frac{g-3}{2}$ there is a Lie algebra isomorphism
	\[\mr{Lie}(W[n{-}1])/(R'^\perp[-2]) \cong \pi_{*+1}(X_1(g)) \otimes \bQ.\]
\end{theorem}

\begin{proof}
	The map $\psi \colon K^\vee \otimes^{\cat{d(s)Br}} Z_n \to H^*(X_1(g);\bQ)$ of $\overline{G}{}^{\mr{fr},[\ell]}_g$-representations is an isomorphism in degrees $* \leq \frac{g-3}{2}$. The left side is Koszul in degrees $* \leq n\cdot \frac{g}{3}$ by \cref{thm:koszulity-zn-en} and \cref{lem:realisation-h-com}. Thus the unstable rational Adams spectral sequence of the proof of \cref{thm:KoszulPositive} (see \cref{sec:proof-KoszulPositive})
	\[E^2_{s,t} = H^\mr{Com}_{s}(H^*(X_1(g);\bQ))_t \Longrightarrow \mr{Hom}(\pi_{t-s}(\Omega X_1(g)),\bQ),\]
	vanishes except when $t = s \cdot n$ as long as $t \leq \min(\frac{g-3}{2},n\cdot \frac{g}{3})$ but since $2n \geq 6$ the first term is always smaller. Moreover, in this range the $E^2$-page is dual to the quadratic dual of $\Lambda^*[W[n]]/(R')$ and the spectral sequence collapses for degree reasons because the differentials have bidegree $(-r,-r+1)$.
\end{proof}

This in turn has consequences for diffeomorphisms of the disc $D^{2n}$ and homeomorphisms of the Euclidean space $\bR^{2n}$. The fundamental diagram of \cite{KR-WDisks} is (30) in Section 6.1 of loc.~cit.:
\begin{equation}\label{eqn:CrossDiagram}
	\begin{tikzcd}
		& \overline{B\mr{Diff}}^\mr{fr}_\partial(D^{2n})_{\ell_0} \dar \arrow[dashed]{rd} \\
		X_1(g)  \rar \arrow[dashed]{rd} & \overline{B\mr{Tor}}^\mr{fr}_\partial(W_{g,1})_\ell \rar \dar& X_0\\
		& \overline{B\mr{TorEmb}}^\mr{fr, \cong}_{\half\partial}(W_{g,1})_\ell.
	\end{tikzcd}
\end{equation}
The row and column are fibration sequences, and all spaces are nilpotent, connected, and of finite type. The common homotopy fibre $F_n$ of the rationalisations of the dashed maps \cite[Definition 6.1]{KR-WDisks} fits into a map of fibration sequences (cf.~\cite[(3)]{KR-WDisks})
\[\begin{tikzcd} F_n \rar \dar{\simeq} & (\overline{B\mr{Diff}}^\mr{fr}_\partial(D^{2n})_{\ell_0})_\bQ \dar{\simeq} \rar & (X_0)_\bQ \dar{\simeq} \\[-2pt]
	\left(\Omega^{2n+1}_0 \tfrac{\mr{Top}}{\mr{Top(2n)}}\right)_\bQ \rar & \left(\Omega^{2n}_0 \mr{Top}(2n)\right)_\bQ \rar & \left(\Omega^{2n}_0 \mr{Top}\right)_\bQ. \end{tikzcd}\]
In particular, $(X_0)_\bQ \simeq \prod K(\bQ,d)$ with product indexed by integers $d>0$ such that $d \equiv 2n-1 \pmod 4$. This directly yields the following improvement of Theorem C of \cite{KR-WDisks}, also improving Theorem B of loc.~cit.

\begin{corollary}The rational homotopy groups of $\Omega^{2n+1}_0 \tfrac{\mr{Top}}{\mr{Top}(2n)}$ are supported in degrees $* \in \bigcup_{r \geq 2} [2r(n-2)-1,2r(n-1)+1]$.\qed
\end{corollary}

\subsection{Applications to graph complexes}

\subsubsection{Relation to classical graph complexes}\label{sec:ClassicalGraphCxes}
Up to a change of degrees the chain complex $G^{Z_n}(S)$ may be identified with the sum $\bigoplus_{g} G^{(g,|S|)}$ of the graph complexes studied by Chan--Galatius--Payne \cite{CGP, CGP2}. When $S=\varnothing$ this is predual to Kontsevich's graph complex $\cat{GC}_2$ as studied by Willwacher \cite{WillwacherGRT}, and for general $S$ it is predual to the ``hairy graph complex" $\cat{HGC}_2^{S}$ of Andersson--Willwacher--Zivkovic \cite{AWZ}. Similarly, the chain complex $G^{E_n}(S)$ may be identified with the sum $\bigoplus_{g} \widetilde{C}_*(\Delta_{g,|S|};\bQ)$ of the reduced chains on the moduli spaces of tropical curves of genus $g$ with $|S|$ marked points from \cite{CGP2}, or the predual of $\cat{HGC}_2^{S, \text{mod}}$ from \cite{AWZ}.

In the degree convention of \cite{CGP2} a connected (weighted) graph $\Gamma$ is given degree $\mr{deg}_{CGP}(\Gamma) := \# E(\Gamma)-2g(\Gamma)$ where $g(\Gamma)$ is the sum of the first Betti number of $\Gamma$ and the weights at all its vertices. Our internal degree can be compared to $g(\Gamma)$ via
\[\deg_\text{int}(\Gamma) = n\sum_{v \in V(\Gamma)} (2w(v) + \mr{val}(v)-2) = n(2g(\Gamma) + |S|-2)\]
and our total degree is $\mr{deg}(\Gamma) = \#E(\Gamma) + \deg_\text{int}(\Gamma)$, so
\[\mr{deg}_{CGP}(\Gamma) = \mr{deg}(\Gamma) - (1+ \tfrac{1}{n})\deg_\text{int}(\Gamma) + |S|-2.\]

\subsubsection{An alternative approach}\label{sec:alternative-graph-proof}

Complementary to the filtration used in the proof of Proposition \ref{prop:RBIsModel} by number of edges, we may filter $RB^{E_n}_\text{conn}$ by the number of red edges: this kills the $d_\text{col}$ part of the differential, leaving
\[\mr{Gr}(RB^{E_n}_\mr{conn})(T) \cong \bigoplus_{n \geq 0} (G^{E_n}(\ul{n}), d_\text{con}) \otimes_{\fS_n} \cat{d(s)Br}(\ul{n}, T).\]
In particular it gives a spectral sequence
\[\begin{tikzcd} E^1_{p,q} = \bigoplus_{n \geq 0} H_{p+q}(G^{E_n}(\ul{n}))_q \otimes_{\fS_n} \cat{d(s)Br}(\ul{n}, T) \dar[Rightarrow] \\
H_{p+q}(RB^{E_n}_\mr{conn}(T))_q \rar{\cong} & H_{p+1}^\mr{Com}(E_n)_q(T).\end{tikzcd}\]

Using this and the spectral sequence \eqref{eq:KanExtSS}, 
it follows that Koszulness of $E_n$ is \emph{equivalent to} the vanishing range
\[H_{p+q}(G^{E_n}(T))_q=0 \text{ for } n \cdot (p+1) < q \text{ and all } T.\]
The same goes with $E_n$ replaced by $Z_n$. 

\begin{remark}\label{rem:FNW}
Willwacher has pointed out the following line of reasoning to us, which is essentially the argument of \cite{FNW}: combining the above with the discussion of Chan--Galatius--Payne's complexes in the last section, it follows that Koszulness of $Z_n$ is equivalent to $H_*(G^{(g, n)})=0$ for $* < n-3$, and Koszulness of $E_n$ is equivalent to $\widetilde{H}_*(\Delta_{g,n};\bQ)=0$ for $* < n-3+2g-1$, but both follow from \cite[Theorem 1.6]{CGP2}.
\end{remark}

\subsubsection{Comparing $E_n$ and $Z_n$} We can use our results to give a new proof of \cite[Theorem 1.3]{CGP} and \cite[Theorem 1.4]{CGP2} (another has been sketched in \cite[Lemma 5]{AWZ}). As explained in Section \ref{sec:MainEx1} there is a map of commutative algebra objects $E_n \to Z_n$, and by \cref{thm:relative-com-zn-vs-en} we have
\begin{equation}\label{eq:relative-com-zn-vs-en-recollected}
H_p^\mr{Com}(Z_n, E_n)_q(S) = \begin{cases}
	\bQ & \text{ if } |S|=1 \text{ and } (p,q)=(2,n)\\
	0 & \text{ otherwise}
\end{cases}.
\end{equation}

\begin{corollary}
We have
\[H_{p+q}(G^{Z_n}(S), G^{E_n}(S))_q = \begin{cases}
\bQ & \text{ if } |S|=1 \text{ and } (p,q)=(1,n)\\
0 & \text{ otherwise}
\end{cases}.\]
\end{corollary}

\begin{proof}
There is a relative form of the spectral sequence \eqref{eq:KanExtSS}, obtained by taking mapping cones of filtered chain complexes, starting with $E^1_{s,t,q}$ given by
\[\bigoplus\limits_{\mathclap{\substack{S_0, \ldots, S_s \\ \in \mathsf{d(s)Br}}}} H_{t-q+1}^\mr{Com}(Z_n,E_n)_q(S_0) \otimes  \mathsf{d(s)Br}(S_0, S_1) \otimes \cdots \otimes \mathsf{d(s)Br}(S_{s-1}, S_s) \otimes \mathsf{FB}(r(S_s), T)\]
	and converging to $H_{s+t}(G^{Z_n}(T), G^{E_n}(T))_q$. 	If $|T| \geq 2$ then only terms with $|S_0| \geq 2$ can contribute, but then $H_{*}^\mr{Com}(Z_n, E_n)(A_0)=0$ by \eqref{eq:relative-com-zn-vs-en-recollected}.
	
	If $0 \leq |T|<2$ then we may argue as in the proof of \cref{thm:koszulity-zn-en} that the subcomplex with $|T| = |S_0| = \cdots = |S_s|$ is acyclic in degrees $s>0$. Once we quotient by this complex, the remaining terms must have $|S_0| \geq |T|+2$ and the above argument still applies. Thus the $E^2$-page of the spectral sequence is supported along $s=0$ so collapses to give an isomorphism
	\[H_{p+1}^\mr{Com}(Z_n, E_n)_q(T) \cong H_{p+q}(G^{Z_n}(T), G^{E_n}(T))_q.\]
	The result then follows from \eqref{eq:relative-com-zn-vs-en-recollected}.
\end{proof}

\section{Injectivity of the geometric Johnson homomorphism}

\subsection{The geometric Johnson homomorphism}

\subsubsection{Construction} Recall from \cref{lem:grlcs-tg-ug-ses} that there is an extension of Lie algebras with additional grading
\[\GrLCS \fp^r_{g,n} \lra \GrLCS \ft^r_{g,n} \lra \GrLCS \ft_g.\]
This exhibits $\GrLCS \fp^r_{g,n}$ as a Lie ideal in $\GrLCS \ft^r_{g,n}$ and hence the adjoint representation induces a map
\[\GrLCS \ft^r_{g,n} \lra \mr{Der}(\GrLCS \fp^r_{g,n}).\]
Specializing to $\Sigma_{g,1}$, we have $\GrLCS \fp_{g,1} = \mr{Lie}(H)$, the free Lie algebra on $H \coloneqq H_1(\Sigma_{g,1};\bQ)$ in weight 1. The map $\GrLCS \ft_{g,1} \to \mr{Der}(\mr{Lie}(H))$ has image in the derivations which increase the weight and annihilate the element $\omega = \sum_{i=1}^g [e_i,f_i]$ with $e_1,\ldots,e_g,f_1,\ldots,f_g$ a symplectic basis of $H_1(\Sigma_{g,1};\bQ) = H_1(\Sigma_g;\bQ)$ \cite[Corollary 3.2]{MoritaAbelian}. This is denoted by $\fh_{g,1} \subset \mr{Der}(\mr{Lie}(H))$ and called the Lie algebra of \emph{positive degree symplectic derivations}.

Following Hain we refer to the resulting map
\[\tau_{g,1} \colon \GrLCS \ft_{g,1} \lra \fh_{g,1}\]
as the \emph{geometric Johnson homomorphism}. Since this is a map of Lie algebras and its domain is generated in weight 1 as long as $g \geq 4$ (see \cref{cor:hain-presentation}), it is determined by its restriction to $\mr{Gr}^1_\mr{LCS}\,\ft_{g,1}$. Furthermore, it is a map of algebraic $\mr{Sp}_{2g}(\bZ)$-representations, and $\mr{Gr}^1_\mr{LCS}\,\ft_{g,1}$ and $(\fh_{g,1})_1$ are both given by $V_1 \oplus V_{1^3}$ by \cref{cor:hain-presentation} and \cite[Computation 5.8]{KR-WDisks}. This map $\mr{Gr}^1_\mr{LCS}\,\ft_{g,1} \to (\fh_{g,1})_1$ is an isomorphism, essentially by construction of $\ft_{g,1}$, and thus up to isomorphism  $\tau_{g,1}$ is the unique map of Lie algebras that is injective in weight 1. \cref{athm:morita} asserts that in weight $\leq \frac{1}{3}g$ the kernel of $\tau_{g,1}$ lies in the centre of $\GrLCS \ft_{g,1}$ and consists of trivial representations of $\mr{Sp}_{2g}(\bZ)$.

\begin{remark}\cref{athm:morita} may be rephrased as saying that in weight $\leq \frac{1}{3}g$ the kernel of $\tau_{g,1}$ lies in the centre of $\GrLCS \fg_{g,1}$, the lower central series Lie algebra of the pro-algebraic group $\cG_{g,1}$ of \cref{sec:torelli-defs}. See \cite[Question 5.11]{HainJohnson} for a question about this centre.\end{remark}

We know that $K^\vee \otimes^{\cat{dsBr}} E_1$ and $K^\vee \otimes^{\cat{dsBr}} Z_1$ are Koszul in weight $\leq \frac{1}{3}g$ by \cref{thm:koszulity-zn-en} and the argument of \cref{cor:E1Suffices}. In particular, they are quadratic in this range and as long as $g \geq 6$ we let $(K^\vee \otimes^{\cat{dsBr}} E_1)^\mr{quad}$ and $(K^\vee \otimes^{\cat{dsBr}} Z_1)^\mr{quad}$ denote the quadratic algebras given by the quadratic duals of the presentations of $K^\vee \otimes^{\cat{dsBr}} E_1$ and $K^\vee \otimes^{\cat{dsBr}} Z_1$ in weights $\leq 2$. These are explicitly described by (see the proofs of \cref{thm:TorRelation} and \cref{thm:homotopy-x1g-koszul})
\[(K^\vee \otimes^{\cat{dsBr}} E_1)^\mr{quad} \cong \frac{\mr{Lie}(\Lambda^3(V_1)[0,1])}{((\ref{eq:IH})^\perp)} \qquad \text{and} \qquad (K^\vee \otimes^{\cat{dsBr}} Z_1)^\mr{quad} \cong \frac{\mr{Lie}(V_{1^3}[0,1])}{((\ref{eq:IH})^\perp)};\]
in both cases the generators are in homological degree 0 and weight 1, so in particular these Lie algebras are supported in homological degree 0.  For $g \geq 6$ there is therefore a commutative diagram
\begin{equation*}
\begin{tikzcd}
((K^\vee \otimes^{\cat{dsBr}} Z_1)^\mr{quad})_{0,w} \rar \dar &  \big[H^\mr{Com}_w(K^\vee \otimes^{\cat{dsBr}} Z_1)_{w,w} \big]^\vee \dar \\[-2pt]
((K^\vee \otimes^{\cat{dsBr}} E_1)^\mr{quad})_{0,w}  \rar & \big[H^\mr{Com}_w(K^\vee \otimes^{\cat{dsBr}} E_1)_{w,w} \big]^\vee
\end{tikzcd}
\end{equation*}
with horizontal maps isomorphisms for $w \leq \frac{1}{3}g$.

\begin{lemma}\label{lem:en-vs-zn-duals} Suppose $g \geq 6$, then the map
	\[((K^\vee \otimes^{\cat{dsBr}} Z_1)^\mr{quad})_{0,w}  \lra ((K^\vee \otimes^{\cat{dsBr}} E_1)^\mr{quad})_{0,w} \]
	is injective for $w=1$ and an isomorphism for $2 \leq w \leq \frac{1}{3}g$ 
\end{lemma}

\begin{proof} Since the weight $w$ parts of $Z_1$ and $E_1$ are supported on sets of size $\leq 3w$, by \cref{lem:realisation-h-com} the right vertical map is dual to
	\[K^\vee \otimes^{\cat{dsBr}} H^\mr{Com}_w(E_1)_{w,w} \lra K^\vee \otimes^{\cat{dsBr}} H^\mr{Com}_w(Z_1)_{w,w}\]
as long as $w \leq \tfrac{1}{3}g$.	By \cref{thm:relative-com-zn-vs-en}, the map $H^\mr{Com}_p(E_1)_{q,w} \to H^\mr{Com}_p(Z_1)_{q,w}$ is an isomorphism except when $(p,q,w) = (1,1,1)$ and evaluated on sets of size $1$, in which case it is surjective by the proof of \cref{cor:EnZnEquiKoszul}. The lemma is then obtained by dualising.
\end{proof}

By the proof of \cref{prop:E1KoszImpliesStablyKosz}, the map
\[(K^\vee \otimes^{\cat{dsBr}} E_1)^\mr{quad}  \lra (K^\vee \otimes^{\cat{dsBr}} E_1/(\kappa_{e^2}))^\mr{quad}  \cong \GrLCS \ft_{g,1}\]
is injective with cokernel a single trivial representation in weight 2. Since the geometric Johnson homomorphism is injective in weight $1$ and we work up to trivial representations, we thus may as well study the composition
\begin{equation}\label{eqn:framed-johnson} (K^\vee \otimes^{\cat{dsBr}} Z_1)^\mr{quad} \lra \GrLCS \ft_{g,1} \lra \fh_{g,1}.\end{equation}
Up to isomorphism, this is the unique map of Lie algebras with a additional weight grading which is injective in weight 1.

\subsubsection{A high-dimensional geometric Johnson homomorphism} We next describe a higher-dimensional incarnation of \eqref{eqn:framed-johnson}, and first give its domain and codomain. By \cref{thm:homotopy-x1g-koszul}, for $n$ odd and degrees $w(n-1) \leq \frac{g-3}{2}$ \cite[Section 9.2]{KR-WTorelli}, there is an isomorphism
\[((K^\vee \otimes^{\cat{dsBr}} Z_n)^\mr{quad})_{w(n-1), w} \overset{\cong}\lra \pi_{w(n-1)+1}(X_1(g)) \otimes \bQ \]
of 
algebraic $\overline{G}{}^{\mr{fr},[\ell]}_g$-representations, with $\overline{G}{}^{\mr{fr},[\ell]}_g \subset \mr{Sp}_{2g}(\bZ)$ a finite index subgroup. The left-hand side is identified with $((K^\vee \otimes^{\cat{dsBr}} Z_1)^\mr{quad})_{0,w}$, 
so the domain of our higher-dimensional geometric Johnson homomorphism is the rational homotopy Lie algebra of $X_1(g)$ with a certain regrading. 

For the codomain, we let $\mr{Der}^+(\mr{Lie}(H[n-1]))$ denote the graded Lie algebra of positive degree derivations of the free graded Lie algebra on $H = H_n(W_{g,1};\bQ)$ placed in degree $n-1$ and weight 1. Writing $\mr{Der}^+_\omega(\mr{Lie}(H[n-1]) \subset \mr{Der}^+(\mr{Lie}(H[n-1])$ for those derivations which annihilate $\omega$, we get a graded Lie algebra which agrees with $\fh_{g,1}$ up to regrading---$(\fh_{g,1})_{0,w}$ is identified with $\mr{Der}^+_\omega(\mr{Lie}(H[n-1])_{w(n-1),w}$. It is well-known that there is an isomorphism of graded Lie algebras
\[\pi_{*+1}(B\hAut^\mr{id}_\partial(W_{g,1})) \otimes \bQ \cong \mr{Der}^+_\omega(\mr{Lie}(H[n-1])),\]
 in algebraic  $\overline{G}{}^{\mr{fr},[\ell]}_g$-representations (see \cite[Corollary 3.3]{berglundmadsen2}), with the superscript $\mr{id}$ indicating the identity component. So the codomain of our higher-dimensional geometric Johnson homomorphism is the rational homotopy Lie algebra of $B\hAut^\mr{id}_\partial(W_{g,1})$.

There are maps
\begin{equation}\label{eqn:higher-dim-johnson} X_1(g) \lra B\Diff^\mr{fr}_\partial(W_{g,1})_\ell \lra B\hAut_\partial(W_{g,1}),\end{equation}
which induce a map $\pi_{*+1}(X_1(g)) \otimes \bQ \to \pi_{*+1}(B\hAut^\mr{id}_\partial(W_{g,1})) \otimes \bQ$ of graded Lie algebras in algebraic $\overline{G}{}^{\mr{fr},[\ell]}_g$-representations. This is the \emph{higher-dimensional geometric Johnson homomorphism}. To see that it coincides with $\tau_{g,1}$, in a stable range and up to isomorphism and regrading, it suffices to verify that it is injective in degree $*=n-1$ in a stable range. This is the content of \cite[Proposition 5.10]{KR-WDisks}. To prove \cref{athm:morita}, we may thus use this higher-dimensional incarnation instead.

\subsection{Proof of \cref{athm:morita}} Translated to high dimensions, \cref{athm:morita} says that in a stable range the kernel of the map induced by \eqref{eqn:higher-dim-johnson} on rational homotopy groups consists solely of trivial $\smash{\overline{G}{}^{\mr{fr},[\ell]}_g}$-representations in the centre of $\pi_{*+1}(X_1(g)) \otimes \bQ$.
	
	We need to recall some of the setup from \cite[Section 3]{KR-WDisks}. Let us write $\half \partial W_{g,1}$ for a fixed subset $D^{2n-1} \subset S^{2n-1} = \partial W_{g,1}$. Then we let $B\mr{Emb}^\mr{fr}_{\half \partial}(W_{g,1})$ denote the path component of the homotopy quotient 
	\[\mr{Bun}_{\half \partial}(TW_{g,1},\theta_\mr{fr}^* \gamma) \sslash \mr{Emb}_{\half \partial}(W_{g,1})^\times\]
	containing $\ell$. Let us consider the commutative diagram
	\[\begin{tikzcd} X_1(g) \rar \arrow[dashed]{rd}[swap]{\eqref{eqn:higher-dim-johnson}}&[3pt] B\Diff^\mr{fr}_\partial(W_{g,1})_\ell \dar \rar &[3pt] B\Emb^\mr{fr}_{\half \partial}(W_{g,1})_\ell \dar{\CircNum{2}} \\[-2pt]
		& B\hAut_\partial(W_{g,1}) \rar{\CircNum{1}} & B\hAut_{\half \partial}(W_{g,1}).\end{tikzcd}\]
	
	The map $\CircNum{1}$ is injective on rational homotopy groups in all degrees, as the following commutes
	\[\begin{tikzcd} \pi_{*+1}(B\hAut^\mr{id}_\partial(W_{g,1})) \otimes \bQ \rar \dar{\cong} & \pi_{*+1}(B\hAut^\mr{id}_\ast(W_{g,1})) \otimes \bQ \dar{\cong} \\[-2pt]
		\mr{Der}^+_\omega(\mr{Lie}(H[n-1])) \rar[hook] & \mr{Der}^+(\mr{Lie}(H[n-1])).\end{tikzcd}\]
	
	We claim that $\CircNum{2}$ is also injective. To see this is the case, we use the Bousfield--Kan spectral sequence for the embedding calculus Taylor tower \cite[Section 5.2.3]{KR-WDisks}. This is an extended spectral sequence in the sense of \cite{BousfieldKan}, given by
	\[{}^{BK}E^1_{p,q} = \begin{cases} \pi_{q-p}(B\hAut_{\half \partial}(W_{g,1})) & \text{if $p=0$,} \\
		\pi_{q-p}(BL_{p+1} \mr{Emb}_{\half \partial}(W_{g,1})^\times_\mr{id}) & \text{if $p \geq 1$,}\end{cases}\]
	and converging completely to $\pi_{q-p}(B\mr{Emb}^\mr{fr}_{\half \partial}(W_{g,1})^\times)$. In \cite[Proposition 5.32]{KR-WDisks} we described which entries on the $E^1$-page can be non-zero. In particular, there is a pattern of ``bands'': for $r \geq 1$ there are non-zero entries only in bidegrees $(p,q)$ lying in the intervals $[0,r+1] \times \{r(n-1)+1\}$ with $r \geq 1$. Since the differentials have bidegree $(r,r-1)$, the spectral sequence collapses rationally at the $E^2$-page in a range. As we increase $n$, the number of intervals in which this collapse occurs increases. Thus by making $n$ large enough, we see that the map $B\Emb^\mr{fr}_\partial(W_{g,1})_\ell \to B\hAut_{\half \partial}(W_{g,1})^\times$ induces an identification
	\[\begin{tikzcd}\pi_{r(n-1)+1}(B\mr{Emb}^\mr{fr}_{\half \partial}(W_{g,1})^\times) \otimes \bQ \dar{\cong} \\[-5pt]
		\ker\Big[\pi_{r(n-1)+1}(B\mr{hAut}_{\half \partial}(W_{g,1})^\times) \otimes \bQ \overset{d^1}\lra \pi_{r(n-1)}(BL_2 \mr{Emb}^\mr{fr}_{\half \partial}(W_{g,1})^\times_\mr{id}) \otimes \bQ\Big],\end{tikzcd}\]
	so in particular is injective.
	
	Thus the kernel of \eqref{eqn:higher-dim-johnson} coincides with that of the map induced on rational homotopy by the composition
	\[X_1(g) \lra B\Diff_\partial^\mr{fr}(W_{g,1})_\ell \lra B\Emb^\mr{fr}_{\half \partial}(W_{g,1})_\ell.\]
	The kernel of the left map coincides with the image of the rational homotopy of $X_0$ under the connecting homomorphism and hence consists of trivial representations in the centre of $\pi_{*+1}(X_1(g)) \otimes \bQ$. Using the framed Weiss fibre sequence \cite[Theorem 3.12]{KR-WDisks}
	\[B\mr{Diff}^\mr{fr}_\partial(W_{g,1}) \lra B\mr{Emb}^\mr{fr}_{\half \partial}(W_{g,1}) \lra B^2\mr{Diff}^\mr{fr}_\partial(D^{2n}),\]
	the kernel of the right map coincides with the image of the rational homotopy of  $B^2\mr{Diff}^\mr{fr}_\partial(D^{2n})_{\ell_0}$ under the connecting homomorphism and hence consists of trivial representations in the centre of $\pi_{*+1}(B\Diff_\partial^\mr{fr}(W_{g,1})_\ell) \otimes \bQ$. 
	
	 To complete the proof of \cref{athm:morita} it remains to establish the explicit ranges. To do so, we use a version of representation stability (cf.~\cite[Proposition 15.1]{HainJohnson}) for the domain and codomain of $\tau_{g,1}$. By \cref{cor:hain-presentation} the Lie algebra $\GrLCS \ft_{g,1}$ is quadratically generated by $\Lambda^3 V_1$ in weight 1, so by stability for symplectic Schur functors and the branching rule (e.g.~\cite{KoikeTerada}), its decomposition into non-zero irreducibles stabilises in weight $\leq \frac{1}{3}g$. The Lie algebra $\fh_{g,1}$ fits into the short exact sequence
	\[0 \lra \fh_{g,1} \lra H[-1] \otimes \mr{Lie}(H[1]) \lra \mr{Lie}^{\geq 2}(H[1])[-2] \lra 0,\]
	where $\mr{Lie}^{\geq 2}$ means we discard bracketings of $<2$ elements. Once more invoking stability for symplectic Schur functors and the branching rule, its decomposition into non-zero irreducibles stabilises in weight $\leq g-2$. These observations imply that for $g'>g$, in the commutative diagram
	\[\begin{tikzcd} \GrLCS \ft_{g,1} \rar \dar & \fh_{g,1} \dar \\[-2pt]
	\GrLCS\ft_{g',1} \rar & \fh_{g',1}\end{tikzcd}\]
	induced by the inclusion $\Sigma_{g,1} \subset \Sigma_{g',1}$, the left vertical map is an isomorphism in weight $\leq \frac{1}{3}g$ and the right vertical map is in weight $ \leq g-2$. Since we can make the bottom horizontal map have kernel consisting of trivial representations that are contained in the centre for arbitrary large weight by increasing $g'$, we conclude that the same is true for the top horizontal map in weight $\leq \min(\frac{1}{3}g,g-2)$. Since $g-2 \geq \frac{1}{3}g$ unless $g \leq 2$---in which case the injectivity of the geometric Johnson homomorphism in weight $\leq \tfrac{g}{3}$ is clear---we conclude that the top map has a kernel consisting of trivial representations that are contained in the centre for weight $\leq \frac{1}{3}g$.
	
\begin{remark}The proof of \cref{athm:morita} also gives a version of this theorem for $n$ even.\end{remark}

\subsection{Further remarks}

\subsubsection{The Euler tangential structure} To prove \cref{athm:morita} we could have used the Euler tangential structure of \cref{sec:euler} instead of framings. Doing so, the space $X_1(g)$ gets replaced by $X_1^e(g)$ and $K^\vee \otimes^{\cat{d(s)Br}} Z_n$ by $(K^\vee \otimes^{\cat{d(s)Br}} E_n)/(\kappa_{e^j} \mid j>1)$, but otherwise the proof proceeds in a similar manner. From this point of view, Lemma \ref{lem:en-vs-zn-duals} amounts to the statement that map induced on rational homotopy groups by
\[\mr{Bun}_\partial(TW_{g,1},\theta_\mr{fr}^* \gamma_{2n}) \lra \mr{Bun}_\partial(TW_{g,1},\theta_e^* \gamma_{2n})\]
is an isomorphism in degrees $*>n$.

\subsubsection{Computing the Johnson cokernel}

	Both $\GrLCS \ft_{g,1}$ and $\fh_{g,1}$ are accessible to computer calculations: more precisely, one can compute the multiplicities of the irreducibles in each degree. As pointed out in the introduction, for $\GrLCS \ft_{g,1}$ one uses the computation of the stable character of $\GrLCS \ft_g$ of Garoufalidis--Getzler \cite{GG}, which was conditional on \cref{athm:main} (or rather \cref{thm:mainPrime}), while $\fh_{g,1}$ admits a description in terms of Schur functors (see \cite[Section 5.4.1]{KR-WDisks}). Then \cref{athm:morita} implies that, up to trivial representations, the Johnson cokernel $\mr{coker}(\tau_{g,1} \colon \GrLCS \ft_{g,1} \to \fh_{g,1})$ is also accessible to computer calculations in a stable range.
	
	In fact, up to trivial representations the Johnson image is the kernel of the restriction to $\fh_{g,1}$ of the connecting homomorphism
	\[\partial \colon \pi_{*+1}(B\hAut_*(W_{g,1})) \otimes \bQ \lra \pi_{*}(BL_2 \Emb^\mr{fr}_{\half \partial}(W_{g,1})) \otimes \bQ\]
	in the fibration sequence $BL_2 \Emb^\mr{fr}_{\half \partial}(W_{g,1}) \to BT_2 \Emb^\mr{fr}_{\half \partial}(W_{g,1}) \to B\hAut_\ast(W_{g,1})$ arising from the framed embedding calculus Taylor tower. It would be interesting to relate $\partial$ to Morita's trace maps and its variants, and to approach Question 12.8 of \cite{HainJohnson} from this perspective (cf.~part (ii) of ``the most optimistic landscape'' in Section 1.9 of loc.~cit.).

\subsubsection{Variants}
Let us next consider the analogous geometric Johnson homomorphisms
\[\tau_g^1 \colon \GrLCS \ft_g^1 \lra \fh_{g}^1 \qquad \text{and} \qquad \tau_g \colon \GrLCS \ft_g \lra \fh_g,\]
where $\smash{\fh_{g}^1}$ (also denoted $\fh_{g,\ast}$ in the literature) is the Lie algebra of the positive degree symplectic derivations of $\GrLCS \fp_g^1$ and $\fh_g$ is the quotient of $\fg_{g}^1$ by the inner derivations \cite[p.~370]{MoritaStructure}. 

\begin{corollary}\label{cor:morita-decorations} In weight $\leq \frac{1}{3}g$ the kernels of $\tau_g^1 \colon \GrLCS \ft_g^1 \to \fh_{g}^1$ and $\tau_g \colon \GrLCS \ft_g \to \fh_g$ lie in the centre and consists of trivial $\mr{Sp}_{2g}(\bZ)$-representations.
\end{corollary}

\begin{proof}For $\tau_g^1$ we use the extension  $\bQ[2\to \GrLCS \ft_{g,1} \to \GrLCS \ft_g^1$ of \cref{lem:grlcs-tg-ug-ses}.
	The statement in weight $\neq 2$ follows because the images of $\tau_{g,1}$ and $\tau^1_g$ coincide by \cite[Proposition 3.5 (ii)]{MSS}. In weight 2, $\fh_{g,1}$ and $\fh_{g}^1$ differ only in a trivial representation, by \cite[Table 1]{MSS} ($\mathfrak{j}_{g,1}$ is the kernel of the surjection $\fh_{g,1} \to \fh_g^1$ \cite[p.~307]{MSS}).
	
	For $\tau_g$ we use the diagram of extensions
	\[\begin{tikzcd}  \GrLCS \fp^1_g \rar \dar[equal] & \GrLCS \ft^1_g \dar{\tau_g^1} \rar & \GrLCS \ft_g  \dar{\tau_g}  \\[-2pt]
		 \GrLCS \fp^1_g  \rar & \fh_g^1 \rar & \fh_g. \end{tikzcd}\]
	The top row appears in \cref{lem:grlcs-tg-ug-ses}, the bottom row appears on p.~370 of \cite{MoritaStructure} (Morita uses $\cL_g$ and $\fh_{g,*}$ in place of $\GrLCS \fp^1_g$ and $\fh_g^1$). See also p.~643 of \cite{HainTorelli}.
\end{proof}

\begin{remark}\cref{athm:morita} and \cref{cor:morita-decorations} remain true when we replace the unipotent completions $\ft_{g,1}$, $\ft_g^1$, and $\ft_g$ with the relative unipotent completions $\fu_{g,1}$, $\fu_g^1$, and $\fu_g$ respectively, since these only differ by trivial representations in the centre by \cref{lem:tg-ug-ses}.\end{remark}

\begin{remark}\label{rem:habiro-massuyeau} Let $\cC_{g,1}$ be the monoid of 3-dimensional homology cylinders under concatenation, which has a filtration by $Y$-equivalence. There is a map of filtered monoids $T_{g,1} \to C_{g,1}$, if the left side is given the lower central series filtration. Taking associated gradeds and rationalising, we get a map of Lie algebras with additional weight grading
	\[\mr{Gr}^\bullet_\mr{LCS}\, T_{g,1} \lra \mr{Gr}^\bullet_Y\, \cC_{g,1}.\]
	In \cite[Question 1.3]{HabiroMassuyeau}, Habiro and Massuyeau ask whether this map is injective. Rationally, in a stable range, and up to trivial representations in the centre, it follows from \cref{athm:morita} that it is. To see this, consider Figure 8.1 of \cite{HabiroMassuyeau} and observe that the map denoted by $\mr{Gr}^\Gamma \cI_{g,1} \otimes \bQ \to \mr{Gr}^{[-]} \cI_{g,1} \otimes \bQ$ is $\tau_{g,1}$.
\end{remark}

\bibliographystyle{amsalpha}
\bibliography{../../cell}

\vspace{.5cm}

\end{document}